\documentclass{article}
\usepackage[english]{babel}
\usepackage[dvipsnames]{xcolor}
\definecolor{britishracinggreen}{rgb}{0.0, 0.26, 0.15}
\definecolor{cobalt}{rgb}{0.0, 0.28, 0.67}
\usepackage[utf8]{inputenc}
\usepackage{braket,caption,comment,mathtools,stmaryrd}
\usepackage{multirow,booktabs,microtype}
\usepackage{calrsfs}
 \usepackage{amsmath}
\usepackage{graphicx}
\usepackage{lipsum}
\usepackage{amsthm}
\usepackage{latexsym}
\usepackage{todonotes}
\usepackage{fancyhdr}
\usepackage{young}
\usepackage{nomencl}
\usepackage{soul} % per testo barrato
\usepackage[colorlinks,bookmarks]{hyperref} %
\hypersetup{colorlinks,%
            citecolor=britishracinggreen,%
            filecolor=black,%
            linkcolor=cobalt,%
            urlcolor=black}
\numberwithin{equation}{section}
\usepackage[utf8]{inputenc}
\usepackage[english]{babel}
 \usepackage{amssymb}
\usepackage{pst-node}
\usepackage{authblk}
\usepackage{tikz-cd}
\usepackage{tikz}
\usepackage{amsthm}
\usepackage[utf8]{inputenc}
\usepackage[all,cmtip]{xy}
\usepackage{multirow}
\usepackage{float}
\usepackage{colortbl}
\usepackage{calrsfs}
 \usepackage{amsmath}
\usepackage{graphicx}
\usepackage{lipsum}
\usepackage{amsthm}
\usepackage[utf8]{inputenc}
\usepackage[english]{babel}
 \usepackage{amssymb}
\usepackage{pst-node}
\usepackage{authblk}
\usepackage{tikz-cd}
\usepackage{tikz}
\usepackage{amsthm}
\usepackage[utf8]{inputenc}
\usepackage[all,cmtip]{xy}
\usepackage{multirow}

\newcommand{\Ja}{\mathcal{J}}
\theoremstyle{definition}
\newtheorem{definition}{Definition}[section]
\newtheorem{theorem}{Theorem}[section]
\newtheorem{corollary}{Corollary}[theorem]
\newtheorem{lemma}[theorem]{Lemma}
\newtheorem{proposition}[theorem]{Proposition}
\newtheorem{remark}{Remark}[section]
\DeclareMathOperator*{\res}{res}
\makeindex
\begin{document}
 
\title{The Differential Geometry of the Orbit space of Extended Affine Jacobi Group $A_n$}
\author{Guilherme F. Almeida}
\affil{SISSA, via Bonomea 265, Trieste, Italy}
\maketitle

\begin{center}
Dedicated to the memory of Professor Boris Dubrovin.
\end{center}

\begin{abstract}
 We define certain extensions of Jacobi groups of $A_n$, prove an analogue of Chevalley Theorem for their invariants, and construct a Dubrovin Frobenius structure on it orbit space.
\end{abstract} 

\tableofcontents

\section{Introduction}
 Dubrovin Frobenius manifold is a geometric interpretation of a remarkable system of differential equations, called WDVV equations \cite{B. Dubrovin2}. Since the beginning of the nineties, there has been a continuous exchange of ideas from fields that are not trivially related to each other, such as: Topological quantum field theory, non-linear waves, singularity theory, random matrices theory, integrable systems, and Painleve equations. Dubrovin Frobenius manifolds theory is a bridge between them.\\

\subsection{Topological quantum field theory}
The connections made by Dubrovin Frobenius manifolds theory works because all the mentioned theories are related with some WDVV equation. In \cite{B. Dubrovin2}, Dubrovin showed that many constructions of Topological field theories (TFT) can be deduced from the geometry of Dubrovin-Frobenius manifolds. For instance, one of the main objects to be computed in TFT are correlation functions, which are mean values of physical quantities. Since in TFT it is possible to have infinite many correlation functions, an efficient way to compute all of them is encoding all the correlators in a single function, called partition function. In \cite{M. Kontsevich}, Konsevitch proved that a partition function of a specific TFT can be obtained from the solution of KdV hierarchy, which is an example of integrable hierarchy, i.e. an infinite list of integrable partial differential equations. This discovery opened a new field of research in mathematical physics, because for this case, it was found an effective way to compute exactly all the correlation functions due to its integrable system nature. In \cite{B. Dubrovin and Y. Zhang2}, Dubrovin and Zhang constructed a method to derive an integrable hierarchy from the data of Dubrovin Frobenius manifold, furthermore, in many important examples, they were able to relate these integrable hierarchy with partition functions of some TFT.\\

\subsection{Orbit space of reflection groups and its extensions}
In \cite{B. Dubrovin2}, Dubrovin point out that, WDVV solutions with certain good analytic properties are related with partition functions of TFT. Afterwards, Dubrovin conjectured that WDVV solutions with certain good analytic properties are in one to one correspondence with discrete groups.  This conjecture is supported in ideas which come from singularity theory, because in this setting there exist a integrable systems/ discrete group correspondence. Furthermore, in minimal models as Gepner chiral rings there exist a correspondence between physical models and discrete groups.
In \cite{C. Hertling}, Hertling proved that a particular class of Dubrovin-Frobenius manifold, called polynomial Dubrovin-Frobenius manifold is isomorphic to orbit space of a finite Coxeter group, which are spaces such that its geometric structure is invariant under the finite Coxeter group.
In  \cite{Bertola M.1}, \cite{Bertola M.2},\cite{B. Dubrovin1}, \cite{B. Dubrovin2}, \cite{B. Dubrovin I. A. B. Strachan Y. Zhang D. Zuo}, \cite{B. Dubrovin and Y. Zhang},   \cite{D. Zuo},  there are  many examples of WDVV solutions  that are associated with orbit spaces of natural extensions of finite Coxeter groups such as
extended affine Weyl groups, and Jacobi groups. Therefore, the construction of Dubrovin Frobenius manifolds on orbit space of reflection groups and its extensions is a prospective project of the classification of WDVV solutions. In addition, WDDV solutions arising from orbit spaces may also have some applications in TFT or some combinatorial problem, because  previously these relation was demonstrated in some examples such as the orbit space of the finite Coxeter group $A_1$, and the extended affine Weyl group $A_1$   \cite{Dubrovin 3}, \cite{B. Dubrovin and Y. Zhang2}. \\

\subsection{Hurwtiz space/ Orbit space correspondence}
There are several others non-trivial connections that Dubrovin Frobenius manifolds theory can make. For example, Hurwitz spaces is the one of the main source of examples of Dubrovin Frobenius manifolds. Hurwitz spaces $H_{g,n_0,n_1,..,n_m}$ are moduli space of covering over $\mathbb{CP}^1$ with a fixed ramification profile.  More specifically, $H_{g,n_0,n_1,..,n_m}:$ is moduli space of pairs
\begin{equation*}
\{ C_g ,\lambda: C_g\mapsto   \mathbb{CP}^1    \}
\end{equation*}
where $C_g$ is a compact Riemann surface of genus $g$, $\lambda$ is meromorphic function with poles in 
\begin{equation*}
\lambda^{-1}(\infty)=\{\infty_0,\infty_1,..,\infty_{m} \}.
\end{equation*}
 Moreover, $\lambda$ has degree $n_i+1$ near $\infty_i$. Hurwitz space with a choice of specific Abelian differential, called quasi-momentum or primary differential, give rise to a Dubrovin Frobenius manifold, see  section \ref{Hurwitz space chapter} for details. In some examples, the Dubrovin Frobenius structure of Hurwitz spaces are isomorphic to Dubrovin Frobenius manifolds associated with orbit spaces of suitable groups. For instance, the orbit space of the finite Coxeter group $A_n$  is isomorphic to the Hurwitz space $H_{0,n}$, furthermore, orbit space of the extended affine Weyl group $\tilde A_n$ and of the Jacobi group $\Ja(A_n)$ are isomorphic to the Hurwitz spaces $H_{0,n-1,0}$ and $H_{1,n}$ respectively.  Motivated by these examples, we construct the following diagram
\[
  \begin{tikzcd}
  H_{0,n}\cong \text{Orbit space of $A_n$} \arrow{r}{1} \arrow[swap]{d}{2} & H_{0,n-1,0}\cong \text{Orbit space of $\tilde A_n$} \arrow{d}{4} \\%
H_{1,n}\cong \text{Orbit space of $\Ja(A_n)$} \arrow{r}{3}& H_{1,n-1,0}\cong ?
  \end{tikzcd}
\]
From the Hurwitz space side, the vertical lines 2 and 4 mean that we increase the genus by $1$, and the horizontal line means that we split one pole of order $n+1$ in a simple pole and a pole of order $n$. From the orbit space side, the vertical line 2 means that we are doing an extension from the finite Coxeter group $A_n$ to the Jacobi group $\Ja(A_n)$, the line horizontal line 1 means that we are extending the \text{Orbit space of $A_n$} to the extended affine Weyl group $\tilde A_n$. Therefore, one might ask if the line 3 and 4 would imply  an orbit space interpretation of the Hurwitz space $H_{1,n-1,0}$. The main goal of this thesis is to define a new class of groups such that its orbit space carries Dubrovin-Frobenius structure of $H_{1,n-1,0}$. The new group is called extended affine Jacobi group $A_n$, and denoted by $\Ja(\tilde A_n)$. This group is an extension of the Jacobi group $\Ja(A_n)$,  and of the extended affine Weyl group $\tilde A_n$. 

\subsection{Results}\label{Thesis results}

The main goal of this thesis is to derive the Dubrovin Frobenius structure of the Hurwitz space $H_{1,n-1,0}$ from the data of the group $\Ja(\tilde A_n)$. First of all, we define  the group $\Ja(\tilde A_n)$. Recall that the group $A_n$ acts on $\mathbb{C}^n$ by permutations, then the group $\Ja(\tilde A_n)$ is an extension of the group $A_n$ in the following sense:

 \begin{proposition}\label{definition of the action of the group tildeAn in the introduction}
The group $\Ja(\tilde A_n)\ni (w,t,\gamma)$ acts on $\Omega:=\mathbb{C}\oplus \mathbb{C}^{n+2}\oplus \mathbb{H} \ni (u,v,\tau)$ as follows
\begin{equation}\label{jacobigroupAntilde in the introduction}
\begin{split}
&w(u,v,\tau)=(u,wv,\tau)\\
&t(u,v,\tau)=(u-<\lambda,v>_{\tilde A_n}-\frac{1}{2}<\lambda,\lambda>_{\tilde A_n}\tau,v+\lambda\tau+\mu,\tau)\\
&\gamma(u,v,\tau)=(u+\frac{c<v,v>_{\tilde A_n}}{2(c\tau+d)},\frac{v}{c\tau+d},\frac{a\tau+b}{c\tau+d})\\
\end{split}
\end{equation}
\end{proposition}
where $w \in A_n$ acts by permutations in the first $n+1$ variables of $ \mathbb{C}^{n+2} \ni v=(v_0,v_1,..,v_n,v_{n+1})$, 
\begin{equation*}
\begin{split}
&t=(\lambda,\mu) \in \mathbb{Z}^{n+2}, \\
&\begin{pmatrix} 
  a    & b\\ 
c & d
\end{pmatrix} \in SL_2(\mathbb{Z}),\\
&<v,v>_{\tilde A_n}=\left.\sum_{i=0}^n v_i^2\right |_{\sum v_i=0}-n(n+1)v_{n+1}^2.
\end{split}
\end{equation*}
See section \ref{The Group Jtilde An}  for details. \\

In order to define any geometric structure in a orbit space, first it is necessary to define a notion of invariant  $\Ja(\tilde A_n)$ sections. For this purpose, we generalise the ring of invariant functions used in \cite{Bertola M.1}, \cite{Bertola M.2} for the group $\Ja( A_n)$, which  are called Jacobi forms. This notion was first defined in \cite{M. Eichler and D. Zagier} by Eichler and Zagier for the group $\Ja(A_1)$, and further it was generalised for the group $\Ja( A_n)$ in \cite{K. Wirthmuller} by Wirthmuller. Furthermore, an explicit base of generators were derived in \cite{Bertola M.1}, \cite{Bertola M.2} by Bertola. The  Jacobi forms used in this thesis are defined by
\begin{definition}\label{Jacobi forms definition jtildean in the introduction}
The weak $\tilde A_n$ -invariant Jacobi forms of weight $k$, order $l$, and index $m$ are functions on $\Omega=\mathbb{C}\oplus \mathbb{C}^{n+2}\oplus\mathbb{H}\ni (u,v_0,v_1,...,v_{n+1},\tau)=(u,v,\tau)$ which satisfy
\begin{equation}\label{jacobiform in the introduction}
\begin{split}
&\varphi(w(u,v,\tau))=\varphi(u,v,\tau),\quad \textrm{$A_n$ invariant condition}\\ 
&\varphi(t(u,v,\tau))=\varphi(u,v,\tau)\\
&\varphi(\gamma(u,v,\tau))=(c\tau+d)^{-k}\varphi(u,v,\tau)\\
&E\varphi(u,v,\tau):=-\frac{1}{2\pi i}\frac{\partial}{\partial u}\varphi(u,v,\tau)=m\varphi(u,v,\tau),\quad \textrm{Euler vector field}\\ 
\end{split}
\end{equation}
\end{definition}
Moreover, the weak $\tilde A_n$ -invariant Jacobi forms  are meromorphic in the variable $v_{n+1}$ on a fixed divisor, in contrast with the Jacobi forms of the group $\Ja(A_n)$ ,which are holomorphic in each variable, see details on the definition \ref{Jacobi forms definition jtildean abc}. The ring of  weak $\tilde A_n$ -invariant Jacobi forms gives the notion of Euler vector field, indeed, the vector field defined in the last equation of (\ref{jacobiform in the introduction}) measures the degree of Jacobi form, which coincides with the index. The differential geometry of the orbit space of the group $\Ja(\tilde A_n)$ should be understood as the space such that its sections are written in terms of  Jacobi forms. Then, in order for this statement to make sense, we prove a Chevalley type theorem, which is
\begin{theorem}\label{chevalley in the introduction}
The trigraded algebra of Jacobi forms $J_{\bullet,\bullet,\bullet}^{\Ja(\tilde A_{n})}=\bigoplus _{k,l,m}J_{k,l,m}^{\tilde A_{n}}$ is freely generated by $n+1$ fundamental Jacobi forms $(\varphi_0,\varphi_1,,\varphi_2,..,,\varphi_n)$ over the graded ring  $E_{\bullet,\bullet}$
\begin{equation}
J_{\bullet,\bullet,\bullet}^{\Ja(\tilde A_n)}=E_{\bullet,\bullet}\left[\varphi_0,\varphi_1,,\varphi_2,..,,\varphi_n\right],
\end{equation}
where 
\begin{equation*}
E_{\bullet,\bullet}=J_{\bullet,\bullet,0} \quad \text{is the ring of coefficients.}
\end{equation*}
More specifically, the ring of function $E_{\bullet,\bullet}$ is the space of functions $f(v_{n+1},\tau)$ such that for fixed $\tau$, the functions $\tau\mapsto f(v_{n+1},\tau)$ is an elliptic function.
\end{theorem}
Moreover, $(\varphi_0,\varphi_1,,\varphi_2,..,,\varphi_n)$ are given by 
\begin{corollary}\label{corollary superpotential antilde in the introduction}
The functions $(\varphi_0,\varphi_1,..,\varphi_n)$ obtained by the formula
\begin{equation}\label{superpotentialAn in the introduction}
\begin{split}
\lambda^{\tilde A_n}&=e^{2\pi i u}\frac{\prod_{i=0}^n \theta_1(z-v_i+v_{n+1},\tau)}{\theta_1^{n}(z,\tau)\theta_1(z+(n+1)v_{n+1})}\\
&=\varphi_n\wp^{n-2}(z,\tau)+\varphi_{n-1}\wp^{n-3}(z,\tau)+...+\varphi_{2}\wp(z,\tau)\\
&+\varphi_{1}\left[\zeta(z,\tau)-\zeta(z+(n+1)v_{n+1},\tau)+\zeta((n+1)v_{n+1})\right]+\varphi_0
\end{split}
\end{equation}
 are Jacobi forms of weight $0,-1,-2,..,-n$ respectively, index 1, and order $0$. Here $\theta_1$ is the Jacobi theta 1 function and $\zeta$ is the Weiestrass zeta function.
\end{corollary}
 This lemma realises the functions $(\varphi_0,\varphi_1,,\varphi_2,..,,\varphi_n,v_{n+1},\tau)$ as coordinates of the orbit space of $\Ja(\tilde A_n)$. 
The unit vector field is chosen to be 
\begin{equation}\label{unit vector field in the introduction}
e=\frac{\partial}{\partial \varphi_0},
\end{equation}
because $\varphi_0$ is the basic generator with maximum weight degree, see the sections \ref{Jacobi forms of Jantilde}, \ref{Proof of the Chevalley theorem}  for details. The last ingredient we need to construct is the flat pencil metric associated with the orbit space of $\Ja(\tilde A_n)$, which is  two compatible flat metric $g^{*}$ and $\eta^{*}$, which are sections of the cotangent bundle of the orbit space of $\Ja(\tilde A_n)$ such that
\begin{equation*}
g^{*}+\lambda\eta^{*}
\end{equation*}
is also flat, and the linear combination of its Christoffel symbols 
\begin{equation*}
{\Gamma_k^{ij}}_{g^{*}}+\lambda{\Gamma_k^{ij}}_{\eta^{*}}
\end{equation*}
is the Christoffel symbol of the flat pencil of metrics (see section \ref{Differential geometry preliminaries}  for the details) \cite{B. Dubrovin1}, \cite{B. Dubrovin2}. The natural candidate to be one of the metrics of the pencil is the invariant metric of the group $\Ja(\tilde A_n)$. This metric is given by
\begin{equation}\label{intersection form in the introduction}
g^{*}=\sum_{i,j}A_{ij}^{-1}\frac{\partial}{\partial v_i}\otimes\frac{\partial}{\partial v_j}-n(n+1)\frac{\partial}{\partial v_{n+1}}\otimes\frac{\partial}{\partial v_{n+1}}+\frac{\partial}{\partial \tau}\otimes\frac{\partial}{\partial u}+\frac{\partial}{\partial u}\otimes\frac{\partial}{\partial \tau},
\end{equation}
where
\begin{equation*}
A_{ij}=\begin{pmatrix}
2 & 1 & 1& ...& 1 \\
1 & 2 & 1& ...& 1 \\
1 & 1 & 2& ...& 1 \\
1 & 1 & 1& ...& 1 \\
1 & 1 & 1& ...& 2 \\
\end{pmatrix}
\end{equation*}
This metric is called intersection form. The second metric is given by
\begin{equation*}
\eta^{*}:=Lie_{e}g^{*},
\end{equation*}
and it is denoted by Saito metric due to K.Saito, who was the first to define this metric for the case of finite Coxeter group \cite{Saito}. One of the main technical problems of the thesis is to prove that the Saito metric $\eta^{*}$ is flat. For this purpose, we construct a generating function for the coefficients of the metric $\eta^{*}$ in the coordinates $(\varphi_0,\varphi_1,,\varphi_2,..,,\varphi_n,v_{n+1},\tau)$. We prove the following.

\begin{corollary}
Let  $\eta^{*}(d\varphi_i,d\varphi_j)$ be given by
\begin{equation}\label{metric eta def in the introduction}
\begin{split}
 \eta^{*}(d\varphi_i,d\varphi_j):=\frac{\partial g^{*}(d\varphi_i,d\varphi_j)}{\partial \varphi_0}.
\end{split}
\end{equation}
 The coefficient $\eta^{*}(d\varphi_i,d\varphi_j)$ is recovered by the generating formula
\begin{equation}\label{generating formula of Mellipticeta in the introduction}
\begin{split}
&\sum_{k,j=0}^{n+1}\frac{(-1)^{k+j}}{(k-1)!(j-1)!}\tilde \eta^{*}(d\varphi_i,d\varphi_j)\wp(v)^{(k-2)}\wp(v^{\prime})^{(j-2)}=\\
&=2\pi i(D_{\tau}\lambda(v)+D_{\tau}\lambda(v^{\prime}))+\frac{1}{2}\frac{\wp^{\prime}(v)+\wp^{\prime}(v^{\prime})}{\wp(v)-\wp(v^{\prime})}[ \frac{d\lambda(v^{\prime})}{dv^{\prime}}-\frac{d\lambda(v)}{dv}]\\
&-\frac{1}{n(n+1)}\frac{\partial}{\partial \varphi_0}\left(\frac{\partial\lambda(p)}{\partial v_{n+1}} \right)\frac{\partial\lambda(p^{\prime})}{\partial v_{n+1}}-\frac{1}{n(n+1)}\frac{\partial\lambda(p)}{\partial v_{n+1}}\frac{\partial}{\partial \varphi_0}\left(\frac{\partial\lambda(p^{\prime})}{\partial v_{n+1}} \right)\\
&+\frac{1}{n(n+1)}\sum_{k,j=0}^{n}\frac{(-1)^{k+j}}{(k-1)!(j-1)!}\frac{\partial}{\partial\varphi_0}\left(\frac{\partial\varphi_j}{\partial v_{n+1}}\right)\frac{\partial\varphi_k}{\partial v_{n+1}}\\
&+\frac{1}{n(n+1)}\sum_{k,j=0}^{n}\frac{(-1)^{k+j}}{(k-1)!(j-1)!}\frac{\partial\varphi_j}{\partial v_{n+1}}\frac{\partial}{\partial\varphi_0}\left(\frac{\partial\varphi_k}{\partial v_{n+1}}\right),
\end{split}
\end{equation}
where 
\begin{equation}\label{metric eta consequence in the introduction}
\begin{split}
\tilde \eta^{*}(d\varphi_i,d\varphi_j)= \eta^{*}(d\varphi_i,d\varphi_j), \quad i,j\neq 0,\\
\tilde \eta^{*}(d\varphi_0,d\varphi_j)= \eta^{*}(d\varphi_i,d\varphi_j)+4\pi i k_j\varphi_j.
\end{split}
\end{equation}
\end{corollary}
See section \ref{Intersection form jantilde}    for the definition of $D_{\tau}$ and for further details. 
Thereafter, we extract the coefficients $\eta^{*}(d\varphi_i,d\varphi_j)$ from the generating function (\ref{generating formula of Mellipticeta in the introduction})
\begin{theorem}\label{main lemma coefficients eta in varphi coordinates in the introduction}
The coefficients  $\eta^{*}(d\varphi_i,d\varphi_j)$ can be obtained by the formula,
\begin{equation}\label{formula main theorem for eta varphi in the introduction}
\begin{split}
\tilde\eta^{*}(d\varphi_i,d\varphi_j)&=(i+j-2)\varphi_{i+j-2}, \quad i,j\neq 0\\
\tilde\eta^{*}(d\varphi_i,d\varphi_0)&=0, \quad i\neq 0, \quad i\neq 1,\\
\tilde\eta^{*}(d\varphi_1,d\varphi_j)&=0. \quad j\neq 0,\\
\tilde\eta^{*}(d\varphi_1,d\varphi_0)&=\wp((n+1)v_{n+1})\varphi_1.\\
\end{split}
\end{equation}
\end{theorem}
Inspired by the construction of the flat coordinates of the Saito metric done in \cite{Saito}, we construct explicitly the  coordinates $t^0,t^1,t^2,..,t^n,v_{n+1},\tau$ by the following formulae

\begin{equation}\label{relation t wrt varphi 1 in the introduction}
\begin{split}
t^{\alpha}&=\frac{n}{n+1-\alpha}\left(\varphi_{n}\right)^{\frac{n+1-\alpha}{n}}\left( 1+\Phi_{n-\alpha} \right)^{\frac{n+1-\alpha}{n}},\\
t^0&=\varphi_0-\frac{\theta_1^{\prime}((n+1)v_{n+1})}{\theta_1((n+1)v_{n+1})}\varphi_1+4\pi ig_1(\tau)\varphi_2,
\end{split}
\end{equation}
where 
\begin{equation}\label{definition of formal powers of varphi in the introduction}
\begin{split}
\left( 1+\Phi_i \right)^{\frac{n+1-\alpha}{n}}=\sum_{d=0}^{\infty}  {{\frac{n+1-\alpha}{n}}\choose{d}}  \Phi_i^d,\\
\Phi_i^d=\sum_{i_1+i_2+..+i_d=i} \frac{\varphi_{\left(n-i_1\right)}}{\varphi_{n}}....\frac{\varphi_{\left(n-i_d\right)}}{\varphi_{n}}.
\end{split}
\end{equation}
The Theorem \ref{main lemma coefficients eta in varphi coordinates in the introduction} together with the formulae (\ref{relation t wrt varphi 1 in the introduction}), and some extra auxiliary technical lemmas, implies the flatness of the Saito metric $\eta$.
\begin{theorem}\label{main theorem 1 in the introduction}
Let $(t^0,t^1,t^2,..,t^{n})$ be defined in (\ref{relation t wrt varphi 1 in the introduction}), and $\eta^{*}$ the Saito metric. Then,
\begin{equation}
\begin{split}
&\eta^{*}(dt^{\alpha},dt^{n+3-\beta})=-(n+1)\delta_{\alpha\beta}, \quad 2\leq\alpha,\beta\leq n\\
&\eta^{*}(dt^1,dt^{\alpha})=0,\\
&\eta^{*}(dt^0,dt^{\alpha})=0,\\
&\eta^{*}(dt^i,d\tau)=-2\pi i \delta_{i0},\\
&\eta^{*}(dt^i,dv_{n+1})=-\frac{\delta_{i1}}{n+1}.\\
\end{split}
\end{equation}
Moreover, the coordinates $t^0,t^1,t^2,..,t^{n},v_{n+1},\tau$ are the flat coordinates of $\eta^{*}$. See details in sections \ref{The second metric of the pencil jantilde} and \ref{Flat coordinates of Jantilde}.
\end{theorem}
A remarkable fact to point out is that, even thought the intersection form $g^{*}$ and their Levi Civita connection are $\Ja(\tilde A_n)$ invariant sections, their coefficients $g^{ij}$, $\Gamma_{k}^{ij}$ are not Jacobi forms, but they live in an extension of the Jacobi forms ring. Hence, we have the following lemma
\begin{lemma}\label{lemma gab index in the introduction}
The coefficients of the intersection form $g^{\alpha\beta}$ and its Christoffel symbol $\Gamma_{\gamma}^{\alpha\beta}$ on the coordinates $t^0,t^1,..,t^n,v_{n+1},\tau$ belong to the ring $\widetilde E_{\bullet,\bullet}[t^0,t^1,..,t^n,\frac{1}{t^n}]$, where $\widetilde E_{\bullet,\bullet}$ is a suitable extension of the ring $E_{\bullet,\bullet}$.
\end{lemma}
This lemma \ref{lemma gab index in the introduction} is important because it gives a tri-graded ring structure for the coefficients  $g^{\alpha\beta}$ and $\Gamma_{\gamma}^{\alpha\beta}$. In particular, the lemma \ref{lemma gab index in the introduction} implies that $g^{\alpha\beta}$ and $\Gamma_{\gamma}^{\alpha\beta}$ are eigenfunctions of the Euler vector field given by the last equation (\ref{jacobiform in the introduction}). Using this fact, one can prove that $g^{\alpha\beta}$ and $\Gamma_{\gamma}^{\alpha\beta}$ are at most linear on the variable $t^0$, and this fact together with Theorem \ref{main theorem 1 in the introduction} proves that $g^{*},\eta^{*}$ form a flat pencil metric, and consequently we can prove 

\begin{lemma}\label{main final lemma last chapter 1 in the introduction}
Let the intersection form be (\ref{intersection form in the introduction}), the unit vector field be (\ref{unit vector field in the introduction}), and the Euler vector field  be given by the last equation of (\ref{jacobiform in the introduction}). Then, there exist a  function 
\begin{equation}\label{WDVV solution in the introduction}
F(t^0,t^1,t^2,..,t^n.v_{n+1},\tau)=-\frac{(t^0)^2\tau}{4\pi i}+\frac{t^0}{2}\sum_{\alpha,\beta\neq 0,\tau}\eta_{\alpha\beta}t^{\alpha}t^{\beta}+G(t^1,t^2,..,t^n,v_{n+1},\tau),
\end{equation}
such that
\begin{equation}\label{quasi homogeneous condition of Dubrovin Frobenius structure in the introduction}
\begin{split}
&Lie_{E}F=2F+\text{quadratic terms},\\
&Lie_{E}\left(  F^{\alpha\beta}   \right)=g^{\alpha\beta},\\
&\frac{\partial ^2G(t^1,t^2,..,t^n,v_{n+1},\tau)}{\partial t^{\alpha}\partial t^{\beta}} \in \widetilde E_{\bullet,\bullet}[t^1,t^2,..,t^n,\frac{1}{t^n}],
\end{split}
\end{equation}
where 
\begin{equation}
F^{\alpha\beta}  =\eta^{\alpha\alpha^{\prime}}\eta^{\beta\beta^{\prime}}\frac{\partial F^2}{\partial t^{\alpha^{\prime}} \partial t^{\beta^{\prime}} }.
\end{equation}

\end{lemma}
Using the lemma \ref{main final lemma last chapter 1 in the introduction} with some more technical results, we obtain our final result

\begin{theorem}\label{Dubrovin Frobenius structure of the orbit space in the introduction}
A suitable covering of the orbit space $\left(\mathbb{C}\oplus\mathbb{C}^{n+1}\oplus\mathbb{H}\right)/\Ja(\tilde A_n)$ with the intersection form (\ref{intersection form in the introduction}), unit vector field (\ref{unit vector field in the introduction}), and Euler vector field  given by the last equation of (\ref{jacobiform in the introduction}) has a Dubrovin Frobenius manifold structure. Moreover, a suitable covering of $\mathbb{C}\oplus\mathbb{C}^{n+1}\oplus\mathbb{H}/\Ja(\tilde A_n)$ is isomorphic as Dubrovin Frobenius manifold to a suitable covering of the Hurwitz space $H_{1,n-1,0}$.
\end{theorem}
See section \ref{Construction of WDVV solution}  for details.

The results of this thesis would be important because of the following 
 \begin{enumerate}
\item The Hurwitz spaces $H_{1,n-1,0}$ are classified by the group $\Ja(\tilde A_n)$, hence we increase the knowledge of the WDVV/ discrete group correspondence. In particular the WDVV solutions associated with this orbit spaces contains a kind of  elliptic function in an exceptional variable, which is exotic  in theory of WDVV solutions, since most of the known examples are polynomial or polynomial with exponential function. Recently, the case $\Ja(\tilde A_1)$ attracted the attention of experts due to its application in integrable systems \cite{M. Cutimanco and V. Shramchenko}, \cite{E. V. Ferapontov M. V. Pavlov L. Xue}, \cite{Romano 1}.
\item It is well known that Hurwitz spaces are related to matrix models, then, if one derive the associated matrix model of the Hurwitz spaces $H_{1,n-1,0}$, we would immediately classify it by the group $\Ja(\tilde A_n)$.
\item Even thought the orbit space of $\Ja(\tilde A_n)$ is locally isomorphic as Dubrovin Frobenius manifold to the Hurwtiz space $H_{1,n-1,0}$, these two space are not necessarily the same. Indeed, the Dubrovin Frobenius manifold associated to Hurwtiz spaces is a local construction, because it is defined in a domain of a solution of a Darboux-Egoroff system. On another hand, orbit spaces are somehow global objects, because their ring of invariant function are polynomial over a suitable ring. In addition, the notion of invariant function gives information about the non-cubic part of the WDDV solution associated with the orbit space, see the last equation of (\ref{quasi homogeneous condition of Dubrovin Frobenius structure in the introduction}) for instance. 
\item The orbit space construction of the group $\Ja(\tilde A_n)$ can be generalised  to the other classical finite Coxeter groups as $B_n, D_n$. Hence, these orbit spaces could give rise to a new class of Dubrovin Frobenius manifolds. Furthermore, the associated integrable hierarchies of this new class of Dubrovin Frobenius manifolds could have applications in Gromow Witten theory and combinatorics.
\end{enumerate}

The thesis is organised  in the following way. In section \ref{Review of Dubrovin-Frobenius manifolds}, we recall the basics definitions of Dubrovin-Frobenius manifolds. In section \ref{Hurwitz space chapter}, we recall the Dubrovin Frobenius manifold construction on Hurwitz spaces.  In section \ref{section An case}, we defined extended affine  Jacobi group $\Ja(\tilde A_n)$ and , we construct Dubrovin-Frobenius structure on the orbit spaces of $\Ja(\tilde A_n)$ . Furthermore, we show that the orbit space of the group $\Ja(\tilde A_n)$ is isomorphic ,as Dubrovin-Frobenius manifold, to the Hurwitz-Frobenius manifold $\tilde H_{1,n-1,0}$ \cite{B. Dubrovin2}, \cite{V. Shramchenko}. See theorem \ref{Dubrovin Frobenius structure of the orbit space} for details. The  section \ref{section An case} generalise the results of section \cite{Guilherme 1} for the group $\Ja(\tilde A_n)$. \\

\section*{Acknowledgements}
I am grateful to Professor Boris Dubrovin for proposing this problem, for his remarkable advises and guidance. I would like also to thanks Prof. Davide Guzzetti, and Prof. Marco Bertola for helpful discussions, and guidance of this paper. Moreover, the author  acknowledges  the support of the H2020-MSCA-RISE-2017 PROJECT No. 778010 IPADEGAN.

\section{Review of Dubrovin-Frobenius manifolds}\label{Review of Dubrovin-Frobenius manifolds}
\subsection{Basic definitions }
We recall the basic definitions of Dubrovin Frobenius manifold, for more details \cite{B. Dubrovin2}.
\begin{definition}
A Frobenius Algebra $\mathcal{A}$ is an unital, commutative, associative algebra endowed with an invariant non degenerate bilinear pairing 
\begin{equation*}
\eta(,):\mathcal{A}\otimes\mathcal{A}\mapsto \mathbb{C},
\end{equation*}
 invariant in the following sense:
\begin{equation*}
\eta(A\bullet B,C)=\eta(A,B\bullet C),
\end{equation*}
$\forall A,B,C \in \mathcal{A}$.
\end{definition}
\begin{definition} 
$M$ is smooth or complex Dubrovin-Frobenius manifold of dimension $n$ if a structure of Frobenius algebra is specified on any tangent plane $T_tM$ at any point $t\in M$, smoothly depending on the point such that: 
\begin{enumerate}
\item The invariant inner product $\eta(,)$ is a flat metric on M. The flat coordinates of $\eta(,)$ will be denoted by $(t^1,t^2,..,t^n)$.
\item The unity vector field $e$ is covariantly constant w.r.t. the Levi-Civita connection $\nabla$ for the metric $\eta(,)$
\begin{equation}
\nabla e=0
\end{equation}
\item Let
\begin{equation}
c(u, v, w) :=\eta( u\bullet v, w )
\end{equation}
(a symmetric 3-tensor). We require the 4-tensor
\begin{equation}
(\nabla_z c)(u, v, w) 
\end{equation}
to be symmetric in the four vector fields $u, v, w, z.$
\item A vector field $E$ must be determined on M such that:
\begin{equation}
\nabla\nabla E=0
\end{equation}
and that the corresponding one-parameter group of diffeomorphisms acts by conformal
transformations of the metric $\eta$, and by rescalings on the Frobenius algebras $T_tM$ .
Equivalently:
\begin{equation}
[E,e]=-e,
\end{equation} 
\begin{equation}\label{conditionFM41}
\begin{split}
\mathcal{L}_E\eta(X,Y)&:=E\eta(X,Y)-\eta([E,X],Y)-\eta(X,[E,Y])\\&=(2-d)\eta(X,Y),
\end{split}
\end{equation} 
\begin{equation}\label{conditionFM42}
\begin{split}
\mathcal{L}_Ec(X,Y,Z)&:=Ec(X,Y,Z)-c([E,X],Y,Z)-c(X,[E,Y],Z)\\&-c(X,Y,[E,Z])=(3-d)c(X,Y,Z).
\end{split}
\end{equation}
\end{enumerate}

The Euler vector $E$ can be represented as follows:
\begin{lemma}
If the grading operator $Q:=\nabla E$ is diagonalizable, then $E$ can be represented as:
\begin{equation}
E=\sum_{i=1}^n ((1-q_i)t_i+r_i)\partial_i
\end{equation}
\end{lemma}

We now define scaling exponent as follows:
\begin{definition}
A function $\varphi:M\mapsto \mathbb{C}$ is said to be quasi-homogeneous of scaling exponent $d_{\varphi}$, if it is a eigenfunction of Euler vector field:
\begin{equation}
E(\varphi)=d_{\varphi}\varphi
\end{equation}
\end{definition}

\begin{definition}
The function $F(t),t=(t^1,t^2,..,t^n)$ is a solution of WDVV equation if its third derivatives  
\begin{equation}\label{freeenergy}
c_{\alpha\beta\gamma}=\frac{\partial^3F}{\partial t^{\alpha}\partial t^{\beta}\partial t^{\gamma}}
\end{equation}
satisfy the following conditions:

\begin{enumerate}

\item \begin{equation*}
\eta_{\alpha\beta}=c_{1\alpha\beta}
\end{equation*}
 is constant nondegenerate matrix.
\item  The function 
\begin{equation*} 
c_{\alpha\beta}^{\gamma}=\eta^{\gamma\delta}c_{\alpha\beta\delta}
\end{equation*}
is structure constant of associative
algebra.
\item  F(t) must be quasi-homogeneous function
\begin{equation*}
F(c^{d_1}t^1,..,c^{d_n}t^n)=c^{d_F}F(t^1,..,t^n)
\end{equation*}
for any nonzero $c$ and for some numbers $d_1, ..., d_n, d_F$.
\end{enumerate}
\end{definition}

\begin{lemma}\label{onetoonefenergy}
Any solution of WDVV equations with $d_1\neq0$ defined in a domain
$t\in M$ determines in this domain the structure of a Dubrovin-Frobenius manifold. Conversely, locally any Dubrovin-Frobenius manifold is related to some solution of WDVV equations.
\end{lemma}

\end{definition}

\subsection{Intersection form}
\begin{definition}
Let $x=\eta(X,),y=\eta(Y,)\in \Gamma(T^{*}M)$ where $X,Y\in \Gamma(TM)$. An induced Frobenius algebra is defined on $\Gamma(T^{*}M)$ by:
\begin{equation*}
x\bullet y=\eta(X\bullet Y, ).
\end{equation*}
\end{definition}

\begin{definition}
The intersection form is a bilinear pairing in $T^{*}M$ defined by:
\begin{equation*}
(\omega_1,\omega_2)^{*}:=\iota_E(\omega_1\bullet\omega_2)
\end{equation*}
where $\omega_1,\omega_2\in T^{*}M$ and $\bullet$ is the induced Frobenius algebra product in the cotangent space. the intersection form will be denoted by $g^{*}$ .
\end{definition}
\noindent\textbf{Remark 1:}
Let $x=\eta(X,),y=\eta(Y,)\in \Gamma(T^{*}M)$. Then:
\begin{equation*}
g^{*}(x,y)=\eta(X\bullet Y, E)=c(X,Y,E).
\end{equation*}
\textbf{Remark 2:}
It is possible to prove that the tensor $g^{*}$ defines a bilinear form on the tangent bundle that is almost everywhere non degenerate \cite{B. Dubrovin2}.

\begin{proposition}
The metric $g^*$ is flat, and $\forall \lambda\in \mathbb{C}$, the contravariant metric $\eta^{*}(,)+\lambda g^*(,)$ is flat, and the contravariant connection is $\nabla^{\eta}+\lambda\nabla^{g}$, where $\nabla^{\eta},\nabla^{g}$ are the contravariant connections of $\eta^{*}$ and $g^*$ respectively. The family of metrics  $\eta^{*}(,)+\lambda g^*(,)$  is called Flat pencil of metrics.
\end{proposition}

\begin{lemma} The induced metric $\eta^*$ on the cotangent bundle $T^*M$ can be written as Lie derivative with respect the unit vector field $e$ of the intersection form $g^*$. i.e
\begin{equation}
\eta^*=\mathcal{L}_eg^*.
\end{equation}
\end{lemma}

\begin{lemma} The correspondent WDVV solution $F(t^1,..,t^n)$ of the Dubrovin-Frobenius manifold  works as potential function for $g^*$. More precisely:
\begin{equation}\label{fgformula}
g^*(dt^i,dt^j)=(1+d-q_i-q_j)\nabla_{(dt^i)^{\sharp}}\nabla_{(dt^j)^{\sharp}}F.
\end{equation}
where the form $(dt^j)^{\sharp}$ is the image of  $dt^j$ by the isomorphism  induced by the metric $\eta$. 
\end{lemma}

\subsection{Reconstruction}\label{Reconstruction subsection}
Suppose that given a Dubrovin-Frobenius manifold M, only the following data are known: intersection form $g^{*}$, unit vector field $e$, Euler vector field $E$. From the previous lemmas we can reconstruct the Dubrovin-Frobenius manifold by setting:
\begin{equation}
\eta^{*}=\mathcal{L}_eg^*.
\end{equation}
Then, we find the flat coordinates of $\eta$  as homogeneous functions, and the structure constants by imposing:
\begin{equation}
g^*(dt^i,dt^j)=(1+d-q_i-q_j)\nabla_{(dt^i)^{\sharp}}\nabla_{(dt^j)^{\sharp}}F.
\end{equation}
Therefore, it is possible to compute the Free-energy by integration. Of course, we may have obstructions when, $1+d=q_i+q_j$.

\subsection{Monodromy of Dubrovin-Frobenius manifold}\label{Monodromy of Dubrovin-Frobenius manifold}
The intersection form g of a Dubrovin-Frobenius manifold is a flat almost everywhere nondegenerate metric. 
Let 
\begin{equation*}
 \Sigma=\{t\in M: det( g)=0     \} 
\end{equation*}

 Hence, the linear system of differential equations determining $g^{*}$-flat coordinates 
 \begin{equation*}
 g^{\alpha\epsilon}(t)\frac{\partial^2 x}{\partial t^{\beta}\partial t^{\epsilon}}+\Gamma_{\beta}^{\alpha\epsilon}(t)\frac{\partial x}{\partial t^{\epsilon}}=0
 \end{equation*}
 has poles, and consequently its solutions $x_a(t^1,..,t^n)$ are multivalued, where $(t^1,..,t^n)$ are flat coordinates of $\eta$. 
The analytical continuation of the solutions $x_a(t^1,..,t^n)$ has monodromy corresponding to loops around $\Sigma$. This gives rise to a monodromy representation of $\pi_1(M\setminus\Sigma)$, which is called Monodromy of the Dubrovin-Frobenius manifold.

\subsection{Dubrovin Connection}

In the theory of Dubrovin Frobenius manifold, there is another way to associate a monodromy group on it. Consider the following deformation of the Levi-Civita connection defined on a Dubrovin Frobenius manifold $M$
\begin{equation}
\tilde\nabla_uv:=\nabla_u v+zu\bullet v, \quad u,v\in \Gamma(TM),
\end{equation}
where $\nabla$ is the Levi-Civita connection of the metric $\eta$, $\bullet$ is the Frobenius product, and $z\in \mathbb{CP}^1$. Then, the following connection defined in $M\times\mathbb{CP}^1$
\begin{equation}
\begin{split}
\tilde\nabla_uv&:=\nabla_u v+zu\bullet v,\\
\tilde\nabla_{\frac{d}{dz}}\frac{d}{dz}&=0, \quad \tilde\nabla_{v} \frac{d}{dz}=0,\\
\tilde\nabla_{\frac{d}{dz}}&v=\partial_zv+E\bullet v-\frac{1}{z}\mu(v).
\end{split}
\end{equation}
where $\mu$ is the diagonal matrix be given by
\begin{equation}
\begin{split}
\mu_{\alpha\beta}=(q_{\alpha}-\frac{d}{2})\delta_{\alpha\beta}.
\end{split}
\end{equation}
The monodromy representation arise by considering the solutions of the flat coordinate systems
\begin{equation}\label{flat coordinate system of the deformed eta}
\begin{split}
\tilde\nabla d\tilde t=0.
\end{split}
\end{equation}
After doing some Gauge transformations in the system (\ref{flat coordinate system of the deformed eta}), and writing it in matricidal form. The system \ref{flat coordinate system of the deformed eta} takes the form
\begin{equation}\label{flat coordinate system of the deformed eta 1}
\begin{split}
\frac{dY}{dt^{\alpha}}&=zC_{\alpha}Y,\\
\frac{dY}{d z}&=\left(U+\frac{\mu}{z} \right)Y,\\
\end{split}
\end{equation}
where 
\begin{equation}
\begin{split}
{C_{\alpha}}_{\beta}^{\gamma}=c_{\alpha\beta}^{\gamma}, \quad U_{\beta}^{\gamma}=E^{\epsilon}c_{\beta\epsilon}^{\gamma}.
\end{split}
\end{equation}

\subsection{Differential geometry preliminaries}\label{Differential geometry preliminaries}
In order to derive the Dubrovin Frobenius manifolds, we recall some results related with Riemannian geometry of the contravariant "metric" $g^{ij}$. By metric, I mean symmetric, bilinear, non-degenerate. In coordinates, let the metric 
\begin{equation*}
g_{ij}dx^idx^j
\end{equation*}
and its induced contravariant metric
\begin{equation*}
g^{ij}\frac{\partial}{\partial x_i}\otimes\frac{\partial}{\partial x_j}
\end{equation*}
The Levi Civita connection is uniquely specified by
\begin{equation}\label{Christoffel symbol covariant 1}
\nabla_kg_{ij}=\partial_kg_{ij}-\Gamma_{ki}^sg_{sj}-\Gamma_{kj}^sg_{is}=0,
\end{equation}
or 
\begin{equation}\label{Christoffel symbol covariant 2}
\nabla_kg^{ij}=\partial_kg^{ij}-\Gamma_{ks}^ig^{sj}-\Gamma_{ks}^jg^{is}=0,
\end{equation}
and 
\begin{equation}\label{Christoffel symbol covariant 3}
\Gamma_{ij}^k=\Gamma_{ji}^k.
\end{equation}
The Christoffel symbol can be written as
\begin{equation*}
\Gamma_{ij}^k=g^{ks}\left( \partial_ig_{sj}+\partial_jg_{is}-\partial_sg_{ij}            \right).
\end{equation*}
But, for our purpose it will be more convenient to use
\begin{equation}
\Gamma^{ij}_k:=-g^{is}\Gamma_{sk}^j.
\end{equation}
Then, the equations (\ref{Christoffel symbol covariant 1}), (\ref{Christoffel symbol covariant 2}), and (\ref{Christoffel symbol covariant 3}) are equivalent to 
\begin{equation}\label{Levi Civita contravariant coxeter chapter}
\begin{split}
&\partial_kg^{ij}=\Gamma_k^{ij}+\Gamma_k^{ji},\\
&g^{is}\Gamma_s^{jk}=g^{js}\Gamma_s^{ik}.
\end{split}
\end{equation}
Introducing the  operators
\begin{equation}
\begin{split}
\nabla^i&=g^{is}\nabla_s,\\
\nabla^i\xi_k&=g^{is}\partial_s\xi_k+\Gamma^{is}_k\xi_s.
\end{split}
\end{equation}
The curvature tensor $R_{ijk}^l$
of the metric measures noncommutativity of the operators
$\nabla_i$or, equivalently $\nabla^i$
\begin{equation}
\begin{split}
\left(\nabla_i\nabla_j-\nabla_j\nabla_i   \right)\xi_k=R_{ijk}^l\xi_l
\end{split}
\end{equation}
where
\begin{equation}
\begin{split}
R_{ijk}^l=\partial_i\Gamma_{jk}^l-\partial_j\Gamma_{ik}^l+\Gamma_{is}^{l}\Gamma^{s}_{jk}-\Gamma_{js}^{l}\Gamma^{s}_{ik}
\end{split}
\end{equation}

We say that the metric is flat if the curvature of it vanishes. For a flat metric local flat
coordinates $p_1, ..., p_n$ exist such that in these coordinates the metric is constant and the
components of the Levi-Civita connection vanish. Conversely, if a system of flat coordinates for a metric exists then the metric is flat. The flat coordinates are determined uniquely up to an affine transformation with constant coefficients. They can be found from the following system
\begin{equation}
\begin{split}
\nabla^i\partial_kp&=g^{is}\partial_s\partial_kp+\Gamma^{is}_k\partial_sp=0.
\end{split}
\end{equation}
The correspondent Riemman tensor for the contravariant metric $g^{ij}$ can be written as
\begin{equation}
\begin{split}
R^{ijk}_l:=g^{is}g^{jt}R_{stl}^k=g^{is}\left(\partial_s\Gamma_l^{jk}-\partial_l\Gamma_s^{jk} \right)+\Gamma_{s}^{ij}\Gamma_l^{sk}-\Gamma_{s}^{ik}\Gamma_l^{sj}.
\end{split}
\end{equation}

The aim of this section is to construct a Dubrovin Frobenius structure on the orbit space of $ A_n$. The strategy to achieve this goal is based on the derivation of a WDVV solution from the geometric data of the orbit space $ A_n$. More specifically, the WDVV solution will be derived from the flat pencil structure which the orbit space of $A_n$ naturally has.

\begin{definition}\cite{B. Dubrovin1}
Two metrics $(g^{*}, \eta^{*})$ form a flat pencil if:
\begin{enumerate}
\item The metric 
\begin{equation}\label{flatpencil An}
g^{ij}_{\lambda}:=g^{ij}+\lambda\eta^{ij},
\end{equation}
is flat for arbitrary $\lambda$. 
\item The Levi-Civita connection of the metric (\ref{flatpencil An}) has the form 
\begin{equation*}
\Gamma^{ij}_{k,\lambda}:=\Gamma_{k,g}^{ij}+\lambda\Gamma_{k,\eta}^{ij}
\end{equation*}
where $\Gamma_{k,g}^{ij}$, $\Gamma_{k,\eta}^{ij}$ are the Levi-Civita connection  of $g^{*}$, and $\eta^{*}$ respectively.
\end{enumerate}
\end{definition}

 The main source of flat pencil metric is the following lemma

\begin{lemma}\cite{B. Dubrovin1}\label{mainlemmaflatpencil An}
If for a flat metric $g^{*}$ in some coordinate $a_2,a_3,..,a_{n+1}$ both the coefficients of the metric $g^{ij}$ and Levi Civita connection $\Gamma_{k}^{ij}$ are linear in the coordinate $a_{n+1}$, and if $det(g^{*})\neq 0$, then, the metric
 \begin{equation}
g^{ij}+\lambda \frac{\partial g^{ij}}{\partial a_{n+1}}
\end{equation}
form a flat pencil. The corresponding Levi-Civita connection have the form
\begin{equation}
\Gamma_{k,g}^{ij}:=\Gamma_{k}^{ij}, \quad \Gamma_{k,\eta}^{ij}:= \frac{\partial \Gamma_{k}^{ij}}{\partial a_{n+1}}
\end{equation}
\end{lemma}
Hence, our goal is first to construct a flat globally well defined metric in the orbit space of $ A_n$ such that there exist coordinates $a_2,a_3,..,a_{n+1}$ in which both the metric and its Christoffel symbols are at most linear on $a_{n+1}$.

The following lemma shows that flat pencil structure is almost the same as Dubrovin Frobenius structure.

\begin{lemma}\label{flat pencil almost Dubrovin frobenius manifold Coxeter}\cite{B. Dubrovin1}
For a flat pencil metric $g^{\alpha\beta}, \eta^{\alpha\beta}$ there exist a vector field $f=f^{\gamma}\partial_{\gamma}$ such that the tensor 
\begin{equation}\label{eq1 Coxeter}
\Delta^{\alpha\beta\gamma}=\eta^{\alpha\delta}\Gamma_{\delta,g}^{\beta\gamma}-g^{\alpha\delta}\Gamma_{\delta,\eta}^{\beta\gamma}
\end{equation}
and the metric $g^{\alpha\beta}$ have the following form
\begin{equation}\label{eq2 Coxeter}
\begin{split}
\Delta^{\alpha\beta\gamma}&=\eta^{\alpha\mu}\eta^{\beta\nu}\partial_{\mu}\partial_{\nu}f^{\gamma}\\
g^{\alpha\beta}&=\eta^{\alpha\mu}\partial_{\mu}f^{\beta}+\eta^{\beta\nu}\partial_{\nu}f^{\alpha}+c\eta^{\alpha\beta}
\end{split}
\end{equation}
for some constant $c$. The vector field $f$ should satisfy
\begin{equation}\label{eq3 Coxeter}
\begin{split}
\Delta^{\alpha\beta}_{\epsilon}\Delta^{\epsilon\gamma}_{\delta}=\Delta^{\alpha\gamma}_{\epsilon}\Delta^{\epsilon\beta}_{\delta}
\end{split}
\end{equation}
where
\begin{equation}\label{eq4 Coxeter}
\begin{split}
&\Delta^{\alpha\beta}_{\gamma}=\eta_{\gamma\epsilon}\Delta^{\alpha\beta\epsilon}=\eta^{\alpha\mu}\partial_{\mu}\partial_{\gamma}f^{\beta}\\
&(\eta^{\alpha\epsilon}g^{\beta\delta}-g^{\alpha\epsilon}\eta^{\beta\delta})\partial_{\epsilon}\partial_{\delta}f^{\gamma}=0
\end{split}
\end{equation}
Conversely, for any metric $\eta^{\alpha\beta}$ and for $f$ solution of the system (\ref{eq3 Coxeter}) and (\ref{eq4 Coxeter}) the metrics $\eta^{\alpha\beta}$ and $g^{\alpha\beta}$  form a flat pencil metric.
\end{lemma}

\subsection{Semisimple Dubrovin Frobenius manifolds}

\begin{definition}
A Frobenius algebra is called semisimple if it does not have nilpotent, i.e. if $a\neq 0$ implies
\begin{equation}
a^m\neq 0, \quad \text{for any} \quad m\in \mathbb{Z}.
\end{equation}
\end{definition}

\begin{lemma}\cite{B. Dubrovin2}
Let $A$ be a semisimple Frobenius algebra, then there exist a base $e_1,e_2,..,e_n$ of $A$, such that the Frobenius product $\bullet$ in this base is described by 
\begin{equation}
e_i\bullet e_j=\delta_{ij}e_i.
\end{equation}
\end{lemma}

\begin{definition}
A point in a Dubrovin Frobenius manifold is called semisimple, if the Frobenius algebra in its tangent space is semisimple.
\end{definition}

\begin{remark}
Note that semisimplicity is an open condition.
\end{remark}

\begin{lemma}\cite{B. Dubrovin2}
In a neighbourhood of a semi semisimple point, there exist local coordinates $(u_1,u_2,..,u_n)$ such that 
\begin{equation}
\frac{\partial}{\partial u_i}\bullet \frac{\partial}{\partial u_j}=\delta_{ij}\frac{\partial}{\partial u_i}.
\end{equation}
The coordinates $(u_1,u_2,..,u_n)$ are called canonical coordinates.
\end{lemma}

\begin{lemma}\cite{B. Dubrovin2}
Let $M$ a semisimple Dubrovin Frobenius manifold, on the canonical coordinates $(u_1,u_2,..,u_n)$ the intersection form, Euler vector field, and unit vector field can be written as 
\begin{equation}
\begin{split}
&g^{ii}=u_i\eta^{ii}\delta_{ij},\\
&e=\sum_{i=1}^n\frac{\partial}{\partial u_i},\\
&E=\sum_{i=1}^n u_i\frac{\partial}{\partial u_i}.\\
\end{split}
\end{equation}

\end{lemma}

\begin{proposition}\cite{B. Dubrovin2}
In a neighborhood of a semisimple point all the roots $(u_1,u_2,..,u_n)$ of the characteristic equation
\begin{equation}
det(g^{\alpha\beta} -u\eta^{\alpha\beta} ) = 0 
\end{equation}
are simple. They are canonical coordinates in this neighbourhood . Conversely, if the roots of the characteristic equation are simple in a point $p \in M$, then $p \in M$ is a semisimple point on the Frobenius manifold and $(u_1,u_2,..,u_n)$ are canonical coordinates in the neighbourhood of this point.
\end{proposition}

\begin{definition}\cite{B. Dubrovin2}
A diagonal metric on a n-dimensional manifold
\begin{equation}
\eta=\sum_{i=1}^n \eta_{ii}du_i^2
\end{equation}
is called potential, if there exist a function $U(u_1,u_2,..,u_n)$ such that 
\begin{equation}
\eta_{ii}=\frac{\partial U}{\partial u_i}.
\end{equation}
\end{definition}

\begin{definition}\cite{B. Dubrovin2}
A potential diagonal flat metric $\eta$ on a n-dimensional manifold is called Darbou-Egoroff metric.
\end{definition}

\begin{lemma}\cite{B. Dubrovin2}\label{Darboux Egoroff system}
Let be $\eta$ a diagonal potential metric on a n-dimensional manifold
\begin{equation}\label{metric diagonal 1}
\eta=\sum_{i=1}^n \eta_{ii}du_i^2.
\end{equation}
Then, the metric (\ref{metric diagonal 1}) is Darboux-Egoroff iff its rotational coefficients $\beta_{ij}$
\begin{equation}
\beta_{ij}=\frac{\partial_j\sqrt{\eta_{ii}}}{\sqrt{\eta_{jj}}}
\end{equation}
satisfy the system of equations
\begin{equation}
\begin{split}
\partial_k\beta_{ij}=\beta_{ik}\beta_{kj},\\
\sum_{k=1}^n\partial_k\beta_{ij}=0.
\end{split}
\end{equation}
\end{lemma}

\section{Review of Dubrovin-Frobenius manifold on Hurwitz spaces}\label{Hurwitz space chapter}

\subsection{Hurwitz spaces}
The main reference of this section are \cite{B. Dubrovin2} and \cite{V. Shramchenko}.\\

\begin{definition}
The Hurwitz space $H_{g,n_0 ,...,n_m}$ is the moduli space of curves $C_g$ of genus g
endowed with a N branched covering, $ \lambda: C_g\mapsto \mathbb{C}P ^1$ of $\mathbb{C}P^1$ with $m + 1$ branching points
over $\infty\in \mathbb{C}P^1$ of branching degree $n_i + 1, i = 0, . . . , m.$
\end{definition}

\begin{definition}
Two pairs $(C_g,\lambda)$ and $(\tilde C_g,\tilde\lambda)$ are said Hurwitz-equivalent if there exist an analytic isomorphic $F:C_g\mapsto \tilde C_g$ such that 
\begin{equation}
\lambda\circ F =\tilde\lambda.
\end{equation}
\end{definition}

Roughly speaking, Hurwitz spaces $H_{g,n_0 ,...,n_m}$ are moduli spaces of meromorphic functions which realise a Riemann surface of genus $g$ $C_g$ as covering over  $\mathbb{C}P ^1$ with a fixed ramification profile. 

\textbf{Example 1:}\\
A generic point of the Hurwitz space  $H_{0,n}$ is 
\begin{equation}
H_{0,n}=\{ \lambda(p,x_0,x_1,x_2,..,x_n)=\prod_{i=0}^n (p-x_i): \sum_{i=0}^n x_i=0\}
\end{equation}

\textbf{Example 2:}\\
A generic point of the Hurwitz space $H_{0,n-1,0}$ is 
\begin{equation}
H_{0,n}=\{ \lambda(p,a_2,a_3,..,a_{n+1},a_{n+2})=p^n+ a_2p^{n-2}+...+a_np+a_{n+1}+\frac{a_{n+2}}{p}.   \}
\end{equation}

\textbf{Example 3:}\\
A generic point of the Hurwitz space $H_{1,n}$ is 
\begin{equation}
H_{1,n}=\{ \lambda(p,u,v_0,v_1,..,v_n,\tau)=e^{-2\pi i u}\frac{\prod_{i=0}^n \theta_1(p-v_i,\tau)}{\theta_1^{n+1}(v,\tau)}: \sum_{i=0}^n v_i=0   \}
\end{equation}
here $\theta_1$ is the Jacobi theta function, see (\ref{theta def}).

\textbf{Example 4:}\\
A generic point of the Hurwitz space is $H_{1,n-1,0}$ is 
\begin{equation}
H_{1,n-1,0}=\{ \lambda(p,u,v_0,v_1,..,v_n,v_{n+1},\tau)=e^{-2\pi i u}\frac{\prod_{i=0}^n \theta_1(p-v_i,\tau)}{\theta_1^{n}(v,\tau)\theta_1(v+(n+1)v_{n+1},\tau)}:\sum_{i=0}^n v_i=-(n+1)v_{n+1}   \}
\end{equation}

The covering $\widetilde H = \tilde H_{g,n_0 ,...,n_m} $ consist of the set of points
\begin{equation*}
(C_g; \lambda; k_0, . . . , k_m; a_1, . . . , a_g, b_1, . . . , b_g) \in \tilde H_{g,n_0 ,...,n_m}
\end{equation*}
where $C_g, \lambda$  are the same as above, $a_1, . . . , a_g, b_1, . . . , b_g \in H_1(C_g,\mathbb{Z})$ are the canonical symplectic basis, and  $k_0, . . . , k_m $ are  roots of $\lambda$ near  $\infty_0 ,\infty_1 , . . . , \infty_m$  of the orders
$n_0 + 1, n_1 + 1, . . . , n_m + 1.$ resp., 
\begin{equation*}
k^{n_i+1}_i (P) = \lambda(P),\quad  \text{P near $\infty_i$}.
\end{equation*}

\subsection{Bidifferential $W$}

\begin{definition}\cite{V. Shramchenko}
Let $P,Q \in C_g$. The meromorphic Bidifferential $W$ is given by
\begin{equation}
W(P,Q)=d_Pd_Q\log E(P,Q),
\end{equation}
where $E(P,Q)$ is the prime form on the Riemann surface $C_g$. Alternatively, it can be characterised by the following properties 
\begin{enumerate}
\item symmetric  meromorphic differential in $C_g\times C_g$, with second order pole on $P=Q$ with biresidue 1
\item 
\begin{equation}
\begin{split}
\int_{a_k} W(P,Q)&=0;\\
\int_{b_k} W(P,Q)&=2\pi i \omega_k(P)
\end{split}
\end{equation}
where $\{\omega_k(P)\}$ are the normalized base of holomorphic differentials, i.e. $\int_{a_j}\omega_k(P)=\delta_{ij}$

\end{enumerate}
\end{definition}

The dependence of the bidifferential $W$ on branch points of the Riemann surface is given by the Rauch variational formulas
\begin{equation}
\begin{split}
\frac{\partial W(P,Q)}{\partial u_i}=\frac{1}{2}W(P,P_i)W(P_i,Q),
\end{split}
\end{equation}
 where $W(P,P_i)$ is the evaluation of  $W(P,Q)$ at $Q = Pj$ with respect to the standard local parameter $x_j(Q)=\sqrt{\lambda-\lambda(Q)}$
\begin{equation}
\begin{split}
W(P,P_i)=\left.\frac{W(P,Q)}{dx_j(Q)}\right |_{Q=P_j}
\end{split}
\end{equation}
A remarkable consequence of the Rauch variational formula is that it induce a flat metric in the Hurwitz space. Indeed, 
\begin{proposition}\cite{V. Shramchenko}
Let be the metric 
\begin{equation}\label{metric induced by the Bidifferential}
ds^2_{W}=\sum_{i=1}^n \left(\oint_l h(Q)W(Q,P_i) \right)^2  (du_i)^2 ,
\end{equation}
where $l$ is a smooth contour in the Riemann surface such that $P_i\notin l$, and $h(Q)$ is a smooth function independent of $\{u_i\}$. Then, the rotational coefficients of (\ref{metric induced by the Bidifferential}) satisfies a Darboux-Egoroff system in lemma \ref{Darboux Egoroff system}.
\end{proposition}

For  particular choices of the function $h(Q)$ the metrics (\ref{metric induced by the Bidifferential}) coincides with the metrics induced by the primary differentials see  section \ref{Reconstruction of Dubrovin Frobenius manifold } and \cite{B. Dubrovin2} for details. This fact is remarkable, because, it shows that from only the data of the Hurwtiz space, one can construct a flat metric for the desired Dubrovin Frobenius manifold.

\subsection{Reconstruction of Dubrovin Frobenius manifold }\label{Reconstruction of Dubrovin Frobenius manifold }

Over the space $\widetilde H_{g,n_0 ,...,n_m}$, it is possible to introduce a Dubrovin-Frobenius structure by taking as canonical coordinates  the ramification points  $(u_1,u_2,..u_n)$ of $H_{g,n_0 ,...,n_m}$. 
The Dubrovin-Frobenius structure is specified by the following objects:
\begin{equation}\label{Multiplication}
 \textnormal{multiplication}\quad  \partial_i\bullet\partial_j=\delta_{ij}\partial_i,\textnormal{where}\quad \partial_i=\frac{\partial}{\partial u_i},
\end{equation}
\begin{equation}\label{Euler vector field}
 \textnormal{Euler vector field}\quad E=\sum_i u_i\partial_i,
\end{equation}
\begin{equation}\label{Unit vector field}
 \textnormal{unit vector field}\quad e=\sum_i \partial_i,
\end{equation}
and the metric $\eta$ defined by the formula
\begin{equation}\label{metric}
ds^2_{\phi}=\sum \text{res}_{P_i}\frac{\phi^2}{d\lambda}(du_i)^2 ,
\end{equation}
where $\phi$ is some primary differential of the underlying Riemann surface $C_g$.
Note that the Dubrovin-Frobenius manifold structure depends on the meromorphic function $\lambda$, and on the primary differential $\phi$. The list of possible  primary differential $\phi$ is in \cite{B. Dubrovin2}.\\

Consider a multivalued function $p$ on C by taking the integral of $\phi$
\begin{equation*}
p(P)=v.p\int_{\infty_0}^P \phi
\end{equation*}
The principal value is defined by omitting the divergent part, when necessary, because $\phi$ may be divergent at $\infty_0$, as function of the local parameter $k_0$. Indeed the primary differentials defined on \cite{B. Dubrovin2} may diverge as functions of $k_i$.
\begin{equation*}
\phi=dp.
\end{equation*}
Let $\tilde H_{\phi}$ be the open domain in $\tilde H$ specifying by the condition
\begin{equation*}
\phi(P_i)\neq 0.
\end{equation*}

\begin{theorem}\label{flatcoord}\cite{B. Dubrovin2}
For any primary differential $\phi$ of the list in \cite{B. Dubrovin2} the multiplication (\ref{Multiplication}), the
unity (\ref{Unit vector field}), the Euler vector field (\ref{Euler vector field}), and the metric (\ref{metric}) determine a structure of Dubrovin Frobenius manifold on $\tilde H_{\phi}$. The corresponding flat coordinates $t_A, A = 1, . . . ,N$ consist of the five parts
\begin{equation}
t_A=(t^{i,\alpha},i=0,..m,\alpha=1,..,n_i;p^i,q^i,i=1,..,m; r^i,s^i,i=1,..g)
\end{equation}
where
\begin{equation}
t^{i,\alpha}=\text{res}_{\infty_i} k_i^{-\alpha}pd\lambda             \quad  i=0,..m,\alpha=1,..,n_i
\end{equation}
\begin{equation}
p^{i}=v.p\int_{\infty_0}^{\infty_i}dp             \quad  i=1,..m.
\end{equation}
\begin{equation}
q^{i}=-\text{res}_{\infty_i} \lambda dp         \quad  i=1,..m.
\end{equation}
\begin{equation}
r^{i}=\int_{a_i}dp             \quad  i=1,..g.
\end{equation}
\begin{equation}
s^{i}=-\frac{1}{2\pi i}\int_{b_i} \lambda dp         \quad  i=1,..g.
\end{equation}
\end{theorem}
Moreover, function $\lambda= \lambda(p)$ is the superpotential of this Dubrovin Frobenius manifold, i.e. we have the following formulas to compute the metric $\eta=\langle,\rangle$, the intersection form $g^{*}=(,)$ and the structure constants $c$.
\begin{equation}
\langle\partial^{\prime},\partial^{\prime\prime}\rangle=-\sum \text{res}_{d\lambda=0}\frac{\partial^{\prime}(\lambda)\partial^{\prime\prime}(\lambda)}{d\lambda}dp
\end{equation}
\begin{equation}
(\partial^{\prime},\partial^{\prime\prime})=-\sum \text{res}_{dLog\lambda=0}\frac{\partial^{\prime}(Log\lambda)\partial^{\prime\prime}(Log\lambda)}{dLog\lambda}dp
\end{equation}
\begin{equation}\label{strucconst}
c(\partial^{\prime} ,\partial^{\prime\prime},\partial^{\prime\prime\prime})=-\sum \text{res}_{d\lambda=0}\frac{\partial^{\prime}(\lambda)\partial^{\prime\prime}(\lambda)\partial^{\prime\prime\prime}(\lambda)}{d\lambda}dp
\end{equation}

\begin{remark}
The Dubrovin Frobenius structure on Hurwitz spaces depend on a choice of suitable primary differentials. Dropping this suitable choice implies that we typically lose the quasi homogeneous condition of the WDVV equation or the fact the unit is covariant constant. 

\end{remark}

\section{Differential geometry of the orbit space of extended Jacobi group $A_n$}\label{section An case}

This section is dedicated to generalise the group $\Ja(\tilde A_1)$ defined in \cite{Guilherme 1} for arbitrary $n$, this new class of groups will be denoted by $\Ja(\tilde A_n)$. From the data of the group $\Ja(\tilde A_n)$, we will construct the Dubrovin Frobenius manifold in the orbit space of $\Ja(\tilde A_n)$. Furthermore, this Dubrovin Frobenius manifold will be locally isomorphic to the Hurwitz space $H_{1,n-1,0}$. In the section \ref{Thesis results}, there is a scheme of the technical steps that we should take to built the desired Dubrovin Frobenius manifold.

\subsection{The Group $\Ja(\tilde A_n)$}\label{The Group Jtilde An}

In this section, we define the group $\Ja(\tilde A_n)$. In order to understand the motivation of this group see \ref{Thesis results}.\\

Consider the $A_n$ in the following extended space 
\begin{equation*}
L^{\tilde A_n}=\{(z_0,z_1,..,z_n,z_{n+1})\in \mathbb{Z}^{n+2}:\sum _{i=0}^n
z_i=0     \}.
\end{equation*}
The action of  $A_n$ on $L^{\tilde A_n}$ is given  by
\begin{equation*} 
%\begin{split}
w(z_0,z_1,z_2,..,z_{n-1},z_n,z_{n+1})=(z_{i_0},z_{i_1},z_{i_2},..,z_{i_{n-1}},z_{i_n},z_{n+1})
%\end{split}
\end{equation*}
permutations in the first $n+1$ variables. Moreover, $A_n$ also acts on the complexfication of $L^{\tilde A_n}\otimes\mathbb{C}$.  Let the quadratic form  $\langle ,\rangle_{\tilde A_n}$ be given by 
\begin{equation}\label{definition of A matrix} 
\begin{split}
\langle v,v \rangle_{\tilde A_n}&=v^TM_{\tilde A_n}v\\
&=v^T \begin{pmatrix}
2 & 1 & 1&...1&1&0\\
1 & 2& 1&....1&1&0\\
1& 1&  2&....1&1&0\\
.& .&   . &...&.&0\\
.& .&   . &...&.&0\\
1& 1&   1 &...2&1&0\\
1& 1&   1 &....&2&0\\
0& 0&   0 &....&0&-n(n+1)\\
\end{pmatrix}v\\
&=\sum_{i=0}^{n-1}A_{ij}v_iv_j-n(n+1)v_{n+1}^2.
\end{split}
\end{equation}
Consider the following group $L^{\tilde A_n}\times L^{\tilde A_n}\times \mathbb{Z}$ with the following group operation
\begin{equation*} 
\begin{split}
&\forall (\lambda,\mu,k), (\tilde\lambda,\tilde\mu,\tilde k) \in L^{\tilde A_n}\times L^{\tilde A_n}\times \mathbb{Z}\\
&(\lambda,\mu,k)\bullet(\tilde\lambda,\tilde\mu,\tilde k)=(\lambda+\tilde\lambda,\mu+\tilde\mu,k+\tilde k+\langle\lambda,\tilde\lambda \rangle_{\tilde A_n})
\end{split}
\end{equation*}
Note that $\langle, \rangle_{\tilde A_n}$ is invariant under $A_n$ group, then $A_n$ acts on $L^{\tilde A_n}\times L^{\tilde A_n}\times \mathbb{Z}$. Hence, we can take the semidirect product $A_n\ltimes (L^{\tilde A_n}\times L^{\tilde A_n}\times \mathbb{Z}) $ given by the following product.
\begin{equation*} 
\begin{split}
&\forall (w,\lambda,\mu,k), (\tilde w,\tilde\lambda,\tilde\mu,\tilde k) \in A_n\times L^{\tilde A_n}\times L^{\tilde A_n}\times \mathbb{Z},\\
&(w,\lambda,\mu,k)\bullet(\tilde w,\tilde\lambda,\tilde\mu,\tilde k)=(w\tilde w,w\lambda+\tilde\lambda,w\mu+\tilde\mu,k+\tilde k+\langle\lambda,\tilde\lambda \rangle_{\tilde A_n}).
\end{split}
\end{equation*}
Denoting $W(\tilde A_n):=A_n\ltimes (L^{\tilde A_n}\times L^{\tilde A_n}\times \mathbb{Z})$, we can define 
\begin{definition}
The Jacobi group $\Ja(\tilde A_n)$ is defined as a semidirect product \\
$W(\tilde A_n)\rtimes SL_2(\mathbb{Z})$. The group action of $SL_2(\mathbb{Z})$ on $W(\tilde A_n)$ is defined as
\begin{equation*} 
\begin{split}
&Ad_{\gamma}(w)=w,\\
&Ad_{\gamma}(\lambda,\mu,k)=(a\mu-b\lambda,-c\mu+d\lambda,k+\frac{ac}{2}\langle\mu,\mu\rangle_{\tilde A_n}-bc\langle\mu,\lambda \rangle_{\tilde A_n}+\frac{bd}{2}\langle\lambda,\lambda \rangle_{\tilde A_n}),
\end{split}
\end{equation*}
for $(w,t=(\lambda,\mu,k))\in W(\tilde A_n), \gamma \in SL_2(\mathbb{Z})$. Then the multiplication rule is given as follows
\begin{equation*} 
\begin{split}
(w,t,\gamma)\bullet(\tilde w,\tilde t,\tilde \gamma)=(w\tilde w,t\bullet Ad_{\gamma}(w\tilde t),\gamma\tilde\gamma).
\end{split}
\end{equation*}
\end{definition}
Let us use the following identification $\mathbb{Z}^{n+1}\cong L^{\tilde A_n} , \mathbb{C}^{n+1}\cong L^{\tilde A_n}\otimes \mathbb{C}$ that is possible  due to maps
\begin{equation*}
\begin{split}
&(v_{0},..,v_{n-1},v_{n+1})\mapsto (v_{0},..,v_{n-1},-\sum_{i=0}^n v_i,v_{n+1}),\\
&(v_{0},..,v_{n-1},v_n,v_{n+1})\mapsto (v_{0},..,v_{n-1},v_{n+1}).\\
\end{split}
\end{equation*}
Then the action of Jacobi group $\Ja(\tilde A_n)$ on $\Omega:=\mathbb{C}\oplus \mathbb{C}^{n+1}\oplus \mathbb{H}$ is given as follows
\begin{proposition}\label{definition of the action of the group tildeAn}
The group $\Ja(\tilde A_n)\ni (w,t,\gamma)$ acts on $\Omega:=\mathbb{C}\oplus \mathbb{C}^{n+1}\oplus \mathbb{H} \ni (u,v,\tau)$ as follows
\begin{equation}\label{jacobigroupAntilde}
\begin{split}
&w(u,v,\tau)=(u,wv,\tau)\\
&t(u,v,\tau)=(u-\langle\lambda,v \rangle_{\tilde A_n}-\frac{1}{2}\langle\lambda,\lambda \rangle_{\tilde A_n}\tau,v+\lambda\tau+\mu,\tau)\\
&\gamma(u,v,\tau)=(\phi+\frac{c\langle v,v \rangle_{\tilde A_n}}{2(c\tau+d)},\frac{v}{c\tau+d},\frac{a\tau+b}{c\tau+d})\\
\end{split}
\end{equation}
\end{proposition}

The proof of the proposition \ref{definition of the action of the group tildeAn} follows from the proposition 1.1 of  \cite{Bertola M.1}  with small adaptation with the extra trivial action in the exceptional variable $v_{n+1}$ with respect the $A_n$ action.

\subsection{Jacobi forms of $\Ja(\tilde A_n)$}\label{Jacobi forms of Jantilde}
In order to understand the differential geometry of orbit space, first we need to study the algebra of the invariant functions. Informally, every time that there is a group $W$ acting on a vector space $V$, one could think  the orbit spaces $V/W$ as $V$, but you should remember yourself that it is only allowed to use the $W-$invariant sections of V. Hence, motivated by the definition of Jacobi forms of group $A_n$ defined in \cite{K. Wirthmuller}, and used in the context of Dubrovin-Frobenius manifold in \cite{Bertola M.1},\cite{Bertola M.2}, and generalising also the meromorphic Jacobi forms introduced in \cite{Guilherme 1} we give the following:

\begin{definition}\label{Jacobi forms definition jtildean abc}
The weak $\tilde A_n$ -invariant Jacobi forms of weight $k$, order $l$, and index $m$ are functions on $\Omega=\mathbb{C}\oplus \mathbb{C}^{n+1}\oplus\mathbb{H}\ni (u,v^{\prime},v_{n+1},\tau)=(u,v,\tau)$ which satisfy
\begin{equation}\label{jacobiform}
\begin{split}
&\varphi(w(u,v,\tau))=\varphi(u,v,\tau),\quad \textrm{$A_n$ invariant condition}\\ 
&\varphi(t(u,v,\tau))=\varphi(u,v,\tau)\\
&\varphi(\gamma(u,v,\tau))=(c\tau+d)^{-k}\varphi(u,v,\tau)\\
&E\varphi(u,v,\tau):=-\frac{1}{2\pi i}\frac{\partial}{\partial u}\varphi(u,v,\tau)=m\varphi(u,v,\tau)\\
\end{split}
\end{equation}
Moreover,
\begin{enumerate}
\item $\varphi$ is locally bounded functions on $v^{\prime}$ as $\Im(\tau)\mapsto +\infty$ (weak condition).
\item For fixed $u,v^{\prime},\tau$ the function $v_{n+1}\mapsto \varphi(u,v^{\prime},v_{n+1},\tau)$ is meromorphic with poles of order at most $l+2m$ on $v_{n+1}=\frac{j}{n}+\frac{l\tau}{n},0\leq l,j\leq n-1$ mod $\mathbb{Z}\oplus\tau\mathbb{Z}$.
\item For  fixed $u,v_{n+1}=\frac{j}{n}+\frac{l\tau}{n},0\leq l,j\leq n-1$ mod $\mathbb{Z}\oplus\tau\mathbb{Z} ,\tau$ the function $(i\neq n+1)$ $v_{i}\mapsto \varphi(u,v^{\prime},v_{n+1},\tau)$ is holomorphic.
\item For  fixed $u,v^{\prime},v_{n+1}=\frac{j}{n}+\frac{l\tau}{n},0\leq l,j\leq n-1$ mod $\mathbb{Z}\oplus\tau\mathbb{Z}$. the function  $\tau\mapsto \varphi(u,v^{\prime},v_{n+1},\tau)$ is holomorphic.
\end{enumerate}
The space of $\tilde A_n$-invariant Jacobi forms of weight $k$, order $l$, and index $m$ is denoted by $J_{k,l,m}^{\tilde A_n}$, and $J_{\bullet,\bullet,\bullet}^{\Ja(\tilde A_n)}=\bigoplus _{k,l,m}J_{k,l,m}^{\tilde A_n}$ is the space of Jacobi forms $\tilde A_n$ invariant.
\end{definition}

\begin{remark}\label{Remark Zagier Jtildean}
The condition $E\varphi(u,v^{\prime},v_{n+1},\tau)=m\varphi(u,v^{\prime},v_{n+1},\tau)$ implies that $\varphi(u,v^{\prime},v_{n+1},\tau)$ has the following form
\begin{equation*}
\varphi(u,v^{\prime},v_{n+1},\tau)=f(v^{\prime},v_{n+1},\tau)e^{2\pi imu}
\end{equation*}
See also remarks \ref{Remark Zagier Ja1}.
\end{remark}

The main result of section is the following.\\

\textbf{ The ring of $\tilde A_n$ invariant Jacobi forms  is polynomial over a suitable ring $E_{\bullet,\bullet}:=J_{\bullet,\bullet,0}^{\Ja(\tilde A_n)}$ on  suitable generators $\varphi_0,\varphi_1,\varphi_2,..\varphi_n$}.\\
 Before state precisely the theorem, I will define the objects $E_{\bullet,\bullet},\varphi_0,\varphi_1,\varphi_2,..\varphi_n$.\\

The ring $E_{\bullet,l}:=J_{\bullet,l,0}^{\Ja(\tilde A_n)}$ is the space of meromorphic Jacobi forms of index $0$ with poles of order at most $l$ on  $v_{n+1}=\frac{j}{n}+\frac{l\tau}{n},0\leq l,j\leq n-1$mod $\mathbb{Z}\oplus\tau\mathbb{Z}$, by definition. The sub-ring $J_{\bullet,0,0}^{\Ja(\tilde A_n)}\subset E_{\bullet,\bullet}$ has a nice structure, indeed:
\begin{lemma}\label{lemmamodular}
The sub-ring $J_{\bullet,0,0}^{\Ja(\tilde A_n)}$ is equal to $M_{\bullet}:=\bigoplus M_{k}$, where $M_k$ is the space of modular forms of weight $k$ for the full group $SL_2(\mathbb{Z})$.
\end{lemma}
\begin{proof}
Using the Remark \ref{Remark Zagier Jtildean}, we know that functions $\varphi(u,v^{\prime},v_{n+1},\tau)\in J_{\bullet,0,0}^{\Ja(\tilde A_n)}$  can not depend on $u$, then $\varphi(u,v^{\prime},v_{n+1},\tau)=\varphi(v^{\prime},v_{n+1},\tau)$.  Moreover, for fixed $v_{n+1},\tau$ the functions $v_i\mapsto \varphi(v^{\prime},v_{n+1},\tau)) $ are holomorphic elliptic function for any $i\neq n+1$. Therefore, by Liouville theorem, these function are constant in $v^{\prime}$. Similar argument shows that these function do not depend on $v_{n+1}$, because $l+2m=0$, i.e there is no pole. Then, $\varphi=\varphi(\tau)$ are standard holomorphic modular forms.
\end{proof}

\begin{lemma}\label{ringcoef}
If $\varphi \in E_{\bullet,\bullet}=J_{\bullet,\bullet,0}^{\Ja(\tilde A_n)}$, then $\varphi$ depends only on the variables $v_{n+1},\tau$. Moreover, if $\varphi\in J_{0,l,0}^{\Ja(\tilde A_n)}$ for fixed $\tau$ the function $\tau\mapsto \varphi(v_{n+1},\tau)$ is a elliptic function with poles of order at most $l$ $v_{n+1}=\frac{j}{n}+\frac{l\tau}{n},0\leq l,j\leq n-1$mod $\mathbb{Z}\oplus\tau\mathbb{Z}$.
\end{lemma}
\begin{proof}
The proof follows essentially in the same way of the lemma \ref{lemmamodular}, the only difference is that now we have poles on $v_{n+1}=\frac{j}{n}+\frac{l\tau}{n},0\leq l,j\leq n-1$mod $\mathbb{Z}\oplus\tau\mathbb{Z}$. Then, we have depedence in $v_{n+1}$.
\end{proof}
As a consequence of lemma \ref{ringcoef}, the function $\varphi \in E_{k,l}=J_{k,l,0}^{\Ja(\tilde A_n)}$ has the following form 
\begin{equation*}
\varphi(v_{n+1},\tau)=f(\tau)g(v_{n+1},\tau)
\end{equation*}
 where $f(\tau)$ is holomorphic modular form of weight $k$, and for fixed $\tau$, the function $v_{n+1}\mapsto g(v_{n+1},\tau)$ is an elliptic function of order at most $l$ on the poles $v_{n+1}=\frac{j}{n}+\frac{l\tau}{n},0\leq l,j\leq n-1$ mod $\mathbb{Z}\oplus\tau\mathbb{Z}$.\\

Note that a natural way to produce meromorphic Jacobi form is by using rational functions of holomorphic Jacobi forms. Starting from now, we will denote the Jacobi forms related with the Jacobi group $\Ja(A_{n+1})$ with the upper index $\Ja(A_{n+1})$, for instance
\begin{equation*}
\varphi^{\Ja(A_{n+1})},
\end{equation*}
and the Jacobi forms related with the Jacobi group $\Ja(\tilde A_n)$ with the with the upper index $\Ja(\tilde A_{n})$
\begin{equation*}
\varphi^{\Ja(\tilde A_n)}.
\end{equation*}

In \cite{Bertola M.1}, Bertola found basis  of the algebra of Jacobi form  by producing a holomorphic Jacobi form of type $A_n$ as product of theta functions. 
\begin{equation}
\varphi_{n+2}^{\Ja(A_{n+1})}=e^{-2\pi i u}\prod_{i=1}^{n+2} \frac{\theta_1(z_i,\tau)}{\theta_1^{\prime}(0,\tau)}.
\end{equation}
Afterwards, Bertola defined a recursive operator to produce the remaining basic generators. In order to recall the details see section \ref{Jacobi forms Bertola}. Our strategy will follow the same logic of Bertola method, we use theta functions to produce a basic generator and thereafter, we produce a recursive operator to produce the remaining part.

\begin{lemma}
 Let be $\varphi_{n+2}^{\Ja(A_{n+1})}(u_1,z_1,z_2,..,z_n,\tau)$  the holomorphic $A_{n+1}-invariant$ Jacobi forms which correspond to the algebra generator of maximal weight degree, in this case degree {n+2}. More explicitly,
\begin{equation}
\varphi_{n+2}^{\Ja(A_2)}=e^{-2\pi i u_1}\prod_{i=1}^{n+2}\frac{\theta_1(z_i,\tau)}{\theta_1^{\prime}(0,\tau)}.
\end{equation}
Let be $\varphi_2^{\Ja(A_1)}(u_2,z_{n+1},\tau)$ the holomorphic $A_1-invariant$ Jacobi form which correspond to the algebra generator of maximal weight degree, in this case degree 2.
\begin{equation}\label{definition of varphi 2 vn1}
\varphi_{2}^{\Ja(A_{1})}=e^{-2\pi i u_2}\frac{{\theta_1(z_{n+2},\tau)}^2}{{\theta_1^{\prime}(0,\tau)}^2}.
\end{equation}
Then, the function
\begin{equation}\label{varphi1 def Jtildean}
\varphi_n^{\Ja(\tilde A_n)}=\frac{\varphi_{n+2}^{\Ja(A_{n+1})}}{\varphi_2^{\Ja(A_1)}}
\end{equation}
is meromorphic Jacobi form of index 1, weight -n, order 0.
\end{lemma}
\begin{proof}
 For our convenience, we change the labels  $u_2-u_1,z_1, z_2,.. ..,z_{n+2} $ to 
 \begin{equation}
 \begin{split}
 u&=u_2-u_1,\\
z_1&=v_0-v_{n+1},\\
z_2&=v_1-v_{n+1},\\
 .\\
 .\\
z_{n+1}&=-\sum_{i=0}^{n} v_i-v_{n+1},\\
z_{n+2}&=(n+1)v_{n+1}.
 \end{split}
\end{equation}
 Then (\ref{varphi1 def Jtildean}) has the following form
\begin{equation}\label{varphi1 def2 Jtildean}
\varphi_n^{\Ja(\tilde A_n)}(u,v_0,v_1,..,v_{n+1},\tau)=e^{-2\pi i u}\frac{\prod_{i=0}^n\theta_1(v_i-v_{n+1},\tau)}{{\theta_1^{\prime}(0,\tau)}^n \theta_1((n+1)v_{n+1},\tau)}
\end{equation}

Let us prove each item separated.
\begin{enumerate}
\item \textbf{$A_n$ invariant}\\
The $A_n$ group acts on  (\ref{varphi1 def2 Jtildean}) by permuting its roots, thus  (\ref{varphi1 def2 Jtildean}) remains invariant under this operation.
\item \textbf{Translation invariant}\\
 Recall that under the translation $v\mapsto v+m+n\tau$, the Jacobi theta function transform as \cite{Bertola M.1}, \cite{E.T. Whittaker and G.N. Watson}:
\begin{equation}\label{transformtheta Jtildean}
\theta_1(v_i+\mu_i+\lambda_i\tau,\tau)=(-1)^{\lambda_i+\mu_i}e^{-2\pi i(\lambda_iv_i+\frac{\lambda_i^2}{2}\tau)}\theta_1(v_i,\tau).
\end{equation}
Then substituting the transformation  (\ref{transformtheta Jtildean}) into (\ref{varphi1 def2 Jtildean}), we conclude that (\ref{varphi1 def2 Jtildean}) remains invariant.
\item \textbf{$SL_2(\mathbb{Z})$ invariant}\\
Under $SL_2(\mathbb{Z})$ action  the following function transform as 
\begin{equation}\label{transformtheta2 Jtildean}
\frac{\theta_1(\frac{v_i}{c\tau+d},\frac{a\tau+d}{c\tau+d})}{\theta_1^{\prime}(0,\frac{a\tau+d}{c\tau+d})}=(c\tau+d)^{-1}\exp(\frac{\pi i cv_i^2 }{c\tau+d})\frac{\theta_1(v_i,\tau)}{\theta_1^{\prime}(0,\tau)}.
\end{equation}
Then, substituting (\ref{transformtheta2 Jtildean}) in (\ref{varphi1 def2 Jtildean}), we get
\begin{equation*}
\varphi_n^{\Ja(\tilde A_n)}\mapsto \frac{\varphi_n^{\Ja(\tilde A_n)}}{(c\tau+d)^n}
\end{equation*}
\item \textbf{Index $1$}\\
\begin{equation}
\begin{split}
\frac{1}{2\pi i}\frac{\partial}{\partial u}\varphi_n^{\Ja(\tilde A_n)}=\varphi_n^{\Ja(\tilde A_n)}.
\end{split}
\end{equation}
\item \textbf{Analytic behavior}\\
Note that $\varphi_n^{\Ja(\tilde A_n)}\theta_1^2((n+1)v_{n+1},\tau)$ is holomorphic function in all the variables $v_i$. Therefore $\varphi_n^{\Ja(\tilde A_n)}$ are holomorphic functions on the variables $\{v_i\}$, and meromorphic function in the variable $v_{n+1}$ with poles on $\frac{j}{n}+\frac{l\tau}{n}, j,l=0,1,..,n-1$ of order 2, i.e $l=0$, since $m=1$.
\end{enumerate}

\end{proof}
In order to define the desired recursive operator, it is necessary to enlarge the domain of the Jacobi forms from $\mathbb{C}\oplus\mathbb{C}^n\oplus\mathbb{H}\ni (u,v_0,v_1,..,v_{n+1},\tau)$ to $\mathbb{C}\oplus\mathbb{C}^{n+1}\oplus\mathbb{H}\ni (u,v_0,v_1,..,v_{n+1},p,\tau)$. In addition, we define lift a of Jacobi forms defined in $ \mathbb{C}\oplus\mathbb{C}^2\oplus\mathbb{H}$ to $\mathbb{C}\oplus\mathbb{C}^3\oplus\mathbb{H}$ as
\begin{equation*}
\varphi(u,v_0-v_{n+1},v_1-v_{n+1},..,(n+1)v_{n+1},\tau)\mapsto \hat\varphi(p):=\varphi(u,v_0-v_{n+1}+p,v_1-v_{n+1}+p,..,(n+1)v_{n+1}+p,\tau)
\end{equation*}
A convenient  way to do computation in these extended Jacobi forms is by using the following coordinates
\begin{equation}\label{coordinates uztau in jtildean}
\begin{split}
s&=u+ng_1(\tau)p^2,\\
z_1&=v_0-v_{n+1}+p,\\
z_2&=v_1-v_{n+1}+p,\\
.\\
.\\
z_{n+1}&=-\sum_{i=0}^n v_i-v_{n+1}+p,\\
z_{n+2}&=(n+1)v_{n+1}+p,\\
\tau&=\tau.
\end{split}
\end{equation}
The bilinear form $\langle v,v \rangle_{\tilde A_1}$ is extended to
\begin{equation}
\langle (z_1,z_2,..,z_{n+2}),(z_1,z_2,..,z_{n+2}) \rangle_E=\sum_{i=1}^{n+1}z_i^2-z_{n+2}^2,
\end{equation}

The action of the Jacobi group $\tilde A_n$ in this extended space is
\begin{equation}\label{jacobigroupA1tilde extended Jtildean}
\begin{split}
&\hat w_E(u,v,p,\tau)=\left(u,w(v),p,\tau\right)\\
&t_E(u,v,p,\tau)=\left(u-\langle \lambda,v \rangle_{E}-\frac{1}{2}\langle \lambda,\lambda \rangle_{E}\tau+k,v+p+\lambda\tau+\mu,\tau\right)\\
&\gamma_E(u,v,p,\tau)=\left(u+\frac{c\langle v,v \rangle_{E}}{2(c\tau+d)},\frac{v}{c\tau+d},\frac{p}{c\tau+d},\frac{a\tau+b}{c\tau+d}\right)\\
\end{split}
\end{equation}

\begin{proposition}
Let be $\varphi\in J_{k, m,\bullet}^{\Ja(\tilde A_n)}$, and $\hat\varphi$ the correspondent extended Jacobi form. Then,
\begin{equation}
\left.\frac{\partial}{\partial p}\left(\hat\varphi\right)\right |_{p=0} \in J_{k-1, m,\bullet}^{\Ja(\tilde A_n)}.
\end{equation}

\end{proposition}
\begin{proof}
\begin{enumerate}
\item\textbf{$A_n$-invariant}\\
The vector field $\frac{\partial}{\partial p}$ in coordinates $s,z_1,z_2,..,z_{n+2},\tau$ reads
\begin{equation*}
\frac{\partial}{\partial p}=\sum_{i=1}^{n+2}\frac{\partial}{\partial z_i}+2ng_1(\tau)p\frac{\partial}{\partial u}
\end{equation*}
Moreover, in the coordinates $s,z_1,z_2,.,z_{n+1},z_{n+2},\tau$  the $A_n$ group acts by permuting $z_1,z_2,...,z_{n+1}$ . Then
\begin{equation*}
\begin{split}
\left. \frac{\partial}{\partial p}\left(\varphi(s,z_2,z_1,z_3,\tau)\right)\right |_{p=0} &=\left(\sum_{i=1}^{n+2}\frac{\partial}{\partial z_i} \right)\left.\left(  \varphi(s,z_{i_1},z_{i_2}..,z_{i_{n+1}},z_{n+2},\tau)   \right)\right |_{p=0}\\
&=\left(\sum_{i=1}^{n+2}\frac{\partial}{\partial z_i}\right)  \left. \left(  \varphi(s,z_1,z_2,,..,z_{n+1},z_{n+2},\tau)   \right)\right |_{p=0}.
\end{split}
\end{equation*}
\item \textbf{Translation invariant}
\begin{equation*}
\begin{split}
&\left. \frac{\partial}{\partial p}\left(\varphi(u-\langle \lambda, v\rangle_E-\langle \lambda,\lambda \rangle_E,v+p+\lambda\tau+\mu,\tau)\right)\right |_{p=0}\\
&=\left .\frac{\partial}{\partial p} \langle \lambda, v\rangle_E \right |_{p=0}\varphi(u,v,\tau)+\frac{\partial \varphi}{\partial p}\left(u-\langle \lambda,v \rangle_{\tilde A_n}-\frac{1}{2}\langle \lambda,\lambda \rangle_{\tilde A_n}\tau+k,v+\lambda\tau+\mu,\tau\right)\\
&=\frac{\partial \varphi}{\partial p}\left(u-\langle \lambda,v \rangle_{\tilde A_n}-\frac{1}{2}\langle \lambda,\lambda \rangle_{\tilde A_n}\tau+k,v+\lambda\tau+\mu,\tau\right)\\
&=\left.\frac{\partial \varphi}{\partial p}(u,v,\tau)\right |_{p=0}.
\end{split}
\end{equation*}
\item \textbf{ $SL_2(\mathbb{Z})$ equivariant }
\begin{equation*}
\begin{split}
&\left. \frac{\partial}{\partial p}\left(\varphi(u+\frac{c\langle v,v \rangle_{E}}{2(c\tau+d)},\frac{v}{c\tau+d},\frac{p}{c\tau+d},\frac{a\tau+b}{c\tau+d}) \right)\right |_{p=0}\\
&=\left .\frac{c}{2(c\tau+d)}\frac{\partial}{\partial p} \langle v, v\rangle_E \right |_{p=0}\varphi(u,v,\tau)+\frac{1}{c\tau+d}\frac{\partial \varphi}{\partial p}\left (u+\frac{c\langle v,v \rangle_{\tilde A_n}}{2(c\tau+d)},\frac{v}{c\tau+d},\frac{p}{c\tau+d},\frac{a\tau+b}{c\tau+d}\right)\\
&=\frac{1}{c\tau+d}\frac{\partial \varphi}{\partial p}\left (u+\frac{c\langle v,v \rangle_{\tilde A_n}}{2(c\tau+d)},\frac{v}{c\tau+d},\frac{p}{c\tau+d},\frac{a\tau+b}{c\tau+d}\right)\\
&=\frac{1}{(c\tau+d)^k}\left.\frac{\partial \varphi}{\partial p}(u,v,\tau)\right |_{p=0}.
\end{split}
\end{equation*}
Then,
\begin{equation*}
\left.\frac{\partial \varphi}{\partial p}\left (u+\frac{c\langle v,v \rangle_{E}}{2(c\tau+d)},\frac{v}{c\tau+d},\frac{p}{c\tau+d},\frac{a\tau+b}{c\tau+d}\right)\right |_{p=0}=\frac{1}{(c\tau+d)^{k-1}}\left.\frac{\partial \varphi}{\partial p}(u,v,\tau)\right |_{p=0}.
\end{equation*}
\item \textbf{Index 1}
\begin{equation*}
\frac{1}{2\pi i}\frac{\partial}{\partial u}\frac{\partial}{\partial p}\hat\varphi=\frac{1}{2\pi i}\frac{\partial}{\partial p}\frac{\partial}{\partial u}\hat\varphi=\frac{\partial}{\partial p}\hat\varphi .
\end{equation*}
\end{enumerate}
\end{proof}

\begin{corollary}
The function
\begin{equation}
\begin{split}
&\left.\left[ e^{z\frac{\partial}{\partial p} } \left(e^{2\pi i u}\frac{\prod_{i=0}^n\theta_1(p+v_i-v_{n+1},\tau)}{{\theta_1^{\prime}(0,\tau)}^n \theta_1(p+(n+1)v_{n+1},\tau)} \right)  \right]\right |_{p=0}\\
&=\varphi_n^{\Ja(\tilde A_n)}+\varphi_{n-1}^{\Ja(\tilde A_n)}z+\varphi_{n-2}^{\Ja(\tilde A_n)}z^2+...+\varphi_0^{\Ja(\tilde A_n)}z^{n}+O(z^{n+1}),
\end{split}
\end{equation}
is a generating function for the Jacobi forms $\varphi_n^{\Ja(\tilde A_n)}, \varphi_{n-1}^{\Ja(\tilde A_n)},\varphi_{n-2}^{\Ja(\tilde A_n)},..,\varphi_0^{\Ja(\tilde A_n)}$, 
where 
\begin{equation}\label{recursive relation varphi Jtildean}
\varphi_{k}^{\Ja(\tilde A_n)}:=\left.\frac{\partial^{n-k} }{\partial p^{n-k}}\left(\hat\varphi_n^{\Ja(\tilde A_n)} \right)\right |_{p=0}.
\end{equation}
\end{corollary}
\begin{proof}
Acting $\frac{\partial}{\partial p}$ $k$ times in $\varphi_n^{\Ja(\tilde A_n)}$, we have
\begin{equation*}
\left.\left[ \frac{\partial^k}{\partial ^k p}\left(e^{2\pi i u}. \frac{\prod_{i=0}^n\theta_1(p+v_i-v_{n+1},\tau)}{{\theta_1^{\prime}(0,\tau)}^n \theta_1(p+(n+1)v_{n+1},\tau)}   \right)  \right]\right |_{p=0} \in J_{-n+k,1,\bullet}^{\Ja(\tilde A_n)}.
\end{equation*}

\end{proof}

\begin{corollary}
The generating function can be written as
\begin{equation}\label{semi superpotential}
\left.\left[ e^{z\frac{\partial}{\partial p} } \left(e^{2\pi i u} \frac{\prod_{i=0}^n\theta_1(p+v_i-v_{n+1},\tau)}{{\theta_1^{\prime}(0,\tau)}^n \theta_1(p+(n+1)v_{n+1},\tau)}\right)  \right]\right |_{p=0}=e^{-2\pi i(u+ ng_1(\tau)z^2)}\frac{\prod_{i=0}^n\theta_1(z+v_i-v_{n+1},\tau)}{{\theta_1^{\prime}(0,\tau)}^n \theta_1(z+(n+1)v_{n+1},\tau)}.
\end{equation}
\end{corollary}

\begin{proof}
\begin{equation}
\begin{split}
&\left.\left[ e^{z\frac{\partial}{\partial p} } \left(e^{2\pi i u}\frac{\prod_{i=0}^n\theta_1(p+v_i-v_{n+1},\tau)}{{\theta_1^{\prime}(0,\tau)}^n \theta_1(p+(n+1)v_{n+1},\tau)}\right)  \right]\right |_{p=0}=\\
&=\left.\left[ e^{z\frac{\partial}{\partial p} } \left(e^{2\pi i (s+ ng_1(\tau)p^2}\frac{\prod_{i=0}^n\theta_1(p+v_i-v_{n+1},\tau)}{{\theta_1^{\prime}(0,\tau)}^n \theta_1(p+(n+1)v_{n+1},\tau)}\right)  \right]\right |_{p=0}\\
&=e^{-2\pi i(u+ ng_1(\tau)z^2)}\frac{\prod_{i=0}^n\theta_1(z+v_i-v_{n+1},\tau)}{{\theta_1^{\prime}(0,\tau)}^n \theta_1(z+(n+1)v_{n+1},\tau)}.
\end{split}
\end{equation}

\end{proof}

The next lemma is one of the main points of this section, because this lemma identify the orbit space of the group $\Ja(\tilde A_n)$ with the Hurwitz space $H_{1,n-1,0}$. This relationship is possible due to the construction of the generating function of the Jacobi forms of type $\tilde A_n$, which can be completed to be the Landau-Ginzburg superpotential of $H_{1,n-1,0}$ as follows

\begin{lemma}\label{jacobiform1}
There is local biholomorphism between the orbit space $\Ja(\tilde A_n)$ and $H_{1,n-1,0}$, i.e  the space of elliptic functions with 1 pole of order $n$, and one simple pole.
\end{lemma}
\begin{proof}
The correspondence is realized by the map:
\begin{equation}\label{lambda0}
[(u,v_0,v_1,..,v_{n-1},v_{n+1},\tau)]\longleftrightarrow \lambda(v)=e^{-2\pi i u}\frac{\prod_{i=0}^{n}\theta_1(z-v_i,\tau)}{\theta_1^{n}(v,\tau)\theta_1(v+(n+1)v_{n+1},\tau)}
\end{equation}
Note that this map is well defined and one to one. Indeed:
\begin{enumerate}
\item \textbf{Well defined}\\
Note that proof that the map does not depend on the choice of the representant of $[(\phi,v_0,v_1,..,v_{n-1},v_{n+1},\tau)]$ is equivalent to prove that the function (\ref{lambda0}) is invariant under the action of $\Ja(\tilde A_n)$. Indeed 
\item \textbf{$A_n$ invariant}\\
The $A_n$ group acts on  (\ref{lambda0}) by permuting its roots, thus  (\ref{lambda0}) remais invariant under this operation.
\item \textbf{Translation invariant}\\
 Recall that under the translation $v\mapsto v+m+n\tau$, the Jacobi theta function transform as \cite{Bertola M.1}, \cite{E.T. Whittaker and G.N. Watson}:
\begin{equation}\label{transformtheta}
\theta_1(v_i+\mu_i+\lambda_i\tau,\tau)=(-1)^{\lambda_i+\mu_i}e^{-2\pi i(\lambda_iv_i+\frac{\lambda_i^2}{2}\tau)}\theta_1(v_i,\tau)
\end{equation}
Then substituting the transformation  (\ref{transformtheta}) into (\ref{lambda0}), we conclude that (\ref{lambda0}) remains invariant.
\item \textbf{$SL_2(\mathbb{Z})$ invariant}\\
Under $SL_2(\mathbb{Z})$ action  the following function transform as 
\begin{equation}\label{transformtheta2}
\frac{\theta_1(\frac{v_i}{c\tau+d},\frac{a\tau+d}{c\tau+d})}{\theta_1^{\prime}(0,\frac{a\tau+d}{c\tau+d})}=\exp(\frac{\pi i cv_i^2}{c\tau+d})\frac{\theta_1(v_i,\tau)}{\theta_1^{\prime}(0,\tau)}
\end{equation}
Then substituting the transformation  (\ref{transformtheta2}) into (\ref{lambda0}), we conclude that (\ref{lambda0}) remains invariant.
\item \textbf{Injectivity}\\
Two elliptic functions are equal if they have the same zeros and poles with multiplicity.
\item \textbf{Surjectivity}\\
 Any elliptic function can be written as rational functions of Weierstrass sigma function up to a multiplication factor \cite{E.T. Whittaker and G.N. Watson}. By using the formula to relate Weierstrass sigma function and Jacobi theta function
\begin{equation}
\sigma(v_i,\tau)=\frac{\theta_1(v_i,\tau)}{\theta_1^{\prime}(0,\tau)}\exp(-2\pi ig_1(\tau)v_i^2)
\end{equation}
\end{enumerate}

\end{proof}

\begin{corollary}\label{corollary superpotential antilde}
The functions $(\varphi_0^{\tilde A_n},\varphi_1^{\tilde A_n},..,\varphi_n^{\tilde A_n})$ obtained by the formula
\begin{equation}\label{superpotentialAn}
\begin{split}
\lambda^{\tilde A_n}&=e^{2\pi i u}\frac{\prod_{i=0}^n \theta_1(z-v_i+v_{n+1},\tau)}{\theta_1^{n}(z,\tau)\theta_1(z+(n+1)v_{n+1})}\\
&=\varphi_n^{\tilde A_n}\wp^{n-2}(z,\tau)+\varphi_{n-1}^{\tilde A_n}\wp^{n-3}(z,\tau)+...+\varphi_{2}^{\tilde A_n}\wp(z,\tau)\\
&+\varphi_{1}^{\tilde A_n}[\zeta(z,\tau)-\zeta(z+(n+1)v_{n+1},\tau)+\varphi_0^{\tilde A_n}
\end{split}
\end{equation}
 are Jacobi forms of weight $0,-1,-2,..,-n$ respectively, index 1, and order $0$.
\end{corollary}
\begin{proof}
Let us prove each item separated.
\begin{enumerate}
\item \textbf{$A_n$ invariant, translation invariant}\\
The l.h.s of (\ref{superpotentialAn}) are $A_n$ invariant, and translation invariant  by the lemma (\ref{jacobiform1}). Then, by the uniqueness of  Laurent expansion of $\lambda^{\tilde A_n}$, we have  that $\varphi_i^{\tilde A_n}$ are $A_n$ invariant, and translation invariant.
\item \textbf{$SL_2(\mathbb{Z})$ equivariant}\\
The l.h.s of (\ref{superpotentialAn}) are $SL_2(\mathbb{Z})$ invariant, but the Weierstrass functions of the r.h.s have the following transformation law
\begin{equation}
\begin{split}
\wp^{(k-2)}(\frac{z}{c\tau+d},\frac{a\tau+b}{c\tau+d})=(c\tau+d)^k\wp^{(k-2)}(z,\tau).
\end{split}
\end{equation}
Then, $\varphi_k^{\tilde A_n}$ must have the following transformation law
\begin{equation}
\begin{split}
\varphi_k^{\tilde A_n}(u+\frac{c\langle v,v \rangle_{\tilde A_n}}{2(c\tau+d)},\frac{v}{c\tau+d},\frac{a\tau+b}{c\tau+d})=(c\tau+d)^{-k}\varphi_k^{\tilde A_n}(u,v,\tau).
\end{split}
\end{equation}
\item \textbf{Index $1$}\\
\begin{equation}
\begin{split}
\frac{1}{2\pi i}\frac{\partial}{\partial u}\lambda^{\tilde A_n}=\lambda^{\tilde A_n}.
\end{split}
\end{equation}
Then
\begin{equation}
\begin{split}
\frac{1}{2\pi i}\frac{\partial}{\partial u}\varphi_i^{\tilde A_n}=\varphi_i^{\tilde A_n}.
\end{split}
\end{equation}
\item \textbf{Analytic behavior}\\
Note that $\lambda^{\tilde A_n}\theta_1^2((n+1)v_{n+1},\tau)$ is holomorphic function in all the variables $v_i$. Therefore $\varphi_i^{\tilde A_n}$ are holomorphic functions on the variables $v_0,v_1,..,v_{n-1}$, and meromorphic function in the variable $(n+1)v_{n+1}$ with poles on $\frac{j}{n}+\frac{l\tau}{n}, j,l=0,..,n-1$ of order 2, i.e $l=0$, since $m=1$ for all $\varphi_i^{\tilde A_n}$.

\end{enumerate}
\end{proof}

\subsection{Proof of the Chevalley theorem}\label{Proof of the Chevalley theorem}

At this stage, the principal theorem can be state in precise way as follows.

\begin{theorem}\label{chevalley}
The trigraded algebra of Jacobi forms $J_{\bullet,\bullet,\bullet}^{\Ja(\tilde A_{n})}=\bigoplus _{k,l,m}J_{k,l,m}^{\tilde A_{n}}$ is freely generated by $n+1$ fundamental Jacobi forms $(\varphi_0^{\tilde A_n },\varphi_1^{\tilde A_n },,\varphi_2^{\tilde A_n },..,,\varphi_n^{\tilde A_n })$ over the graded ring  $E_{\bullet,\bullet}$
\begin{equation}
J_{\bullet,\bullet,\bullet}^{\Ja(\tilde A_n)}=E_{\bullet,\bullet}\left[\varphi_0^{\tilde A_n },\varphi_1^{\tilde A_n },,\varphi_2^{\tilde A_n },..,,\varphi_n^{\tilde A_n }\right].
\end{equation}
\end{theorem}

Before proving this Theorem, some auxiliary lemmas will be necessary.\\

\begin{proposition}\cite{Bertola M.1}
The functions $(\varphi_0,\varphi_2,..,\varphi_{n+1})$ obtained by the formula
\begin{equation}\label{superpotentialAn Bertola}
\begin{split}
\lambda^{\Ja(A_n)}(z)&=e^{-2\pi i u}\frac{\prod_{i=0}^n \theta_1(z-v_i,\tau)}{\theta_1^{n+1}(z,\tau))}\\
&=\varphi_{n+1}^{\Ja(A_n)}\wp^{n-1}(z,\tau)+\varphi_{n}^{\Ja(A_n)}\wp^{n-2}(z,\tau)+...+\varphi_{2}^{\Ja(A_n)}\wp(z,\tau)+\varphi_0^{\Ja(A_n)},
\end{split}
\end{equation}
 are Jacobi forms of the Jacobi group $\Ja(A_n)$ of weight $0,-2,..,-n-1$ respectively, index 1.
\end{proposition}

\begin{proposition}\cite{Bertola M.1}
Let the Jacobi forms $\varphi_0,..,\varphi_{n+1}$ be defined by (\ref{superpotentialAn Bertola}), then the lowest term of the Taylor expansion in the variables $\{v_i\}$ are given by
\begin{equation}\label{relation Coxeter Bertola Jacobi forms}
\begin{split}
\varphi_{n+1}^{\Ja(A_n)}&=a_{n+1}+O(||v||^n),\\
\varphi_{n}^{\Ja(A_n)}&=a_{n}+O(||v||^{n-1})\\
\varphi_{n-1}^{\Ja(A_n)}&=a_{n-1}+O(||v||^{n-2})\\
&.\\
&.\\
\varphi_{2}^{\Ja(A_n)}&=a_{2}+O(||v||^{4})\\
\varphi_{1}^{\Ja(A_n)}&=a_{1}=0\\
\varphi_{0}^{\Ja(A_n)}&=a_{0}+O(||v||^{2})\\
\end{split}
\end{equation}
where $a_2,a_3,.,a_{n+1}$ are the elementary symmetric polynomials defined by the formula
\begin{equation}
\begin{split}
\lambda^{A_n}(p)&=\prod_{i=0}^n (p-x_i)\\
&=p^{n+1}+a_2p^{n-1}+a_3p^{n-2}+..+a_np+a_{n+1}.
\end{split}
\end{equation}

\end{proposition}

\begin{lemma}\label{principallema}
Let $\{\varphi_i^{\tilde A_n }\}$ be set of functions given by the formula (\ref{superpotentialAn})  ,and $\{\varphi_j^{A_{n+1} }\}$ given by (\ref{superpotentialAn Bertola}) , then
\begin{equation}\label{relationship between Jacobi forms}
\begin{split}
\varphi_{n+2}^{\Ja(A_{n+1})}&=\varphi_{n}^{\tilde A_{n}}\varphi_{2}^{\Ja(A_{1})},\\
\varphi_{n+1}^{\Ja(A_{n+1})}&=\varphi_{n-1}^{\Ja(\tilde A_{n})}\varphi_{2}^{\Ja(A_{1})}+a_{n-1}^{n}\varphi_{n}^{\Ja(\tilde A_{n})}\varphi_{2}^{\Ja(A_{1})},\\
\varphi_{n+2}^{\Ja(A_{n+1})}&=\varphi_{n-2}^{\Ja(\tilde A_{n})}\varphi_{2}^{\Ja(A_{1})}+a_{n-2}^{n-1}\varphi_{n-1}^{\Ja(\tilde A_{n})}\varphi_{2}^{\Ja(A_{1})}+a_{n-2}^n\varphi_{n}^{\Ja(\tilde A_{n})}\varphi_{2}^{\Ja(A_{1})},\\
.\\
.\\
\varphi_{2}^{\Ja(A_{n+1})}&=\varphi_{0}^{\Ja(\tilde A_{n})}\varphi_{2}^{\Ja(A_{1})}+\sum_{j=1}^{n}  a_0^j\varphi_{j}^{\Ja(\tilde A_{n})}\varphi_{2}^{\Ja(A_{1})},\\
\varphi_{0}^{\Ja(A_{n+1})}&=\sum_{j=0}^{n}  a_{-1}^j\varphi_{j}^{\Ja(\tilde A_{n})}\varphi_{2}^{\Ja(A_{1})}.\\
\end{split}
\end{equation}
where $\varphi_2^{\Ja(A_1)}$ is defined on (\ref{definition of varphi 2 vn1}) for $z_{n+2}=(n+1)v_{n+1}$, and $a_i^j=a_i^j(v_{n+1},\tau)$ are elliptic functions on $v_{n+1}$.

\end{lemma}
\begin{proof}
Note the following relation 

\begin{equation*}
\begin{split}
\frac{\lambda^{\Ja(A_{n+1})}}{\lambda^{\Ja(\tilde A_n)}}&=\frac{\theta_1(z-(n+1)v_{n+1},\tau)\theta_1(z+(n+1)v_{n+1}),\tau}{\theta_1^2(z,\tau)}e^{-2\pi iu_2}\\
&=\varphi_2^{\Ja(A_1)}\wp(z,\tau)-\varphi_2^{\Ja(A_1)}\wp((n+1)v_{n+1},\tau)
\end{split}
\end{equation*}
Hence,
\begin{equation}
\begin{split}
&\varphi_{n+2}^{\Ja(A_{n+1})}\wp^{n-2}(z,\tau)+\varphi_{n+1}^{\Ja(A_{n+1})}\wp^{n-3}(z,\tau)+...+\varphi_{2}^{\Ja(A_{n+1})}\wp(z,\tau)+\varphi_0^{\Ja(A_{n+1})}\\
&=(\varphi_n^{\Ja(\tilde A_n)}\wp^{n-2}(z,\tau)+\varphi_{n-1}^{\Ja(\tilde A_n)}\wp^{n-3}(z,\tau)+...+\varphi_{2}^{\Ja(\tilde A_n)}\wp(z,\tau)\\
&+\varphi_{1}^{\Ja(\tilde A_n)}[\zeta(z,\tau)-\zeta(z+(n+1)v_{n+1},\tau)+\varphi_0^{\Ja(\tilde A_n)})(\varphi_2^{\Ja(A_1)}\wp(z,\tau)-\varphi_2^{\Ja(A_1)}\wp((n+1)v_{n+1},\tau)).
\end{split}
\end{equation}
Then, the desired result is obtained by doing a Laurent expansion in the variable $z$ in both side of the equality.

\end{proof}

As a consequence of the previous lemma, we have 

\begin{corollary}
The Jacobi forms  $\{\varphi_i^{\Ja(\tilde A_n )}\}$  are algebraically independent.
\end{corollary}
\begin{proof}
Suppose that there exist  polynomial $h(x_0,x_1,..,x_n)$ not identically $0$, such that
\begin{equation*}
 h(\varphi_0^{\Ja(\tilde A_n) },\varphi_1^{\Ja(\tilde A_n) },\varphi_2^{\Ja(\tilde A_n) },..,\varphi_n^{\Ja(\tilde A_n) })=0
\end{equation*}
then, because $J^{\Ja(\tilde A_n)}_{\bullet,\bullet,\bullet}$ is graded ring $h(x_0,x_1,..,x_n)$ should be $0$ in each homogeneous component $h_m(x_0,x_1,..,x_n)$ of index $m$. Let $\tilde h_m:={(\varphi_2^{\Ja(A_1)})}^mh_m(\varphi_0^{\tilde \Ja(A_n) },\varphi_1^{\Ja(\tilde A_n) },\varphi_2^{\Ja(\tilde A_n) },..,\varphi_n^{\Ja(\tilde A_n) })$. Let us expand the functions $\varphi_i^{\Ja(\tilde A_n)}$ in the variables $v_i$, then $\tilde h_m$ vanishes iff its vanishes in each order of this expansion.

From  equations  (\ref{relation Coxeter Bertola Jacobi forms}), we know that the lowest term of the taylor expansion of $\varphi_{n+2}^{\Ja(A_{n+1})}$ are the elementary symmetric polynomials. Using lemma \ref{principallema}, we conclude that the lowest term of $\varphi_{j}^{\Ja(\tilde A_n)}$ is the same as the lowest term of $\varphi_{j+2}^{\Ja(A_{n+1})}$, but those terms are exactly the elementary symmetric polynomials. But the elementary symmetric polynomials are algebraically independent, then they can not solve any polynomial equation. Lemma proved.
\end{proof}

\begin{corollary}
\begin{equation*}
E_{\bullet,\bullet}\left[\varphi_0^{\Ja(\tilde A_n) },\varphi_1^{\Ja(\tilde A_n) },\varphi_2^{\Ja(\tilde A_n) },..,\varphi_n^{\Ja(\tilde A_n) }\right]=E_{\bullet,\bullet}\left[\frac{\varphi_0^{\Ja(A_{n+1})}}{\varphi_2^{\Ja(A_1)}},\frac{\varphi_2^{\Ja(A_{n+1})}}{\varphi_2^{\Ja(A_1)}},...,\frac{\varphi_n^{\Ja(A_{n+1})}}{\varphi_2^{\Ja(A_1)}}\right]
\end{equation*}
\end{corollary}

Moreover, we have the following lemma

\begin{lemma}
Let $\varphi\in J_{\bullet,\bullet,m}^{\Ja(\tilde A_n)}$, then $\varphi \in E_{\bullet,\bullet}\left[\frac{\varphi_0^{\Ja(A_{n+1})}}{\varphi_2^{\Ja(A_1)}},\frac{\varphi_2^{\Ja(A_{n+1})}}{\varphi_2^{\Ja(A_1)}},...,\frac{\varphi_n^{\Ja(A_{n+1})}}{\varphi_2^{\Ja(A_1)}}\right].$
\end{lemma}

\begin{proof}
Let $\varphi\in J_{\bullet,\bullet,m}^{\Ja(\tilde A_n)}$, then the function $\frac{\varphi}{\varphi_n^{\Ja(A_n)}}$ is an elliptic function on the variables $(v_0,v_1,..,v_{n-1},v_{n+1})$ with poles on $v_i-v_{n+1},(n+1)v_{n+1}$. Expanding the function $\frac{\varphi}{\varphi_n^{\Ja(A_n)}}$ in the variables $v_0,v_1,..,v_{n-1}$ we get 
\begin{equation*}
\begin{split}
\frac{\varphi}{\varphi_n^{\Ja(A_n)}}&=\sum_{i=0}^{n-1} a_m^i\wp^{(m-2)}(v_i-v_{n+1})+\sum_{i=0}^{n-1} a_{m-1}^i\wp^{(m-3)}(v_i-v_{n+1})+..\\
&+\sum_{i=0}^{n-1} a_1^i\zeta^{(m-2)}(v_i-v_{n+1})+b(v_{n+1},\tau)
\end{split}
\end{equation*}
But the function $\frac{\varphi}{\varphi_n^{\Ja(A_n)}}$ is invariant under the permutations of the variables $v_i$, then
\begin{equation}\label{equationtocomplete}
\begin{split}
\frac{\varphi}{\varphi_n^{\Ja(A_n)}}&=a_m\sum_{i=0}^{n-1} \wp^{(m-2)}(v_i-v_{n+1})+a_{m-1}\sum_{i=0}^{n-1} \wp^{(m-3)}(v_i-v_{n+1})+..\\
&+ a_1\sum_{i=0}^{n-1}\zeta^{(m-2)}(v_i-v_{n+1})+b(v_{n+1},\tau)
\end{split}
\end{equation}
Now we complete this function to $A_{n+1}$ invariant function by summing and subtracting the following function in e.q (\ref{equationtocomplete})
\begin{equation*}
\begin{split}
f(v_{n+1},\tau)&=a_m\sum_{i=0}^{n-1} \wp^{(m-2)}((n+1)v_{n+1})+a_{m-1}\sum_{i=0}^{n-1} \wp^{(m-3)}((n+1)v_{n+1})+..\\
&+ a_1\sum_{i=0}^{n-1}\zeta^{(m-2)}((n+1)v_{n+1})
\end{split}
\end{equation*}
Hence,
\begin{equation}
\begin{split}
\frac{\varphi}{\varphi_n^{\Ja(A_n)}}&=a_m(\sum_{i=0}^{n-1}\wp^{(m-2)}(v_i-v_{n+1})+\wp^{(m-2)}(n+1)v_{n+1}))\\
&+a_{m-1}\sum_{i=0}^{n-1} (\wp^{(m-3)}(v_i-v_{n+1})+\wp^{(m-3)}((n+1)v_{n+1}))+..\\
&+ a_1\sum_{i=0}^{n-1}(\zeta^{(m-2)}(v_i-v_{n+1})+\zeta^{(m-2)}((n+1)v_{n+1}))+\tilde b(v_{n+1},\tau)
\end{split}
\end{equation}
To finish the proof note the following
\begin{enumerate}
\item The functions $\varphi_{n+2}^{\Ja(A_{n+1})}[\wp^{(j)}(v_i-v_{n+1})+\wp^{(j)}(n+1)v_{n+1})]$ are $A_{n+1}$ by construction,
\item The functions $\varphi_{n+2}^{\Ja(A_{n+1})}[\wp^{(j)}(v_i-v_{n+1})+\wp^{(j)}(n+1)v_{n+1})]$ are invariant under the action of $(\mathbb{Z}\oplus\tau\mathbb{Z})^{2n+2}$, because $\varphi_{n+2}^{\Ja(A_{n+1})}$ invariant, and 
$\wp^{(j)}(v_i-v_{n+1})+\wp^{(j)}(n+1)v_{n+1})]$  are elliptic functions.
\item The functions $\varphi_{n+2}^{\Ja(A_{n+1})}[\wp^{(j)}(v_i-v_{n+1})+\wp^{(j)}(n+1)v_{n+1})]$ are equivariant under the action of $SL_2(\mathbb{Z})$, because $\varphi_{n+2}^{\Ja(A_{n+1})}$ is equivariant, and 
$\wp^{(j)}(v_i-v_{n+1})+\wp^{(j)}(n+1)v_{n+1})]$  are elliptic functions.
\item The function $\varphi_{n+2}^{\Ja(A_{n+1})}$ has zeros on $v_i-v_{n+1},(n+1)v_{n+1}$ of order $m$, and 
$\wp^{(j)}(v_i-v_{n+1})+\wp^{(j)}(n+1)v_{n+1})]$ has poles on $v_i-v_{n+1},(n+1)v_{n+1}$ of order $j+2\leq m$. Then, the functions $\varphi_{n+2}^{\Ja(A_{n+1})}[\wp^{(j)}(v_i-v_{n+1})+\wp^{(j)}(n+1)v_{n+1})]$ are holomorphic.
\item We conclude that $g_j:=\varphi_{n+2}^{\Ja(A_{n+1})}[\wp^{(j)}(v_i-v_{n+1})+\wp^{(j)}(n+1)v_{n+1})]\in J_{\bullet,\bullet}^{\Ja(A_{n+1})}$. Hence,
\begin{equation}
\begin{split}
\varphi=\sum_{i=1}^m a_i\frac{g_i}{{(\varphi_2^{\Ja(A_1)})}^m}+\tilde b(v_{n+1},\tau){(\frac{\varphi_{n+2}^{\Ja(A_{n+1})}}{\varphi_2^{\Ja(A_1)}})}^m\in E_{\bullet,\bullet}[\frac{\varphi_0^{\Ja(A_{n+1})}}{\varphi_2^{\Ja(A_1)}},\frac{\varphi_2^{\Ja(A_{n+1})}}{\varphi_2^{\Ja(A_1)}},...,\frac{\varphi_n^{\Ja(A_{n+1})}}{\varphi_2^{\Ja(A_1)}}].
\end{split}
\end{equation}
\end{enumerate}
\end{proof}

\begin{proof}
\begin{equation}
J_{\bullet,\bullet,\bullet}^{\Ja(\tilde A_n)}\subset E_{\bullet,\bullet}\left[\frac{\varphi_0^{\Ja(A_{n+1})}}{\varphi_2^{\Ja(A_1)}},\frac{\varphi_2^{\Ja(A_{n+1})}}{\varphi_2^{\Ja(A_1)}},...,\frac{\varphi_n^{\Ja(A_{n+1})}}{\varphi_2^{\Ja(A_1)}}\right]=E_{\bullet,\bullet}\left[\varphi_0^{\Ja(\tilde A_n) },\varphi_1^{\Ja(\tilde A_n) },\varphi_2^{\Ja(\tilde A_n) },..,\varphi_n^{\Ja(\tilde A_n )}\right]\subset J_{\bullet,\bullet,\bullet}^{\Ja(\tilde A_n)}.
\end{equation}

\end{proof}

\subsection{Intersection form}\label{Intersection form jantilde}

In the next sections, we introduce the minimal geometric data to reconstruct a Dubrovin Frobenius manifold structure as it was already announced in the subsection \ref{Thesis results}. The first step of this process is to build the invariant $\Ja(A_n)$ invariant sections : intersection form, Unit vector field and Euler vector field, and from this data, we reconstruct the Dubrovin Frobenius structure of the Hurwitz space $H_{1,n-1,0}$. In this section, we start by introducing the intersection form of the orbit space of the group $\Ja(\tilde A_n)$, which is defined to be a natural extensions of the intersection forms of the orbit space of the extended affine Weyl group $\tilde A_n$ \cite{B. Dubrovin and Y. Zhang}, and of the Jacobi group $\Ja(A_n)$ \cite{Bertola M.1}, \cite{Bertola M.2}. In this section, we prove that  the intersection form can be understood as an invariant section of the orbit space of $\Ja(\tilde A_n)$. Moreover , we define a new metric on the orbit space of $\Ja(\tilde A_n)$, which are closed related with the intersection form and it is denoted by $ \tilde M$, further, we construct a generating function of the coefficients for this new metric $ \tilde M$ in the coordinates $(\varphi_0,\varphi_1,..,\varphi_n,v_{n+1},\tau)$.

\begin{remark}
From this point, we will use $(\varphi_0,\varphi_1,..,\varphi_n)$ to denote the Jacobi forms of the group $\Ja(\tilde A_n)$ again, since there will not be anymore  Jacobi form from others Jacobi groups.
\end{remark}

\begin{definition}
Let
\begin{equation}\label{metrich1n}
\begin{split}
g&=\sum_{i=0}^n \left.dv_{i}^2\right|_{\sum_{i=0}^n v_i=0}-n(n+1)dv_{n+1}^2+2du d\tau,\\
&=\sum_{i,j=0}^{n-1} A_{ij}dv_idv_j-n(n+1)dv_{n+1}^2+2du d\tau,\\
&=\sum_{i,j=0}^{n+1} g_{ij}dv_idv_j+2du d\tau.\\
\end{split}
\end{equation}
where $A_{ij}$ is defined in (\ref{definition of A matrix} ).
The intersection form is given by 
\begin{equation}\label{metrich1n cotangent}
\begin{split}
g^{*}&=\sum_{i,j=0}^{n-1} A_{ij}^{-1}\frac{\partial}{\partial v_i}\otimes \frac{\partial}{\partial v_j}-\frac{1}{n(n+1)}\frac{\partial}{\partial v_{n+1}}\otimes \frac{\partial}{\partial v_{n+1}}+\frac{\partial}{\partial u}\otimes\frac{\partial}{\partial \tau}+\frac{\partial}{\partial \tau}\otimes\frac{\partial}{\partial u}.\\
\end{split}
\end{equation}
\end{definition}

\begin{proposition}\label{prop metric invariant}
The intersection form (\ref{metrich1n}) is invariant under the first two actions of (\ref{jacobigroupAntilde}), and behaves a modular form of weight $2$ under the last action of  (\ref{jacobigroupAntilde}).
\end{proposition}
\begin{proof}
The first action of (\ref{jacobigroupAntilde}) only acts on the variables $v_i$, the intersection form is invariant under this action because $A_{ij}dv_idv_j$ is invariant under permutation.\\
Under  the second action of (\ref{jacobigroupAntilde}), the differentials transform as:
\begin{equation}
\begin{split}
&dv_i\mapsto dv_i+\lambda_id\tau,\\
&du\mapsto du-\langle\lambda, \lambda  \rangle d\tau-2\sum_{i=0}^{n+1}g_{ij}\lambda_j dv_i,\\
&d\tau\mapsto d\tau.
\end{split}
\end{equation}
Hence:
\begin{equation}
\begin{split}
&g_{ij}dv_idv_j\mapsto g_{ij}dv_idv_j+2g_{ij}\lambda_idv_id\tau+g_{ij}\lambda_i\lambda_jd\tau^2,\\
&2du d\tau\mapsto 2du d\tau -\langle\lambda, \lambda  \rangle d^2\tau-2\sum_{i=0}^{n-1}g^{ij}\lambda_j dv_id\tau .\\
\end{split}
\end{equation}
Then:
\begin{equation}
\sum_{i,j=0}^{n+1} g_{ij}dv_idv_j+2du d\tau \mapsto \sum_{i,j=0}^{n+1} g_{ij}dv_idv_j+2du d\tau.
\end{equation}
Let us show that the metric has conformal transformation under the  third action of transformations (\ref{jacobigroupAntilde}):
\begin{equation}
\begin{split}
&dv_i\mapsto \frac{dv_i}{c\tau+d}-\frac{cv_id\tau}{(c\tau+d)^2},\\
&d\tau\mapsto \frac{d\tau}{(c\tau+d)^2},\\
&du\mapsto du+\frac{cg_{ij}v_idv_j}{c\tau+d}-\frac{c\langle v,v \rangle d\tau}{2(c\tau+d)^2}.\\
\end{split}
\end{equation}
Then:
\begin{equation}
\begin{split}
&g_{ij}dv_idv_j\mapsto \frac{g_{ij}dv_idv_j}{(c\tau+d)^2}-\frac{2cg_{ij}v_idv_jd\tau}{(c\tau+d)^3}+\frac{g_{ij}v_iv_jd\tau^2}{(c\tau+d)^4},\\
&2du d\tau \mapsto \frac{2du d\tau}{(c\tau+d)^2}+\frac{2cg_{ij}v_idv_jd\tau}{(c\tau+d)^3}-\frac{c\langle v,v \rangle d\tau^2}{(c\tau+d)^4}.\\
\end{split}
\end{equation}
Then,
\begin{equation}
\sum_{i,j=0}^{n+1} g_{ij}dv_idv_j+2du d\tau \mapsto \frac{\sum_{i,j=0}^{n+1} g_{ij}dv_idv_j+2du d\tau}{(c\tau+d)^2}.
\end{equation}
\end{proof}

An efficient way to compute all $g^{*}(d\varphi_i,d\varphi_j)$ is by collecting all of then in a generating function. Note that  $(d\varphi_i,d\varphi_j)$ is not a Jacobi form, and this fact makes the computation more difficult. Hence, in order to circle around this problem, we define the following coefficients in (\ref{definition of M coef}).

\begin{lemma}\label{intersection form and M}
Let $\varphi_i\in J^{A_n}_{k_i,m_i}$, then the metric given by
\begin{equation}\label{metrich1n m}
\begin{split}
\frac{1}{\eta^{2i+2j}}g^{*}(d\eta^{2i}\varphi_i, d\eta^{2j}\varphi_j) \frac{\partial}{\partial\varphi_i}\left(\eta^{2i}.  \right)\otimes\frac{\partial}{\partial\varphi_j}\left(\eta^{2j}.\right)
\end{split}
\end{equation}
is invariant under the first two actions of (\ref{jacobigroupAntilde}), and behaves a modular form of weight $2$ under the last action of  (\ref{jacobigroupAntilde}). \\
Moreover, the coeffiecients of the metric (\ref{metrich1n m})  
\begin{equation}\label{definition of M coef}
\begin{split}
M\left(d\varphi_i,d\varphi_j\right)&:=\frac{1}{\eta^{2i+2j}}g^{*}(d\eta^{2i}\varphi_i, d\eta^{2j}\varphi_j)\\
&=g^{*}\left(d\varphi_i,d\varphi_j\right)-4\pi i g_1(\tau)\left(k_im_j+k_jm_i\right)\varphi_i\varphi_j,
\end{split}
\end{equation}
belong to $ J^{A_n}_{k_i+k_j-2, m_i+m_j}$.
\end{lemma}
\begin{proof}
The metric (\ref{metrich1n m}) is is invariant under the first two actions of (\ref{jacobigroupAntilde}), because proposition \ref{prop metric invariant}, and because $\eta$ do not change under this action. The equivariance with respect the $SL_2(\mathbb{Z})$ follows again from proposition \ref{prop metric invariant}, and from the fact that the transformation laws of $\eta$ get canceled.\\

 The equation (\ref{definition of M coef}) follows from the chain rule, from the identity
\begin{equation}
\frac{\eta^{\prime}}{\eta}(\tau)=g_1(\tau)
\end{equation}

\end{proof}

\begin{proposition}\cite{Bertola M.1}
Let $E^{k}$ the space of elliptic function of weight $k$. The elliptic connection $D_{\tau}:E^K\mapsto E^k$ is linear map defined by
\begin{equation}\label{elliptic connection}
D_{\tau}F(v,\tau)=\partial_{\tau}F(v,\tau)-2kg_1(\tau)F(v,\tau)-\frac{1}{2\pi i}\frac{\theta_1^{\prime}(v,\tau)}{\theta_1(v,\tau)}F^{\prime}(v,\tau),
\end{equation}
where $F(v,\tau)\in E^{k}$.
\end{proposition}

In order to compute the coefficient of $M^{*}(d\varphi_i,d\varphi_j)$ it will be necessary to define an extended intersection form $g$.
\begin{definition}
The extended metric $g$ is defined by
\begin{equation}
\widetilde g=\sum_{i=0}^n \left.dv_{i}^2\right|_{\sum_{i=0}^n v_i=0}-n(n+1)dv_{n+1}^2+2du d\tau+ndp^2+4ng_1(\tau)pdpd\tau+2ng_1^{\prime}(\tau)p^2d\tau^2,
\end{equation}
or alternatively,
\begin{equation}
\widetilde g=\sum_{i=1}^{n+1} dz_{i}^2-dz_{n+2}^2+2ds d\tau,
\end{equation}
where $(s,z_1,...,z_{n+2},\tau)$ is given by (\ref{coordinates uztau in jtildean}).
The extended intersection form read as 
\begin{equation}
\begin{split}
\widetilde g^{*}=\sum_{i,j} A_{ij}^{-1}\frac{\partial}{\partial v_i}\otimes \frac{\partial}{\partial v_j}-\frac{1}{n(n+1)}\frac{\partial}{\partial v_{n+1}}\otimes \frac{\partial}{\partial v_{n+1}}+\frac{1}{n}\frac{\partial}{\partial p}\otimes\frac{\partial}{\partial p} +\frac{\partial}{\partial s}\otimes \frac{\partial}{\partial \tau}+\frac{\partial}{\partial \tau}\otimes\frac{\partial}{\partial s}
\end{split}
\end{equation}
\end{definition}

The following technical  result proved by Bertola in \cite{Bertola M.1} will be useful to prove the subsequent results.

\begin{proposition}\cite{Bertola M.1}
The following formula holds
\begin{equation}\label{nice technical proposition}
\left( \frac{\alpha^{\prime\prime}(x)}{\alpha(x)}+\frac{\alpha^{\prime\prime}(y)}{\alpha(y)}-\frac{\alpha^{\prime}(x)\alpha^{\prime}(y)}{\alpha(x)\alpha(y)}           \right)=-4\pi i\frac{\partial_{\tau}\alpha(x-y)}{\alpha(x-y)}+2\frac{\alpha^{\prime}(x-y)}{\alpha(x-y)}\left[ \frac{\alpha^{\prime}(x)}{\alpha(x)}-\frac{\alpha^{\prime}(y)}{\alpha(y)}\right].
\end{equation}
where $\alpha(p)$ is given by the second equation (\ref{equation inside lemma coefficient M}).
\end{proposition}

The desired generating function is a consequence of the following lemmas.

\begin{lemma}
Let $\Phi(p)$ be given by
\begin{equation}
\Phi(p)=e^{-2\pi i u-2\pi ip^2n g_1(\tau)}\prod_{i=1}^{n+1}\frac{\theta_1(z_i,\tau)}{\theta_1^{\prime}(0,\tau)}\frac{\theta_1^{\prime}(0,\tau)}{\theta_1(z_{n+2},\tau)},
\end{equation}

and $\widetilde M$ the extended modified intersection form
\begin{equation}
\widetilde M(d\Phi(p),d\Phi(p^{\prime}))=\frac{1}{\eta^{4n+4}}\widetilde g^{*}\left(d\left(\eta^{2n+2}\Phi(p)\right),d\left(\eta^{2n+2}\Phi(p^{\prime})\right)\right),
\end{equation}
then,
\begin{equation}
\begin{split}
&e^{2\pi in(p^2+{p^{\prime}}^2)}\widetilde M(d\Phi(p),d\Phi(p^{\prime}))=\\
&=2\pi i n\frac{\nabla_{\tau}\alpha(p-p^{\prime})}{\alpha(p-p^{\prime})}+\frac{\alpha^{\prime}(p-p^{\prime})}{\alpha(p-p^{\prime})}\left[ P(p)\frac{d P(p^{\prime})}{dp^{\prime}}-P(p^{\prime})\frac{dP(p)}{dp}\right],
\end{split}
\end{equation}
where
\begin{equation}\label{equation inside lemma coefficient M}
\begin{split}
\nabla_{\tau}F(v,\tau)&=\frac{1}{\eta^{2k}}\frac{\partial \left(\eta^{2k}F \right)}{\partial \tau}, \quad F\in E^{k},\\
\alpha(p)&=\frac{\theta_1(p,\tau)}{\theta_1^{\prime}(0,\tau)},\\
P(p)&=e^{-2\pi iu}\prod_{i=1}^{n+1}\frac{\theta_1(z_i,\tau)}{\theta_1^{\prime}(0,\tau)}\frac{\theta_1^{\prime}(0,\tau)}{\theta_1(z_{n+2},\tau)}.
\end{split}
\end{equation}

\end{lemma}
\begin{proof}

\begin{equation}\label{eq 1 of this lemma metric}
\begin{split}
&e^{2\pi ing_1(\tau)(p^2+{p^{\prime}}^2)}\widetilde M(d\Phi(p),d\Phi(p^{\prime}))\\
&=e^{2\pi ig_1(\tau)n(p^2+{p^{\prime}}^2)}\widetilde M\left(d\left(e^{-2\pi ing_1(\tau)p^2}P(p)\right),d\left(e^{-2\pi ing_1(\tau){p^{\prime}}^2}P(p^{\prime})\right)\right)
\end{split}
\end{equation}

Note that 
\begin{equation*}
\begin{split}
\frac{\partial}{\partial p}&=\sum_{i=1}^{n+2}\frac{\partial}{\partial z_i}+2npg_1(\tau)\frac{\partial}{\partial u},\\
\frac{\partial}{\partial v_{n+1}}&=\sum_{i=1}^{n+1}\frac{\partial}{\partial z_i}+(n+1)\frac{\partial}{\partial z_{n+2}},\\
\frac{\partial}{\partial \tau}&=\frac{\partial }{\partial \tilde \tau}+np^2g_1^{\prime}(\tau)\frac{\partial}{\partial u}.\\
\end{split}
\end{equation*}
Hence,
\begin{equation}\label{eq 2 of this lemma metric}
\begin{split}
&e^{2\pi ing_1(\tau)p^2}\frac{\partial}{\partial p}\left(e^{-2\pi ing_1(\tau)p^2} P(p) \right)=\sum_{i=1}^{n+1}\frac{\partial P(p)}{\partial z_i}-\frac{\partial P(p)}{\partial z_{n+2}},\\
&e^{2\pi ing_1(\tau)p^2}\frac{\partial}{\partial \tau}\left(e^{-2\pi ing_1(\tau)p^2} P(p) \right)=\frac{\partial P(p)}{\partial \tilde \tau}.\\
\end{split}
\end{equation}
Substituting (\ref{eq 2 of this lemma metric}) in (\ref{eq 1 of this lemma metric}) we get 
\begin{equation}\label{nice equation in the middle lemma}
\begin{split}
&e^{2\pi ing_1(\tau)(p^2+{p^{\prime}}^2)}\widetilde M(d\Phi(p),d\Phi(p^{\prime}))=\\
&=\sum_{i,j=0}^{n-1} A_{ij}^{-1}\frac{\partial P(p)}{\partial v_i} \frac{\partial P(p^{\prime})}{\partial v_j}+\frac{1}{n}\sum_{i=1}^{n+1}\frac{\partial P(p)}{\partial z_i}\sum_{i=1}^{n+1}\frac{\partial P(p^{\prime})}{\partial z_i}-\frac{1}{n(n+1)}\frac{\partial P(p)}{\partial v_{n+1}}\frac{\partial P(p^{\prime})}{\partial v_{n+1}}\\
&-\frac{1}{n}\frac{\partial P(p)}{\partial z_{n+2}}\sum_{i=1}^{n+1}\frac{\partial P(p^{\prime})}{\partial z_i}-\frac{1}{n}\frac{\partial P(p^{\prime})}{\partial z_{n+2}}\sum_{i=1}^{n+1}\frac{\partial P(p)}{\partial z_i}+\frac{1}{n}\frac{\partial P(p)}{\partial z_{n+2}}\frac{\partial P(p^{\prime})}{\partial z_{n+2}}\\
& -2\pi iP(p) \nabla_{\tau}P(p^{\prime})-2\pi iP(p^{\prime}) \nabla_{\tau}P(p)\\
&=\sum_{i=1}^{n+1}\frac{\partial P(p)}{\partial z_i}\frac{\partial P(p^{\prime})}{\partial z_i} -2\pi iP(p) \nabla_{\tau}P(p^{\prime})-2\pi iP(p^{\prime}) \nabla_{\tau}P(p).
\end{split}
\end{equation}
Using the following identity in (\ref{nice equation in the middle lemma})
\begin{equation}
\begin{split}
&\sum_{i,j=0}^{n-1} A_{ij}^{-1}\frac{\partial P(p)}{\partial v_i} \frac{\partial P(p^{\prime})}{\partial v_j}+\frac{1}{n}\sum_{i=1}^{n+1}\frac{\partial P(p)}{\partial z_i}\sum_{i=1}^{n+1}\frac{\partial P(p^{\prime})}{\partial z_i}\\
&=\sum_{i=1}^{n+1}\frac{\partial P(p)}{\partial z_i}\frac{\partial P(p^{\prime})}{\partial z_i}+\frac{1}{n(n+1)}\sum_{i=1}^{n+1}\frac{\partial P(p)}{\partial z_i}\sum_{i=1}^{n+1}\frac{\partial P(p^{\prime})}{\partial z_i},
\end{split}
\end{equation}
we get 
\begin{equation}\label{nice equation in the middle lemma1}
\begin{split}
&e^{2\pi ing_1(\tau)(p^2+{p^{\prime}}^2)}\widetilde M(d\Phi(p),d\Phi(p^{\prime}))=\\
&=\sum_{i=1}^{n+1}\frac{\partial P(p)}{\partial z_i}\frac{\partial P(p^{\prime})}{\partial z_i}-\frac{\partial P(p)}{\partial z_{n+2}}\frac{\partial P(p^{\prime})}{\partial z_{n+2}} -2\pi iP(p) \nabla_{\tau}P(p^{\prime})-2\pi iP(p^{\prime}) \nabla_{\tau}P(p),
\end{split}
\end{equation}

we now compute
\begin{equation}\label{equation in lemma to proof generating function of the metric}
\begin{split}
&\sum_{i=1}^{n+1}\frac{\partial P(p)}{\partial z_i}\frac{\partial P(p^{\prime})}{\partial z_i} -2\pi iP(p) \nabla_{\tau}P(p^{\prime})-2\pi iP(p^{\prime}) \nabla_{\tau}P(p)=\\
&=\sum_{i=1}^{n+1}\frac{\partial P(p)}{\partial z_i}\frac{\partial P(p^{\prime})}{\partial z_i}-\frac{2\pi i}{\eta^{2n}}\left[ P(p)\frac{\partial \left(\eta^{2n}P(p^{\prime})  \right)}{\partial \tau} +P(p^{\prime})\frac{\partial \left(\eta^{2n}P(p)  \right)}{\partial \tau}  \right]\\
&=-\frac{1}{2}\left( \sum_{i=1}^{n+1}    \frac{\alpha^{\prime\prime}(z_i)}{\alpha(z_i)} +\frac{\alpha^{\prime\prime}(w_i)}{\alpha(w_i)}-2\frac{\alpha^{\prime}(z_i)\alpha^{\prime}(w_i)}{\alpha(z_i)\alpha(w_i)}  \right)P(p)P(p^{\prime})+4\pi i g_1nP(p)P(p^{\prime})\\
&+\frac{1}{2}\left(    \frac{\alpha^{\prime\prime}(z_{n+2})}{\alpha(z_{n+2})} +\frac{\alpha^{\prime\prime}(w_{n+2})}{\alpha(w_{n+2})}-2\frac{\alpha^{\prime}(z_{n+2})\alpha^{\prime}(w_{n+2})}{\alpha(z_{n+2})\alpha(w_{n+2})}  \right)P(p)P(p^{\prime}),\\
\end{split}
\end{equation}
where $z_i:=z_i(v_i,p)$, $w_i:=z_i(v_i,p^{\prime})$. Substituting (\ref{nice technical proposition}) in (\ref{equation in lemma to proof generating function of the metric})
\begin{equation}
\begin{split}
&=-\frac{1}{2}\sum_{i=1}^{n+1}\left(    \frac{\alpha^{\prime\prime}(z_i)}{\alpha(z_i)} +\frac{\alpha^{\prime\prime}(w_i)}{\alpha(w_i)}-2\frac{\alpha^{\prime}(z_i)\alpha^{\prime}(w_i)}{\alpha(z_i)\alpha(w_i)}  \right)P(p)P(p^{\prime})+4\pi i g_1nP(p)P(p^{\prime})\\
&+\frac{1}{2}\left(    \frac{\alpha^{\prime\prime}(z_{n+2})}{\alpha(z_{n+2})} +\frac{\alpha^{\prime\prime}(w_{n+2})}{\alpha(w_{n+2})}-2\frac{\alpha^{\prime}(z_{n+2})\alpha^{\prime}(w_{n+2})}{\alpha(z_{n+2})\alpha(w_{n+2})}  \right)P(p)P(p^{\prime})\\
&=2\pi i n\frac{\nabla_{\tau}\alpha(p-p^{\prime})}{\alpha(p-p^{\prime})}+\frac{\alpha^{\prime}(p-p^{\prime})}{\alpha(p-p^{\prime})}\left[ P(p)\frac{d P(p^{\prime})}{dp^{\prime}}-P(p^{\prime})\frac{dP(p)}{dp}\right].
\end{split}
\end{equation}

\end{proof}

\begin{lemma}
For the coefficients $\tilde M$ of intersection form, we have the following formula 
\begin{equation}\label{pre generating function 2}
\begin{split}
&\sum_{k,j}C_k(p)C_j(p^{\prime})\widetilde M(d\varphi_{k},d\varphi_{j})=\\
&=2\pi i n\frac{\nabla_{\tau}\alpha(p-p^{\prime})}{\alpha(p-p^{\prime})}\lambda(p)\lambda(p^{\prime})+\frac{\alpha^{\prime}(p-p^{\prime})}{\alpha(p-p^{\prime})}\left[ P(p)\frac{d P(p^{\prime})}{dp^{\prime}}-P(p^{\prime})\frac{dP(p)}{dp}\right]\\
&-\sum_{k,j}\widetilde M\left( dC_k(p), dC_j(p^{\prime})\right)\varphi_{n+1-k}\varphi_{n+1-j}
\end{split}
\end{equation}
where $C_k(p)$ is given by
\begin{equation}
\begin{split}
&C_k(p)=\frac{(-1)^{k}}{(k-1)!}\alpha^{n}(p)\wp^{\left(k-2\right)}(p),\\
&\wp^{-1}(p)=\zeta(p)-\zeta(p+(n+1)v_{n+1})+(n+1)\zeta(v_{n+1}),\\
&\wp^{-2}(p)=1,
\end{split}
\end{equation}
and
\begin{equation}
\begin{split}
\widetilde M\left( dC_k(p), dC_j(p^{\prime})\right):=& -2\pi i\left( C_k(p) \nabla_{\tau}C_j(p^{\prime})+C_j(p^{\prime}) \nabla_{\tau}C_k(p)\right)+\frac{1}{n}\frac{\partial C_k(p)}{\partial p}\frac{\partial C_j(p^{\prime})}{\partial p^{\prime}}\\
&-\frac{1}{n(n+1)}\frac{\partial C_k(p)}{\partial v_{n+1}}\frac{\partial C_j(p^{\prime})}{\partial v_{n+1}}+\sum_{k,j=0}^n C_k(p)C_j(p^{\prime})\frac{\partial \varphi_j}{\partial v_{n+1}}\frac{\partial \varphi_k}{\partial v_{n+1}}.
\end{split}
\end{equation}
\end{lemma}

\begin{proof}
Note that $P(p)=\sum_{k=0}^{n+1}C_k(p)\varphi_{k}$, then
\begin{equation}
\begin{split}
&e^{2\pi in(p^2+{p^{\prime}}^2)}\widetilde M(d\Phi(p),d\Phi(p^{\prime}))=\\
&=\sum_{j,k} \widetilde M\left( d\left( C_k(p)\varphi_{k}   \right), d\left( C_j(p^{\prime})\varphi_{j}   \right)   \right)\\
&=\sum_{j,k} C_k(p)C_j(p^{\prime})\widetilde M\left( d \varphi_{k}  , d \varphi_{j}  \right)+\sum_{j,k} C_j(p^{\prime})\varphi_{k}\widetilde M \left( d C_k(p)  , d\varphi_{j}    \right)\\
&+\sum_{j,k}C_k(p)\varphi_{j}   \widetilde M\left( d \varphi_{k}, d C_j(p^{\prime}) \right)+\sum_{j,k} \varphi_{k}\varphi_{j}\widetilde M\left( d C_k(p)  , d C_j(p^{\prime})  \right)\\
&=\sum_{j,k} C_k(p)C_j(p^{\prime})\widetilde M\left( d \varphi_{k}  , d \varphi_{j}  \right)-2\pi i\sum_{j,k} C_j(p^{\prime})\varphi_{k} \varphi_{j} \nabla_{\tau}C_k(p)   \\
&-2\pi i\sum_{j,k}C_k(p)\varphi_{j}\varphi_{k}\nabla_{\tau}C_j(p^{\prime}) +\frac{1}{n}\sum_{j,k} \varphi_{k}\varphi_{j}\frac{\partial C_k(p)}{\partial p}\frac{\partial C_j(p^{\prime})}{\partial p^{\prime}}\\
&-\frac{1}{n(n+1)}\sum_{j,k} \varphi_{k}\varphi_{j}\frac{\partial C_k(p)}{\partial v_{n+1}}\frac{\partial C_j(p^{\prime})}{\partial v_{n+1}}-\frac{1}{n(n+1)}\sum_{j,k}  C_k(p)\varphi_{j}\frac{\partial \varphi_{k}}{\partial v_{n+1}}\frac{\partial C_j(p^{\prime})}{\partial v_{n+1}}\\
&-\frac{1}{n(n+1)}\sum_{j,k}  C_j(p^{\prime})\varphi_{k}\frac{\partial \varphi_{j}}{\partial v_{n+1}}\frac{\partial C_k(p)}{\partial v_{n+1}}.
\end{split}
\end{equation}
Then, isolating $\sum_{k,j}C_k(p)C_j(p^{\prime})\widetilde M(d\varphi_{k},d\varphi_{j})$
\begin{equation}\label{pre generating function 2}
\begin{split}
&\sum_{k,j}C_k(p)C_j(p^{\prime})\widetilde M(d\varphi_{k},d\varphi_{j})=\\
&=2\pi i n\frac{\nabla_{\tau}\alpha(p-p^{\prime})}{\alpha(p-p^{\prime})}\lambda(p)\lambda(p^{\prime})+\frac{\alpha^{\prime}(p-p^{\prime})}{\alpha(p-p^{\prime})}\left[ P(p)\frac{d P(p^{\prime})}{dp^{\prime}}-P(p^{\prime})\frac{dP(p)}{dp}\right]\\
&-\sum_{k,j}\widetilde M\left( dC_k(p), dC_j(p^{\prime})\right)\varphi_{n+1-k}\varphi_{n+1-j}.
\end{split}
\end{equation}

\end{proof}

\begin{theorem}
The coefficient of $M^{*}(d\varphi_i,d\varphi_j)$ is recovered by the generating formula
\begin{equation}\label{generating formula of Melliptic}
\begin{split}
&\sum_{k,j=0}^{n}\frac{(-1)^{k+j}}{(k-1)!(j-1)!}M^{*}(d\varphi_i,d\varphi_j)\wp(v)^{(k-2)}\wp(v^{\prime})^{(j-2)}=\\
&=2\pi i(\lambda(v^{\prime})D_{\tau}\lambda(v)+\lambda(v)D_{\tau}\lambda(v^{\prime}))-\frac{1}{n+1}\frac{d\lambda(v)}{dv}\frac{d\lambda(v^{\prime})}{dv^{\prime}}\\
&+\frac{1}{2}\frac{\wp^{\prime}(v)+\wp^{\prime}(v^{\prime})}{\wp(v)-\wp(v^{\prime})}[ \lambda(v)\frac{d\lambda(v^{\prime})}{dv^{\prime}}-\frac{d\lambda(v)}{dv}\lambda(v^{\prime})]-\frac{1}{n}\lambda^{\prime}(p)\lambda^{\prime}(p^{\prime})\\
&-\frac{1}{n(n+1)}\frac{\partial\lambda(p)}{\partial v_{n+1}}\frac{\partial\lambda(p^{\prime})}{\partial v_{n+1}}+\frac{1}{n(n+1)}\sum_{k,j=0}^n \frac{(-1)^{k+j}}{(k-1)!(j-1)!}\wp(v)^{(k-2)}\wp(v^{\prime})^{(j-2)}\frac{\partial \varphi_j}{\partial v_{n+1}}\frac{\partial \varphi_k}{\partial v_{n+1}}.
\end{split}
\end{equation}

\end{theorem}
\begin{proof}
We start by dividing the expression (\ref{pre generating function 2}) by $\alpha^n(p)\alpha^n(p^{\prime})$
\begin{equation}\label{pre generating function 2 sub}
\begin{split}
&\sum_{k,j=0}^{n}\frac{(-1)^{k+j}}{(k-1)!(j-1)!}M^{*}(d\varphi_i,d\varphi_j)\wp(v)^{(k-2)}\wp(v^{\prime})^{(j-2)}=\\
&=2\pi i n\frac{\nabla_{\tau}\alpha(p-p^{\prime})}{\alpha(p-p^{\prime})}\lambda(p)\lambda(p^{\prime})+\frac{\alpha^{\prime}(p-p^{\prime})}{\alpha(p-p^{\prime})}\left[ \lambda(p)\frac{d P(p^{\prime})}{dp^{\prime}}\frac{1}{\alpha^n(p^{\prime})}-\lambda(p^{\prime})\frac{dP(p)}{dp}\frac{1}{\alpha^n(p)}\right]\\
&-\sum_{k,j} \frac{\widetilde M\left( dC_k(p), dC_j(p^{\prime})\right)}{\alpha^n(p)\alpha^n(p^{\prime})}   \varphi_{n+1-k}\varphi_{n+1-j}\\
&=(1)-(2).
\end{split}
\end{equation}
Computing separately 
\begin{equation}
\begin{split}
&(1):=2\pi i n\frac{\nabla_{\tau}\alpha(p-p^{\prime})}{\alpha(p-p^{\prime})}\lambda(p)\lambda(p^{\prime})+\frac{\alpha^{\prime}(p-p^{\prime})}{\alpha(p-p^{\prime})}\left[ \lambda(p)\frac{d P(p^{\prime})}{dp^{\prime}}\frac{1}{\alpha^n(p^{\prime})}-\lambda(p^{\prime})\frac{dP(p)}{dp}\frac{1}{\alpha^n(p)}\right]=\\
&=n \left(2\pi i \frac{\nabla_{\tau}\alpha(p-p^{\prime})}{\alpha(p-p^{\prime})}-\frac{\alpha^{\prime}(p-p^{\prime} )}{\alpha(p-p^{\prime})}\left[\frac{\alpha^{\prime}(p)}{\alpha(p)}-\frac{\alpha^{\prime}(p^{\prime})}{\alpha(p^{\prime})}  \right]        \right)\lambda(p)\lambda(p^{\prime})\\
&+\frac{\alpha^{\prime}(p-p^{\prime})}{\alpha(p-p^{\prime})}\left[ \lambda(p)\frac{d \lambda(p^{\prime})}{dp^{\prime}}-\lambda(p^{\prime})\frac{d\lambda(p)}{dp}\right].
\end{split}
\end{equation}
and 
\begin{equation}
\begin{split}
&(2):=\sum_{k,j} \frac{\widetilde M\left( dC_k(p), dC_j(p^{\prime})\right)}{\alpha^n(p)\alpha^n(p^{\prime})}   \varphi_{n+1-k}\varphi_{n+1-j}=\\
&=-2\pi i\left( \lambda(p) \nabla_{\tau}\lambda(p^{\prime})+\lambda(p^{\prime}) \nabla_{\tau}\lambda(p)\right)-2\pi in\left[4g_1+\frac{\partial_{\tau}\alpha(p)}{\alpha(p)}+\frac{\partial_{\tau}\alpha(p)}{\alpha(p)} \right]\lambda(p)\lambda(p^{\prime})\\
&+\frac{1}{n}\left[\frac{\partial \lambda(p)}{\partial p}+n\frac{\alpha^{\prime}(p)}{\alpha(p)}\lambda(p)\right] \left[\frac{\partial \lambda(p^{\prime})}{\partial p^{\prime}}+n\frac{\alpha^{\prime}(p^{\prime})}{\alpha(p^{\prime})}\lambda(p^{\prime}) \right]\\
&-\frac{1}{n(n+1)}\frac{\partial \lambda(p)}{\partial v_{n+1}}\frac{\partial \lambda(p^{\prime})}{\partial v_{n+1}}+\sum_{k,j=0}^n \frac{(-1)^{k+j}}{(k-1)!(j-1)!}\wp(v)^{(k-2)}\wp(v^{\prime})^{(j-2)}\frac{\partial \varphi_j}{\partial v_{n+1}}\frac{\partial \varphi_k}{\partial v_{n+1}}\\
&=-2\pi i\left( \lambda(p) \nabla_{\tau}\lambda(p^{\prime})+\lambda(p^{\prime}) \nabla_{\tau}\lambda(p)\right)\\
&-n\left[8\pi ig_1+2\pi i\left(\frac{\partial_{\tau}\alpha(p)}{\alpha(p)}+\frac{\partial_{\tau}\alpha(p)}{\alpha(p)}\right)-\frac{\alpha^{\prime}(p)\alpha^{\prime}(p^{\prime})}{\alpha(p)\alpha(p^{\prime})}  \right]\lambda(p)\lambda(p^{\prime})\\
&+\frac{1}{n}\frac{\partial \lambda(p)}{\partial p}\frac{\partial \lambda(p^{\prime})}{\partial p^{\prime}}+\frac{\alpha^{\prime}(p)}{\alpha(p)}\lambda(p)\frac{\partial \lambda(p^{\prime})}{\partial p^{\prime}}+\frac{\alpha^{\prime}(p^{\prime})}{\alpha(p^{\prime})}\lambda(p^{\prime})\frac{\partial \lambda(p^{\prime})}{\partial p^{\prime}}\\
&-\frac{1}{n(n+1)}\frac{\partial \lambda(p)}{\partial v_{n+1}}\frac{\partial \lambda(p^{\prime})}{\partial v_{n+1}}+\sum_{k,j=0}^n \frac{(-1)^{k+j}}{(k-1)!(j-1)!}\wp(v)^{(k-2)}\wp(v^{\prime})^{(j-2)}\frac{\partial \varphi_j}{\partial v_{n+1}}\frac{\partial \varphi_k}{\partial v_{n+1}}\\
&=-2\pi i\left( \lambda(p) D_{\tau}\lambda(p^{\prime})+\lambda(p^{\prime}) D_{\tau}\lambda(p)\right)\\
&-n\left[8\pi ig_1+2\pi i\left(\frac{\partial_{\tau}\alpha(p)}{\alpha(p)}+\frac{\partial_{\tau}\alpha(p)}{\alpha(p)}\right)-\frac{\alpha^{\prime}(p)\alpha^{\prime}(p^{\prime})}{\alpha(p)\alpha(p^{\prime})}  \right]\lambda(p)\lambda(p^{\prime})\\
&+\frac{1}{n}\frac{\partial \lambda(p)}{\partial p}\frac{\partial \lambda(p^{\prime})}{\partial p^{\prime}}    +\left(\frac{\alpha^{\prime}(p)}{\alpha(p)}-\frac{\alpha^{\prime}(p^{\prime})}{\alpha(p^{\prime})}  \right)\left[ \lambda(p)\frac{d \lambda(p^{\prime})}{dp^{\prime}}-\lambda(p^{\prime})\frac{d\lambda(p)}{dp}\right]\\
&-\frac{1}{n(n+1)}\frac{\partial \lambda(p)}{\partial v_{n+1}}\frac{\partial \lambda(p^{\prime})}{\partial v_{n+1}}+\sum_{k,j=0}^n \frac{(-1)^{k+j}}{(k-1)!(j-1)!}\wp(v)^{(k-2)}\wp(v^{\prime})^{(j-2)}\frac{\partial \varphi_j}{\partial v_{n+1}}\frac{\partial \varphi_k}{\partial v_{n+1}}.\\
\end{split}
\end{equation}
Computing (1)-(2), and using equation (\ref{nice technical proposition}), we obtain
\begin{equation}
\begin{split}
&\sum_{k,j=0}^{n}\frac{(-1)^{k+j}}{(k-1)!(j-1)!}M^{*}(d\varphi_i,d\varphi_j)\wp(v)^{(k-2)}\wp(v^{\prime})^{(j-2)}=\\
&=2\pi i(\lambda(v^{\prime})D_{\tau}\lambda(v)+\lambda(v)D_{\tau}\lambda(v^{\prime}))-\frac{1}{n+1}\frac{d\lambda(v)}{dv}\frac{d\lambda(v^{\prime})}{dv^{\prime}}\\
&+\frac{1}{2}\frac{\wp^{\prime}(v)+\wp^{\prime}(v^{\prime})}{\wp(v)-\wp(v^{\prime})}[ \lambda(v)\frac{d\lambda(v^{\prime})}{dv^{\prime}}-\frac{d\lambda(v)}{dv}\lambda(v^{\prime})]-\frac{1}{n}\lambda^{\prime}(p)\lambda^{\prime}(p^{\prime})\\
&\frac{1}{n(n+1)}\frac{\partial\lambda(p)}{\partial v_{n+1}}\frac{\partial\lambda(p^{\prime})}{\partial v_{n+1}}-\frac{1}{n(n+1)}\sum_{k,j=0}^n \frac{(-1)^{k+j}}{(k-1)!(j-1)!}\wp(v)^{(k-2)}\wp(v^{\prime})^{(j-2)}\frac{\partial \varphi_j}{\partial v_{n+1}}\frac{\partial \varphi_k}{\partial v_{n+1}}.
\end{split}
\end{equation}

\end{proof}

\begin{corollary}
Let $\tilde \eta^{*}(d\varphi_i,d\varphi_j)$ and $\eta^{*}(d\varphi_i,d\varphi_j)$ be given by
\begin{equation}\label{metric eta def}
\begin{split}
\tilde \eta^{*}(d\varphi_i,d\varphi_j):=\frac{\partial M^{*}(d\varphi_i,d\varphi_j)}{\partial \varphi_0},\\
 \eta^{*}(d\varphi_i,d\varphi_j):=\frac{\partial g^{*}(d\varphi_i,d\varphi_j)}{\partial \varphi_0}.
\end{split}
\end{equation}
 The coefficient of $\tilde \eta^{*}(d\varphi_i,d\varphi_j)$ is recovered by the generating formula
\begin{equation}\label{generating formula of Mellipticeta}
\begin{split}
&\sum_{k,j=0}^{n+1}\frac{(-1)^{k+j}}{(k-1)!(j-1)!}\tilde \eta^{*}(d\varphi_i,d\varphi_j)\wp(v)^{(k-2)}\wp(v^{\prime})^{(j-2)}=\\
&=2\pi i(D_{\tau}\lambda(v)+D_{\tau}\lambda(v^{\prime}))+\frac{1}{2}\frac{\wp^{\prime}(v)+\wp^{\prime}(v^{\prime})}{\wp(v)-\wp(v^{\prime})}[ \frac{d\lambda(v^{\prime})}{dv^{\prime}}-\frac{d\lambda(v)}{dv}]\\
&-\frac{1}{n(n+1)}\frac{\partial}{\partial \varphi_0}\left(\frac{\partial\lambda(p)}{\partial v_{n+1}} \right)\frac{\partial\lambda(p^{\prime})}{\partial v_{n+1}}-\frac{1}{n(n+1)}\frac{\partial\lambda(p)}{\partial v_{n+1}}\frac{\partial}{\partial \varphi_0}\left(\frac{\partial\lambda(p^{\prime})}{\partial v_{n+1}} \right)\\
&+\frac{1}{n(n+1)}\sum_{k,j=0}^{n}\frac{(-1)^{k+j}}{(k-1)!(j-1)!}\frac{\partial}{\partial\varphi_0}\left(\frac{\partial\varphi_j}{\partial v_{n+1}}\right)\frac{\partial\varphi_k}{\partial v_{n+1}}\\
&+\frac{1}{n(n+1)}\sum_{k,j=0}^{n}\frac{(-1)^{k+j}}{(k-1)!(j-1)!}\frac{\partial\varphi_j}{\partial v_{n+1}}\frac{\partial}{\partial\varphi_0}\left(\frac{\partial\varphi_k}{\partial v_{n+1}}\right)
\end{split}
\end{equation}
Moreover,
\begin{equation}\label{metric eta consequence}
\begin{split}
\tilde \eta^{*}(d\varphi_i,d\varphi_j)= \eta^{*}(d\varphi_i,d\varphi_j), \quad i,j\neq 0,\\
\tilde \eta^{*}(d\varphi_0,d\varphi_j)= \eta^{*}(d\varphi_i,d\varphi_j)+4\pi i k_j\varphi_j.
\end{split}
\end{equation}

\end{corollary}
\begin{proof}
Just differentiate equation (\ref{generating formula of Melliptic}) with respect $\varphi_0$, and use the equation (\ref{superpotentialAn}).
\end{proof}

\begin{corollary}\label{important final corollary}
The metric $\tilde\eta^{*}$ and $\eta^{*}$ defined in  (\ref{metric eta def})is invariant under the second action of (\ref{jacobigroupAntilde}), furthermore, behave as modular form of weight 2  under the last action of  (\ref{jacobigroupAntilde}).
\end{corollary}
\begin{proof}
The metric $\tilde\eta^{*}$ and $\eta^{*}$ are given by 
\begin{equation}
\begin{split}
\tilde\eta^{*}&=Lie_{\frac{\partial}{\partial \varphi_0}}M\\
\eta^{*}&=Lie_{\frac{\partial}{\partial \varphi_0}}g^{*}
\end{split}
\end{equation}
The fact that $\frac{\partial}{\partial \varphi_0}, g^{*}, M^{*}$ is invariant under the second action of (\ref{jacobigroupAntilde}),furthermore, behave as modular form of weight 2  under the last action of  (\ref{jacobigroupAntilde}) give the desired result.

\end{proof}

\subsection{The second metric of the pencil}\label{The second metric of the pencil jantilde}
In the (\ref{metric eta def}), it was defined the second metric $\eta$, furthermore, it was derived a generating function for it. The main goal of this section is to extract the coefficients $\eta(d\varphi_i,d\varphi_j)$ from its generating function. In order to do this extraction, some auxiliaries lemmas are needed.

\begin{lemma}
Let $\varphi_1,\varphi_n$ be defined on (\ref{superpotentialAn}), then
\begin{equation}\label{exp of varphi1 varphin}
\begin{split}
\varphi_1&=e^{2\pi i u}\frac{\prod_{i=0}^n\theta_1(v_i+nv_{n+1})}{\theta_1^n((n+1)v_{n+1})\theta_1^{\prime}(0)},\\
\varphi_n&=e^{2\pi i u}\frac{\prod_{i=0}^n\theta_1(-v_i+v_{n+1})}{\theta_1(-(n+1)v_{n+1}){\theta_1^{\prime}(0)}^n}.
\end{split}
\end{equation}
\end{lemma}
\begin{proof}

\begin{equation*}
\begin{split}
\varphi_1&=\lim_{p\mapsto -(n+1)v_{n+1}} \left(p+(n+1)v_{n+1}\right)\lambda(p)=-e^{2\pi i u}\frac{\prod_{i=0}^n\theta_1^n(-v_i-nv_{n+1})}{\theta_1^n(-(n+1)v_{n+1})\theta_1^{\prime}(0)}=e^{2\pi i u}\frac{\prod_{i=0}^n\theta_1(v_i+nv_{n+1})}{\theta_1^n((n+1)v_{n+1})\theta_1^{\prime}(0)},\\
\varphi_n&=\lim_{p\mapsto 0} p^n\lambda(p)=e^{2\pi i u}\frac{\prod_{i=0}^n\theta_1(-v_i+v_{n+1})}{\theta_1(-(n+1)v_{n+1}){\theta_1^{\prime}(0)}^n}.
\end{split}
\end{equation*}

\end{proof}
\begin{lemma}\label{lemma of derivatives with respect vn1}
Let $\varphi_0,\varphi_1,\varphi_2,..,\varphi_n$ be defined on (\ref{superpotentialAn}), then
\begin{equation}\label{derivates of varphi0 wrt vn1}
\begin{split}
\frac{\partial}{\partial\varphi_0}\left ( \frac{\partial\varphi_i}{\partial v_{n+1}}              \right)&=0, \quad i>1,\\
\frac{\partial}{\partial\varphi_0}\left ( \frac{\partial\varphi_1}{\partial v_{n+1}}              \right)&=n,\\
\frac{\partial}{\partial\varphi_0}\left ( \frac{\partial\varphi_0}{\partial v_{n+1}}              \right)&=-n\frac{\theta_1^{\prime}((n+1)v_{n+1})}{\theta_1((n+1)v_{n+1})}.
\end{split}
\end{equation}

\end{lemma}
\begin{proof}
Computing $\frac{\partial \varphi_n}{\partial v_{n+1}}$ by using the first equation of (\ref{exp of varphi1 varphin})
\begin{equation}
\frac{\partial \varphi_{n}}{\partial v_{n+1}}=\left(\sum_{i=0}^n\frac{\theta_1^{\prime}(-v_i+v_{n+1})}{\theta_1(-v_i+v_{n+1})} - (n+1)\frac{\theta_1^{\prime}((n+1)v_{n+1})}{\theta_1((n+1)v_{n+1})}   \right)\varphi_n.
\end{equation}
Recall the recursive relation between $\{\varphi_i\}$ in (\ref{recursive relation varphi Jtildean})
\begin{equation}\label{recursive relation varphi}
\varphi_i=\left.\frac{\left.\partial^{\left(n-i\right)} \varphi_n(p)\right.}{\partial p^{n-i}}\right |_{p=0}.
\end{equation}
In particular,
\begin{equation}
\varphi_{n-1}=\left.\frac{\partial \varphi_{n}(p)}{\partial p}\right |_{p=0}=\left(\sum_{i=0}^n\frac{\theta_1^{\prime}(-v_i+v_{n+1})}{\theta_1(-v_i+v_{n+1})} -\frac{\theta_1^{\prime}((n+1)v_{n+1})}{\theta_1((n+1)v_{n+1})}   \right)\varphi_n,
\end{equation}
Then, 
\begin{equation}\label{derivative of varphij wrt vn1}
\frac{\partial \varphi_{n}}{\partial v_{n+1}}= \varphi_{n-1}-n\frac{\theta_1^{\prime}((n+1)v_{n+1})}{\theta_1((n+1)v_{n+1})} \varphi_n,
\end{equation}
consequently 
\begin{equation*}
\frac{\partial}{\partial \varphi_0}\left(\frac{\partial \varphi_{n}}{\partial v_{n+1}}\right)= 0.
\end{equation*}
Suppose that for $i>1$, we have
\begin{equation}
\varphi_i=f(-v_0+v_{n+1},-v_1+v_{n+1},..,(n+1)v_{n+1},\tau)\varphi_n,
\end{equation}
where $f(-v_0+v_{n+1},-v_1+v_{n+1},..,(n+1)v_{n+1},\tau)$ is an elliptic function on the variables  $v_0,v_1,..,v_{n+1}$ with zeros on $-v_i+v_{n+1}$ and poles on $(n+1)v_{n+1}$. Consider the extended $\varphi_i(p)$ as
\begin{equation*}
\varphi_i(p)=f(p-v_0+v_{n+1},p-v_1+v_{n+1},..,p+(n+1)v_{n+1},\tau)\varphi_n.
\end{equation*}
The action of  the vector fields $\frac{\partial}{\partial p}$ and $\frac{\partial}{\partial v_{n+1}}$ in $\varphi_i(p)$ and $\varphi_i$ are given by
\begin{equation*}
\begin{split}
\left.\frac{\partial \varphi_{i}(p)}{\partial p}\right |_{p=0}&=\left.\frac{\partial f}{\partial p}\right |_{p=0}\varphi_n+f\left.\frac{\partial \varphi_{n}(p)}{\partial p}\right |_{p=0},\\
\frac{\partial \varphi_{i}}{\partial v_{n+1}}&=\frac{\partial f}{\partial v_{n+1}}\varphi_n+f\frac{\partial \varphi_{n}}{\partial v_{n+1}}.\\
\end{split}
\end{equation*}
Note that  
\begin{equation*}
\begin{split}
\left.\frac{\partial f}{\partial p}\right |_{p=0}-\frac{\partial f}{\partial v_{n+1}}=b(v_{n+1},\tau),
\end{split}
\end{equation*}
where $b((n+1)v_{n+1},\tau)$ is an elliptic function on $(n+1)v_{n+1}$, because $\left.\frac{\partial f}{\partial p}\right |_{p=0}$ and $\frac{\partial f}{\partial v_{n+1}}$ are elliptic functions with the same Laurent tail in the variables $-v_i+v_{n+1}$. Hence, using equation (\ref{recursive relation varphi})
\begin{equation*}
\begin{split}
\frac{\partial \varphi_{i}}{\partial v_{n+1}}&=\left.\frac{\partial \varphi_{i}(p)}{\partial p}\right |_{p=0}+h(v_{n+1},\tau)\varphi_n,\\
&=\varphi_{i-1}+h(v_{n+1},\tau)\varphi_n.
\end{split}
\end{equation*}
Then,
\begin{equation*}
\begin{split}
\frac{\partial}{\partial\varphi_0}\left ( \frac{\partial\varphi_i}{\partial v_{n+1}}              \right)&=0, \quad i>1,\\
\end{split}
\end{equation*}

Computing $\varphi_0$ 
\begin{equation*}
\begin{split}
\varphi_0&=\left.\frac{\partial\varphi_1}{\partial p}\right |_{p=0}\\
&=\left[ \sum_{i=0}^{n}\frac{\theta_1^{\prime}(v_i+nv_{n+1})}{\theta_1(v_i+nv_{n+1})}  -   n\frac{\theta_1^{\prime}((n+1)v_{n+1})}{\theta_1((n+1)v_{n+1})}          \right]\varphi_1.
\end{split}
\end{equation*}
Computing $\frac{\partial\varphi_1}{\partial v_{n+1}}$ in terms of $\varphi_0$
\begin{equation}\label{derivative of varphi1 wrt vn1}
\begin{split}
\frac{\partial\varphi_1}{\partial v_{n+1}}&=\left[ \sum_{i=0}^{n}n\frac{\theta_1^{\prime}(v_i+nv_{n+1})}{\theta_1(v_i+nv_{n+1})}  - n(n+1)  \frac{\theta_1^{\prime}((n+1)v_{n+1})}{\theta_1((n+1)v_{n+1})}          \right]\varphi_1\\
&=n\varphi_0-n\frac{\theta_1((n+1)v_{n+1}}{\theta_1^{\prime}((n+1)v_{n+1}}\varphi_1.
\end{split}
\end{equation}
Then,
\begin{equation}
\frac{\partial}{\partial\varphi_0}\left ( \frac{\partial\varphi_1}{\partial v_{n+1}}              \right)=n.
\end{equation}

Computing $\frac{\partial}{\partial\varphi_0}\left ( \frac{\partial\varphi_0}{\partial v_{n+1}} \right)$
\begin{equation}\label{derivative of varphi0 wrt vn1}
\begin{split}
 \frac{\partial\varphi_0}{\partial v_{n+1}}&=\left[ n\sum_{i=0}^{n} \frac{\partial^2\log\theta_1(v_i+nv_{n+1})}{\partial v_{n+1}^2}-n(n+1)\frac{\partial^2\log\theta_1((n+1)v_{n+1})}{\partial v_{n+1}^2} \right]\varphi_1\\
& +\left[ \sum_{i=0}^{n}\frac{\theta_1^{\prime}(v_i+nv_{n+1})}{\theta_1(v_i+nv_{n+1})}  -  n \frac{\theta_1^{\prime}((n+1)v_{n+1})}{\theta_1((n+1)v_{n+1})}          \right]\frac{\partial\varphi_1}{\partial v_{n+1}}\\
&=\left[ n\sum_{i=0}^{n} \frac{\partial^2\log\theta_1(v_i+nv_{n+1})}{\partial v_{n+1}^2}-n(n+1)\frac{\partial^2\log\theta_1((n+1)v_{n+1})}{\partial v_{n+1}^2} \right]\varphi_1\\
& +\left[ \sum_{i=0}^{n}\frac{\theta_1^{\prime}(v_i+nv_{n+1})}{\theta_1(v_i+nv_{n+1})}  -   n\frac{\theta_1^{\prime}((n+1)v_{n+1})}{\theta_1((n+1)v_{n+1})}          \right]
\left[-n\varphi_0  -n \frac{\theta_1^{\prime}((n+1)v_{n+1})}{\theta_1((n+1)v_{n+1})} \varphi_1                   \right]
\\
&=\left[ n\sum_{i=0}^{n} \frac{\partial^2\log\theta_1(v_i+nv_{n+1})}{\partial v_{n+1}^2}-n(n+1)\frac{\partial^2\log\theta_1((n+1)v_{n+1})}{\partial v_{n+1}^2} \right]\varphi_1\\
& -n\left[ \sum_{i=0}^{n}\frac{\theta_1^{\prime}(v_i+nv_{n+1})}{\theta_1(v_i+nv_{n+1})}  -   \frac{\theta_1^{\prime}((n+1)v_{n+1})}{\theta_1((n+1)v_{n+1})}    \right]^2\varphi_1  
- n\frac{\theta_1^{\prime}((n+1)v_{n+1})}{\theta_1((n+1)v_{n+1})} \varphi_0             
\\
&=-\frac{\partial^2\log\theta_1((n+1)v_{n+1})}{\partial v_{n+1}^2}\varphi_1- n\frac{\theta_1^{\prime}((n+1)v_{n+1})}{\theta_1((n+1)v_{n+1})} \varphi_0+na\varphi_1.         
\\
\end{split}
\end{equation}
where $a$ is defined by
\begin{equation}\label{defition of the rest denoted by a}
\begin{split}
a=\sum_{i=0}^{n} \frac{\partial^2\log\theta_1(v_i+nv_{n+1})}{\partial v_{n+1}^2}-\frac{\partial^2\log\theta_1((n+1)v_{n+1})}{\partial v_{n+1}^2}-\left[ \sum_{i=0}^{n}\frac{\theta_1^{\prime}(v_i+nv_{n+1})}{\theta_1(v_i+nv_{n+1})}  -   \frac{\theta_1^{\prime}((n+1)v_{n+1})}{\theta_1((n+1)v_{n+1})}    \right]^2.
\end{split}
\end{equation}
Note that $na\varphi_1$ can not be proportional to $b\varphi_0$, for any $b=b(v_{n+1},\tau)$ elliptic function in the variable $v_{n+1}$. Indeed, if
\begin{equation*}
\begin{split}
na\varphi_1=b\left[ \sum_{i=0}^{n}\frac{\theta_1^{\prime}(v_i+nv_{n+1})}{\theta_1(v_i+nv_{n+1})}  -   n\frac{\theta_1^{\prime}((n+1)v_{n+1})}{\theta_1((n+1)v_{n+1})}          \right]\varphi_1,
\end{split}
\end{equation*}
we obtain,
\begin{equation}\label{last prop equation lemma}
\begin{split}
na=b\left[ \sum_{i=0}^{n}\frac{\theta_1^{\prime}(v_i+nv_{n+1})}{\theta_1(v_i+nv_{n+1})}  -   n\frac{\theta_1^{\prime}((n+1)v_{n+1})}{\theta_1((n+1)v_{n+1})}          \right].
\end{split}
\end{equation}
Analysing the Laurent tail in $v_i+nv_{n+1}$ of (\ref{defition of the rest denoted by a})
\begin{equation}\label{analyse of the Laurent tail of a}
\begin{split}
a=-2\sum_{i\leq j}\frac{\theta_1^{\prime}(v_i+nv_{n+1})}{\theta_1(v_i+nv_{n+1})} \frac{\theta_1^{\prime}(v_j+nv_{n+1})}{\theta_1(v_j+nv_{n+1})}  
  +2 \frac{\theta_1^{\prime}((n+1)v_{n+1})}{\theta_1((n+1)v_{n+1})}\sum_{i=0}^{n}\frac{\theta_1^{\prime}(v_i+nv_{n+1})}{\theta_1(v_i+nv_{n+1})}+\text{regular terms},   
\end{split}
\end{equation}
then, the first term of the equation (\ref{analyse of the Laurent tail of a}) implies that the left-hand side  and right hand side of the equation (\ref{last prop equation lemma}) have a different Laurent tail which is an absurd. Hence,
\begin{equation*}
\begin{split}
\frac{\partial}{\partial\varphi_0}\left ( \frac{\partial\varphi_0}{\partial v_{n+1}}              \right)&=-n\frac{\theta_1^{\prime}((n+1)v_{n+1})}{\theta_1((n+1)v_{n+1})}.
\end{split}
\end{equation*}

\end{proof}

\begin{corollary}\label{corollary vn1 eta}
Let $\varphi_0,\varphi_1,\varphi_2,..,\varphi_n$ be defined on (\ref{superpotentialAn}) and the metric $\eta^{*}$ defined in (\ref{metric eta def}), then
\begin{equation}
\begin{split}
\eta^{*}(d\varphi_i,dv_{n+1})&=0, \quad i>1,\\
\eta^{*}(d\varphi_1,dv_{n+1})&=-\frac{1}{n+1},\\
\eta^{*}(d\varphi_0,dv_{n+1})&=-\frac{1}{n+1}\frac{\theta_1^{\prime}((n+1)v_{n+1})}{\theta_1((n+1)v_{n+1})}.
\end{split}
\end{equation}
\end{corollary}

\begin{proof}
\begin{equation*}
\begin{split}
\eta^{*}(d\varphi_i,dv_{n+1})&=\frac{\partial}{\partial\varphi_0}\left ( \frac{\partial\varphi_i}{\partial v_{n+1}}              \right)=0, \quad i>1,\\
\eta^{*}(d\varphi_1,dv_{n+1})&=-\frac{1}{n(n+1)}\frac{\partial}{\partial\varphi_0}\left ( \frac{\partial\varphi_1}{\partial v_{n+1}}              \right)=-\frac{1}{n+1},\\
\eta^{*}(d\varphi_0,dv_{n+1})&=-\frac{1}{n(n+1)}\frac{\partial}{\partial\varphi_0}\left ( \frac{\partial\varphi_0}{\partial v_{n+1}}              \right)=-\frac{1}{n+1}\frac{\theta_1^{\prime}((n+1)v_{n+1})}{\theta_1((n+1)v_{n+1})}.
\end{split}
\end{equation*}

\end{proof}

\begin{lemma}\label{lemma tau varphi}
Let $\varphi_0,\varphi_1,\varphi_2,..,\varphi_n$ be defined on (\ref{superpotentialAn}) and the metric $\eta^{*}$ defined in (\ref{metric eta def}), then
\begin{equation}
\begin{split}
&\eta^{*}(d\varphi_i,d\tau)=0, \quad i\neq 0,\\
&\eta^{*}(d\varphi_0,d\tau)=-2\pi i.\\
\end{split}
\end{equation}
\end{lemma}
\begin{proof}

\begin{equation*}
\begin{split}
\eta^{*}(d\varphi_i,d\tau)=-2\pi i \frac{\partial\varphi_i}{\partial\varphi_0}    &=\delta_{i0}.\\
\end{split}
\end{equation*}

\end{proof}

\begin{theorem}\label{main lemma coefficients eta in varphi coordinates}
Let $\eta^{*}(d\varphi_i,d\varphi_j)$ be defined in (\ref{metric eta def}), then its coefficients can be obtained by the formula 
\begin{equation}\label{formula main theorem for eta varphi}
\begin{split}
\tilde\eta^{*}(d\varphi_i,d\varphi_j)&=(i+j-2)\varphi_{i+j-2}, \quad i,j\neq 0\\
\tilde\eta^{*}(d\varphi_i,d\varphi_0)&=0, \quad i\neq 0, \quad i\neq 1,\\
\tilde\eta^{*}(d\varphi_1,d\varphi_j)&=0. \quad j\neq 0,\\
\tilde\eta^{*}(d\varphi_1,d\varphi_0)&=\wp((n+1)v_{n+1})\varphi_1.\\
\end{split}
\end{equation}
\end{theorem}

\begin{proof}
We start by dividing the right hand side of the expression (\ref{generating formula of Mellipticeta}) in two parts:
\begin{equation}\label{divided exp generating formula of Mellipticeta }
\begin{split}
(a)&:=2\pi i(D_{\tau}\lambda(v)+D_{\tau}\lambda(v^{\prime}))+\frac{1}{2}\frac{\wp^{\prime}(v)+\wp^{\prime}(v^{\prime})}{\wp(v)-\wp(v^{\prime})}[ \frac{d\lambda(v^{\prime})}{dv^{\prime}}-\frac{d\lambda(v)}{dv}],\\
(b)&:=\frac{1}{n(n+1)}\frac{\partial}{\partial \varphi_0}\left(\frac{\partial\lambda(p)}{\partial v_{n+1}} \right)\frac{\partial\lambda(p^{\prime})}{\partial v_{n+1}}+\frac{1}{n(n+1)}\frac{\partial\lambda(p)}{\partial v_{n+1}}\frac{\partial}{\partial \varphi_0}\left(\frac{\partial\lambda(p^{\prime})}{\partial v_{n+1}} \right)\\
&-\frac{1}{n(n+1)}\sum_{k,j=0}^{n}\frac{(-1)^{k+j}}{(k-1)!(j-1)!}\frac{\partial}{\partial\varphi_0}\left(\frac{\partial\varphi_j}{\partial v_{n+1}}\right)\frac{\partial\varphi_k}{\partial v_{n+1}}\\
&-\frac{1}{n(n+1)}\sum_{k,j=0}^{n}\frac{(-1)^{k+j}}{(k-1)!(j-1)!}\frac{\partial\varphi_j}{\partial v_{n+1}}\frac{\partial}{\partial\varphi_0}\left(\frac{\partial\varphi_k}{\partial v_{n+1}}\right).
\end{split}
\end{equation}
Consider the equation (\ref{superpotentialAn}) written in a concise way as follows
\begin{equation}\label{generatorsAn+1 concise }
\lambda(v)=\sum_{k=0}^{n}\frac{(-1)^{k}}{(k-1)!}\varphi_{k}\wp^{k-2}(v).
\end{equation}
Substituting (\ref{generatorsAn+1 concise }) in the first equation of (\ref{divided exp generating formula of Mellipticeta })
\begin{equation}\label{generating formula of Mellipticeta mod}
\begin{split}
(a)&=\sum_{k=0}^{n}\frac{(-1)^{k}}{(k-1)!}\varphi_{k}\left[2\pi i\left(D_{\tau}\wp^{n-1-k}(v)+D_{\tau}\wp^{n-1-k}(v^{\prime}) \right)\right]\\
&+\sum_{k=0}^{n}\frac{(-1)^{k}}{(k-1)!}\varphi_{k}\left[\left(\zeta(v-v^{\prime})+\zeta(v^{\prime})-\zeta(v)\right)\left( \wp^{n-k}(v^{\prime})- \wp^{n-k}(v^{\prime})\right) \right].
\end{split}
\end{equation}
Dividing the expression once more 
\begin{equation}\label{generating formula of Mellipticeta mod divided}
\begin{split}
(a)_1&=\sum_{k=2}^{n}\frac{(-1)^{k}}{(k-1)!}\varphi_{k}\left[2\pi i\left(D_{\tau}\wp^{n-1-k}(v)+D_{\tau}\wp^{n-1-k}(v^{\prime}) \right)\right]\\
&+\sum_{k=2}^{n}\frac{(-1)^{k}}{(k-1)!}\varphi_{k}\left[\left(\zeta(v-v^{\prime})+\zeta(v^{\prime})-\zeta(v)\right)\left( \wp^{n-k}(v^{\prime})- \wp^{n-k}(v^{\prime})\right) \right],\\
(a)_2&:=\varphi_1\left[2\pi i\left(D_{\tau}\wp^{-1}(v)+D_{\tau}\wp^{-1}(v^{\prime}) \right)\right]\\
&+\varphi_1\left[\left(\zeta(v-v^{\prime})+\zeta(v^{\prime})-\zeta(v)\right)\left( \wp^{-1}(v^{\prime})- \wp^{-1}(v^{\prime})\right) \right].
\end{split}
\end{equation}
Expanding the left-hand side of (\ref{generating formula of Mellipticeta}), we get
\begin{equation}\label{generating formula of Mellipticeta mod lhs}
\begin{split}
&\sum_{k,j=0}^{n}\frac{(-1)^{k+j}}{(k-1)!(j-1)!}\frac{\partial M^{*}(d\varphi_i,d\varphi_j)}{\partial \varphi_0}\wp(v)^{(k-2)}\wp(v^{\prime})^{(j-2)}\\
&=\sum_{k,j=0}^{n}\frac{\partial M^{*}(d\varphi_i,d\varphi_j)}{\partial \varphi_0}\frac{1}{v^{k}(v^{\prime})^{j}}+\text{Other terms},
\end{split}
\end{equation}
where "Other terms" in the equation (\ref{generating formula of Mellipticeta mod lhs}) means positive powers of either $v$ or $v^{\prime}$.
For convenience, define 
\begin{equation}\label{ eq defined for convenience}
\begin{split}
(1):&=2\pi i\left(D_{\tau}\wp^{(k-2)}(v)+D_{\tau}\wp^{(j-2)}(v^{\prime}) \right)\\
&+\left(\zeta(v-v^{\prime})+\zeta(v^{\prime})-\zeta(v)\right)\left( \wp^{(j-1)}(v^{\prime})- \wp^{(k-1)}(v)\right)
\end{split}
\end{equation}
In order to better to compute (\ref{ eq defined for convenience}), consider the analytical behaviour  of the term 
\begin{equation}\label{elliptic connection for wp high order derivatives}
D_{\tau}\wp^{k-2}(v)=\partial_{\tau}\wp^{k-2}(v)-2(k-2)g_1(\tau)\wp^{k-2}(v)-\frac{1}{2\pi i}\frac{\theta_1^{\prime}(v,\tau)}{\theta_1(v,\tau)}\wp^{k-1}(v).
\end{equation}
The term $\partial_{\tau}\wp^{k-2}(v)$ in (\ref{elliptic connection for wp high order derivatives} ) is holomorphic, therefore, it does not contribute for the Laurent tail. The term 
\begin{equation}\label{ term 2 of elliptic connection of wp h}
2(k-2)g_1(\tau)\wp^{k-2}(v)
\end{equation}
 also do not contribute, because the full expression (\ref{elliptic connection for wp high order derivatives}) behaves as modular form under the $SL_2(\mathbb{Z})$, but (\ref{ term 2 of elliptic connection of wp h}) is clear a quasi-modular form, since it contains $g_1(\tau)$. Hence, (\ref{ term 2 of elliptic connection of wp h}) is canceled with the Laurent tail of 
\begin{equation}\label{ term 3 of elliptic connection of wp h}
\frac{1}{2\pi i}\frac{\theta_1^{\prime}(v,\tau)}{\theta_1(v,\tau)}\wp^{k-1}(v).
\end{equation}
 To sum up, the analytical behaviour avior of (\ref{elliptic connection for wp high order derivatives} ) is essentially given by (\ref{ term 3 of elliptic connection of wp h}), under this consideration, and by using the equation
\begin{equation}
\begin{split}
\zeta(v,\tau)&=\frac{\theta_1^{\prime}(v,\tau)}{\theta_1(v,\tau)}-4\pi ig_1(\tau)v\\
&=\frac{1}{v}+O(v^3),
\end{split}
\end{equation}
 the equation (\ref{ eq defined for convenience}) became 

\begin{equation}\label{ eq defined for convenience mod}
\begin{split}
(1)&= -\zeta(v)\wp^{k-1}(v)-\zeta(v^{\prime})\wp^{k-1}(v^{\prime})  \\
&+\left(\zeta(v-v^{\prime})+\zeta(v^{\prime})-\zeta(v)\right)\left( \wp^{k-1}(v^{\prime})- \wp^{k-1}(v)\right)+\text{Other terms}\\
&=\zeta(v-v^{\prime})\left( \wp^{k-1}(v^{\prime})- \wp^{k-1}(v^{\prime})\right)-\zeta(v)\wp^{k-1}(v^{\prime})-\zeta(v^{\prime})\wp^{k-1}(v)+\text{Other terms}\\
&=\frac{1}{v-v^{\prime}}\left( \frac{(-1)^{k}k!}{{v^{\prime}}^{k+1}}- \frac{(-1)^{k}k!}{v^{k+1}}\right)-\frac{1}{v}\frac{(-1)^{k}k!}{{v^{\prime}}^{k+1}}-\frac{1}{v^{\prime}}\frac{(-1)^{k}k!}{v^{k+1}}+\text{Other terms}\\
&=(-1)^{k}k!\left( \frac{1}{v-v^{\prime}}\frac{v^{k+1}-{v^{\prime}}^{k+1}   }{\left(v^{\prime}v\right)^{k+1}}\right)-\frac{1}{v}\frac{(-1)^{k}k!}{{v^{\prime}}^{k+1}}-\frac{1}{v^{\prime}}\frac{(-1)^{k}k!}{v^{k+1}}+\text{Other terms}\\
&=(-1)^{k}k!\left( \sum_{j=0}^{k} \frac{v^{k-j}{v^{\prime}}^j   }{\left(v^{\prime}v\right)^{k+1}}         \right)-\frac{1}{v}\frac{(-1)^{k}k!}{{v^{\prime}}^{k+1}}-\frac{1}{v^{\prime}}\frac{(-1)^{k}k!}{v^{k+1}}+\text{Other terms}\\
&=(-1)^{k}k!\left( \sum_{j=0}^{k} \frac{1  }{v^{1+j}{v^{\prime}}^{k+1-j}    }         \right)-\frac{1}{v}\frac{(-1)^{k}k!}{{v^{\prime}}^{k+1}}-\frac{1}{v^{\prime}}\frac{(-1)^{k}k!}{v^{k+1}}+\text{Other terms}\\
&=(-1)^{k}k!\left( \sum_{j=1}^{k} \frac{1  }{v^{1+j}{v^{\prime}}^{k+1-j}    }         \right)+\text{Other terms}\\
&=(-1)^{k}k!\left( \sum_{j=2}^{k+1} \frac{1  }{v^{j}{v^{\prime}}^{k+2-j}    }         \right)+\text{Other terms}.\\
\end{split}
\end{equation}
Substituting (\ref{ eq defined for convenience mod}) in right-hand side of (\ref{generating formula of Mellipticeta mod})
\begin{equation}\label{generating formula of Mellipticeta mod rhs}
\begin{split}
&\sum_{k=0}^{n}\frac{(-1)^{k}}{(k-1)!}\varphi_{k}\left[2\pi i\left(D_{\tau}\wp^{k-2}(v)+D_{\tau}\wp^{k-1}(v^{\prime}) \right)\right]\\
&+\sum_{k=0}^{n}\frac{(-1)^{k}}{(k-1)!}\varphi_{k}\left[\left(\zeta(v-v^{\prime})+\zeta(v^{\prime})-\zeta(v)\right)\left( \wp^{k-1}(v^{\prime})- \wp^{k-1}(v^{\prime})\right) \right]\\
&=\sum_{k=0}^{n}\sum_{j=2}^{k+1}\frac{\left(k\right)\varphi_{k}}{v^j {v^{\prime}}^{k+2-j}   }+\text{Other terms}\\
&=\sum_{k=0}^{n}\sum_{j=2}^{n+1}\frac{\left(k+j\right)\varphi_{k+j}}{v^j {v^{\prime}}^{k+2}   }+\text{Other terms}\\
&=\sum_{k=2}^{n+2}\sum_{j=2}^{n+1}\frac{\left(k+j-2\right)\varphi_{k+j-2}}{v^j {v^{\prime}}^{k}   }+\text{Other terms}.\\
\end{split}
\end{equation}

Computing the second expression of (\ref{generating formula of Mellipticeta mod divided})
\begin{equation}\label{equation of almost end 2}
\begin{split}
(a)_2&:=\varphi_1\left[2\pi i\left(D_{\tau}\wp^{-1}(v)+D_{\tau}\wp^{-1}(v^{\prime}) \right)\right]\\
&+\varphi_1\left[\left(\zeta(v-v^{\prime})+\zeta(v^{\prime})-\zeta(v)\right)\left( {\wp^{-1}}^{\prime}(v^{\prime})-{ \wp^{-1}}^{\prime}(v)\right) \right]\\
&=\varphi_1 \left [ \zeta(v)\left( -\wp(v)+\wp(v+(n+1)v_{n+1})  \right)+ \zeta(v^{\prime})\left( -\wp(v^{\prime})+\wp(v^{\prime}+(n+1)v_{n+1})  \right)\right]\\
&+\varphi_1\left(\zeta(v-v^{\prime})+\zeta(v^{\prime})-\zeta(v)\right)\left( \wp(v)-\wp(v+(n+1)v_{n+1})  \right) \\
&+\varphi_1\left(\zeta(v-v^{\prime})+\zeta(v^{\prime})-\zeta(v)\right)\left( -\wp(v^{\prime})+\wp(v^{\prime}+(n+1)v_{n+1})  \right)+\text{Other terms} \\
&=\varphi_1\zeta(v-v^{\prime})\left( \wp(v)-\wp(v+(n+1)v_{n+1}) -\wp(v^{\prime})+\wp(v^{\prime}+(n+1)v_{n+1}) \right) \\
&+ \varphi_1 \left [ \zeta(v)\left( -\wp(v^{\prime})+\wp(v^{\prime}+(n+1)v_{n+1})  \right)+ \zeta(v^{\prime})\left( -\wp(v)+\wp(v+(n+1)v_{n+1})  \right)\right]+\text{Other terms} \\\
&=\varphi_1\zeta(v-v^{\prime})\left( -\wp(v+(n+1)v_{n+1}) +\wp(v^{\prime}+(n+1)v_{n+1}) \right) \\
&+ \varphi_1 \left [ \zeta(v)\wp(v^{\prime}+(n+1)v_{n+1})+ \zeta(v^{\prime})\wp(v+(n+1)v_{n+1})   \right]+\widetilde{(a)_2}+\text{Other terms}, \\
\end{split}
\end{equation}
where $\widetilde{(a)_2}$ are terms that were already computed in (\ref{generating formula of Mellipticeta mod rhs}). To compute the second expression in (\ref{divided exp generating formula of Mellipticeta }), consider 
\begin{equation}\label{derivative of lambda with vn1}
\begin{split}
\frac{\partial \lambda(v)}{\partial v_{n+1}}&=\sum_{k=0}^{n}\frac{(-1)^k}{(k-1)!}\frac{\partial\varphi_k }{\partial v_{n+1}}\wp^{(k-2)}(v)+(n+1)\varphi_1\left[ \wp(v+(n+1)v_{n+1})-\wp((n+1)v_{n+1})  \right],\\
\end{split}
\end{equation}
and
\begin{equation}\label{derivative of lambda wiht varphi0}
\begin{split}
\frac{\partial}{\partial \varphi_0} \left( \frac{\partial \lambda(v)}{\partial v_{n+1}} \right )&=\frac{\partial}{\partial \varphi_0}\left(\frac{\partial\varphi_1 }{\partial v_{n+1}}\right) \left[ \zeta(v)-\zeta(v+(n+1)v_{n+1})+\zeta((n+1)v_{n+1})\right]  +\frac{\partial}{\partial \varphi_0}\left(\frac{\partial\varphi_0 }{\partial v_{n+1}}\right) \\
&=n \left[ \zeta(v)-\zeta(v+(n+1)v_{n+1})+\zeta((n+1)v_{n+1})\right]  -n\frac{\theta_1^{\prime}((n+1)v_{n+1})}{\theta_1((n+1)v_{n+1)}} \\
&=n \left[ \zeta(v)-\zeta(v+(n+1)v_{n+1})\right].  \\
\end{split}
\end{equation}
Substituting (\ref{derivative of lambda with vn1}) and (\ref{derivative of lambda wiht varphi0}) in the the second expression in (\ref{divided exp generating formula of Mellipticeta })
\begin{equation}\label{equation of the almost end 1}
\begin{split}
(b)&:=\frac{1}{n(n+1)}\frac{\partial}{\partial \varphi_0}\left(\frac{\partial\lambda(p)}{\partial v_{n+1}} \right)\frac{\partial\lambda(p^{\prime})}{\partial v_{n+1}}+\frac{1}{n(n+1)}\frac{\partial\lambda(p)}{\partial v_{n+1}}\frac{\partial}{\partial \varphi_0}\left(\frac{\partial\lambda(p^{\prime})}{\partial v_{n+1}} \right)\\
&-\frac{1}{n(n+1)}\sum_{k,j=0}^{n}\frac{(-1)^{k+j}}{(k-1)!(j-1)!}\frac{\partial}{\partial\varphi_0}\left(\frac{\partial\varphi_j}{\partial v_{n+1}}\right)\frac{\partial\varphi_k}{\partial v_{n+1}}\\
&-\frac{1}{n(n+1)}\sum_{k,j=0}^{n}\frac{(-1)^{k+j}}{(k-1)!(j-1)!}\frac{\partial\varphi_j}{\partial v_{n+1}}\frac{\partial}{\partial\varphi_0}\left(\frac{\partial\varphi_k}{\partial v_{n+1}}\right)\\
&=\varphi_1 \left[ \zeta(v^{\prime})-\zeta(v^{\prime}+(n+1)v_{n+1})\right]\left[ \wp(v+(n+1)v_{n+1})-\wp((n+1)v_{n+1})  \right]\\
&+\varphi_1 \left[ \zeta(v)-\zeta(v+(n+1)v_{n+1})\right]\left[ \wp(v^{\prime}+(n+1)v_{n+1})-\wp((n+1)v_{n+1})  \right].\\
\end{split}
\end{equation}
Subtracting  (\ref{equation of almost end 2}) and (\ref{equation of the almost end 1})
\begin{equation}\label{equation of the almost end 3}
\begin{split}
(a)_2-(b)&=-\varphi_1 \left[ \zeta(v^{\prime})-\zeta(v^{\prime}+(n+1)v_{n+1})\right]\left[ \wp(v+(n+1)v_{n+1})-\wp((n+1)v_{n+1})  \right]\\
&-\varphi_1 \left[ \zeta(v)-\zeta(v+(n+1)v_{n+1})\right]\left[ \wp(v^{\prime}+(n+1)v_{n+1})-\wp((n+1)v_{n+1})  \right]\\
&+\varphi_1\zeta(v-v^{\prime})\left( -\wp(v+(n+1)v_{n+1}) +\wp(v^{\prime}+(n+1)v_{n+1}) \right) \\
&+ \varphi_1 \left [ \zeta(v)\wp(v^{\prime}+(n+1)v_{n+1})+ \zeta(v^{\prime})\wp(v+(n+1)v_{n+1})   \right]+\widetilde{(a)_2}+\text{Other terms} \\
&=\varphi_1 \wp((n+1)v_{n+1}) \left[ \zeta(v^{\prime})-\zeta(v^{\prime}+(n+1)v_{n+1})\right]+\varphi_1\wp((n+1)v_{n+1}) \left[ \zeta(v)-\zeta(v+(n+1)v_{n+1})\right]\\
&+\varphi_1\left[ \zeta(v-v^{\prime})     \right]\left( \wp(v+(n+1)v_{n+1}) +\wp(v^{\prime}+(n+1)v_{n+1}) \right)\\
&+\varphi_1\left[ \zeta(v^{\prime}+(n+1)v_{n+1})     \right]\left( \wp(v+(n+1)v_{n+1}) \right)\\
&+\varphi_1\left[ \zeta(v+(n+1)v_{n+1})     \right]\left( \wp(v^{\prime}+(n+1)v_{n+1}) \right)\\
 &+\widetilde{(a)_2}+\text{Other terms} \\
 &=\varphi_1 \wp((n+1)v_{n+1}) \left[ \zeta(v^{\prime})-\zeta(v^{\prime}+(n+1)v_{n+1})\right]+\varphi_1\wp((n+1)v_{n+1}) \left[ \zeta(v)-\zeta(v+(n+1)v_{n+1})\right]\\
 &+\widetilde{(a)_2}+\text{Other terms}. \\
\end{split}
\end{equation}
Summing (\ref{equation of the almost end 3}) and (\ref{generating formula of Mellipticeta mod rhs}), we have
\begin{equation*}
\begin{split}
&\sum_{k,j=0}^{n}\frac{(-1)^{k+j}}{(k-1)!(j-1)!}\frac{\partial M^{*}(d\varphi_i,d\varphi_j)}{\partial \varphi_0}\wp(v)^{(k-2)}\wp(v^{\prime})^{(j-2)}\\
&=\sum_{k,j=0}^{n}\frac{\partial M^{*}(d\varphi_i,d\varphi_j)}{\partial \varphi_0}\frac{1}{v^{k}(v^{\prime})^{j}}+\text{Other terms},\\
&=\sum_{k=2}^{n+2}\sum_{j=2}^{n+1}\frac{\left(k+j-2\right)\varphi_{k+j-2}}{v^j {v^{\prime}}^{k}   }\\
&+\varphi_1 \wp((n+1)v_{n+1}) \left[ \zeta(v^{\prime})-\zeta(v^{\prime}+(n+1)v_{n+1})\right]\\
&+\varphi_1\wp((n+1)v_{n+1}) \left[ \zeta(v)-\zeta(v+(n+1)v_{n+1})\right]\\
&+\text{Other terms}.\\
\end{split}
\end{equation*}

Hence, we get the desired result.

\end{proof}

\begin{corollary}\label{main corollary eta}
Let $\eta^{*}(d\varphi_i,d\varphi_j)$ be defined in (\ref{metric eta def}), then its coefficients can be obtained by the formula 
\begin{equation}\label{formula main theorem for eta varphi sem tilde}
\begin{split}
\eta^{*}(d\varphi_i,d\varphi_j)&=(i+j-2)\varphi_{i+j-2}, \quad i,j\neq 0\\
\eta^{*}(d\varphi_i,d\varphi_0)&=0, \quad i\neq 0, \quad i\neq 1,\\
\eta^{*}(d\varphi_1,d\varphi_j)&=0. \quad j\neq 0,\\
\eta^{*}(d\varphi_1,d\varphi_0)&=\frac{\partial^2\log(\theta_1(n+1)v_{n+1})}{\partial v_{n+1}^2}\varphi_1.\\
\end{split}
\end{equation}
\end{corollary}

\subsection{Flat coordinates of $\eta$}\label{Flat coordinates of Jantilde}

This section is dedicated to construct the flat coordinates of $\eta$ and its relationship with the invariant coordinates $\varphi_0,\varphi_1,..,\varphi_n,v_{n+1},\tau$. Our strategy will be an adaptation of the work done in \cite{Saito}.  The flatness of the Saito metric $\eta$ is proved in the Theorem \ref{main theorem 1}.

Let $t^1,t^2,..,t^n$ be given by the following generating function
\begin{equation}\label{generating function of flat of eta}
\begin{split}
v(z)=\frac{-1}{n}\left(t^{n}z+t^{n-1}z^2+.....+t^2z^{n-1}+t^1z^n+O(z^{n+1}) \right),
\end{split}
\end{equation}
be defined by the following condition 
\begin{equation*}
\begin{split}
\lambda(v)=\frac{1}{z^{n}}.
\end{split}
\end{equation*}
Moreover,
\begin{equation}\label{definition of t0}
t^0=\varphi_0-\frac{\theta_1^{\prime}((n+1)v_{n+1})}{\theta_1((n+1)v_{n+1})}\varphi_1+4\pi ig_1(\tau)\varphi_2.
\end{equation}

\begin{lemma}
The following identity holds
\begin{equation}\label{generating function of flat of eta 1}
\frac{n}{n+1-\alpha}\res_{p=\infty}\left( \lambda^{\frac{n+1-\alpha}{n}}(p) dp\right)=-\res_{\lambda=\infty}\left(v(z)\lambda^{\frac{1-\alpha}{n}} d\lambda\right)
\end{equation}
\end{lemma}

\begin{proof}
Consider the integration by parts 
\begin{equation}
\begin{split}
\frac{n}{n+1-\alpha}\int\left( \lambda^{\frac{n+1-\alpha}{n}}(p) dp\right)= p\lambda^{\frac{n+1-\alpha}{n}}-\int p \lambda^{\frac{1-\alpha}{n}}d\lambda
\end{split}
\end{equation}
Lemma proved.

\end{proof}

\begin{lemma}
The functions $t^1,t^2,..,t^n$  defined in (\ref{generating function of flat of eta}) can be obtained by the formula 
\begin{equation}\label{generating function of flat of eta 2}
\begin{split}
t^{\alpha}=-\res_{\lambda=\infty}\left(v(z)\lambda^{\frac{1-\alpha}{n}} d\lambda\right).
\end{split}
\end{equation}
\end{lemma}
\begin{proof}
Let 
\begin{equation*}
z=\left(\frac{1}{\lambda}\right)^{\frac{1}{n}},
\end{equation*}
 then, 
\begin{equation*}
\begin{split}
v(z)\lambda^{\frac{1-\alpha}{n}} d\lambda&=\left(  \frac{1}{n}\right) \left(t^{n}z+t^{n-1}z^2+.....+t^2z^{n-1}+t^1z^n+O(z^{n+1})     \right)z^{\alpha-1}nz^{-n-1}dz\\
&=\left(\sum_{\beta=1}^{n}  t^{\beta}z^{n+1-\beta}+O(z^{n+1})     \right)z^{\alpha-n-2}dz\\
&=\left( \sum_{\beta=1}^{n}  t^{\beta}z^{\alpha-\beta-1}+O(z^{\alpha-1}) \right) dz.
\end{split}
\end{equation*}
Hence, the residue is different from 0, when $\alpha=\beta$, resulting in this way the desired result.

\end{proof}

\begin{corollary}\label{corollary tn+1}
The coordinate $t^{n}$ can be written in terms of the coordinates $\varphi_0,\varphi_1,\varphi_2,..,\varphi_{n+1}$ as
\begin{equation}
t^{n}=n\left( \varphi_{n}   \right)^{ \frac{1}{n}    }.
\end{equation}
\end{corollary}
\begin{proof}

\begin{equation*}
\begin{split}
t^{n}&=n\res_{v=0} \lambda^{ \frac{1}{n}}(v) dv\\
&= n\res_{v=0} \left(  \frac{\varphi_{n}}{v^{n}}+\frac{\varphi_{n-1}}{v^{n-1}}+..+\frac{\varphi_{2}}{v^{2}}+\frac{\varphi_{1}}{v} +O(1)       \right)^{\frac{1}{n}} dv\\
&=n \res_{v=0}\frac{\left(\varphi_{n}\right)^{ \frac{1}{n}} }{v} \left(  1 +O(v) \right)^{\frac{1}{n}} dv=n\left(\varphi_{n}\right)^{ \frac{1}{n}}.
\end{split}
\end{equation*}

\end{proof}

\begin{lemma}
Let $t^1,t^2,..,t^n$ be defined in (\ref{generating function of flat of eta}), then
\begin{equation}\label{relation t wrt varphi 1}
t^{\alpha}=\frac{n}{n+1-\alpha}\left(\varphi_{n}\right)^{\frac{n+1-\alpha}{n}}\left( 1+\Phi_{n-\alpha} \right)^{\frac{n+1-\alpha}{n}},
\end{equation}
where 
\begin{equation}\label{definition of formal powers of varphi}
\begin{split}
\left( 1+\Phi_i \right)^{\frac{n+1-\alpha}{n}}=\sum_{d=0}^{\infty}  {{\frac{n+1-\alpha}{n}}\choose{d}}  \Phi_i^d,\\
\Phi_i^d=\sum_{i_1+i_2+..+i_d=i} \frac{\varphi_{\left(n-i_1\right)}}{\varphi_{n}}....\frac{\varphi_{\left(n-i_d\right)}}{\varphi_{n}}.
\end{split}
\end{equation}
\end{lemma}

\begin{proof}

\begin{equation*}
\begin{split}
t^{\alpha}&=\frac{n}{n+1-\alpha}\res_{v=0}\left( \lambda^{\frac{n+1-\alpha}{n}}(v) dv\right)\\
&=\frac{n}{n+1-\alpha}\res_{v=0}\left( \frac{\varphi_{n}}{v^{n}}+\frac{\varphi_{n-1}}{v^{n-1}}+..+\frac{\varphi_{2}}{v^{2}}+\frac{\varphi_{1}}{v}  +O(1)\right)^{\frac{n+1-\alpha}{n}} dv\\
&=\frac{n}{n+1-\alpha}\res_{v=0}\left(\frac{\varphi_{n}}{v^{n}}\right)^{\frac{n+1-\alpha}{n}}\left(1+\frac{\varphi_{n-1}}{\varphi_{n}}v+\frac{\varphi_{n-2}}{\varphi_{n}}v^2+..+\frac{\varphi_2}{\varphi_{n}}v^{n-2}+\frac{\varphi_1}{\varphi_{n}}v^{n-1}+O\left(v^{n+1}\right) \right)^{\frac{n+1-\alpha}{n}} dv\\
&=\frac{n}{n+1-\alpha}\res_{v=0}\left(\frac{\varphi_{n}}{v^{n}}\right)^{\frac{n+1-\alpha}{n}}\sum_{d=0}^{\infty}  {{\frac{n+1-\alpha}{n}}\choose{d}} \left(1+\frac{\varphi_{n-1}}{\varphi_{n}}v+..+\frac{\varphi_1}{\varphi_n}v^{n-1}+O\left(v^{n+1}\right) \right)^d dv\\
&=\frac{n}{n+1-\alpha}\res_{v=0}\left(\frac{\varphi_{n}}{v^{n}}\right)^{\frac{n+1-\alpha}{n}}\sum_{d=0}^{\infty}  {{\frac{n+1-\alpha}{n}}\choose{d}} \sum_{j_1+..+j_n=d}\frac{d!}{j_1!j_2!..j_n!}\prod_{i=1}^{n-1}\left(\frac{\varphi_{n-i}v^i}{\varphi_{n}}\right)^{j_i}(O(v^{n}))^{j_n} dv\\
&=\frac{n}{n+1-\alpha}\res_{v=0}\left(\varphi_{n}\right)^{\frac{n+1-\alpha}{n}}\sum_{d=0}^{\infty}  {{\frac{n+1-\alpha}{n}}\choose{d}} \sum_{j_1+..+j_n=d}\frac{d!}{j_1!j_2!..j_n!}\prod_{i=1}^{n-1}\left(\frac{\varphi_{n-i}}{\varphi_{n}}\right)^{j_i}v^{\left(\sum_{i=1}^{n-1} ij_i-n-1+\alpha\right)} dv\\
&+O(1)\\
&=\frac{n}{n+1-\alpha}\left(\varphi_{n}\right)^{\frac{n+1-\alpha}{n}}\sum_{d=0}^{\infty}  {{\frac{n+1-\alpha}{n}}\choose{d}} \sum_{\substack{ j_1+..+j_n=d\\ j_1+2j_2+3j_3+..+(n-1)j_{n-1}=n-\alpha}} \frac{d!}{j_1!j_2!..j_n!}\prod_{i=1}^{n-1}\left(\frac{\varphi_{n-i}}{\varphi_{n}}\right)^{j_i}\\
&=\frac{n}{n+1-\alpha}\left(\varphi_{n}\right)^{\frac{n+1-\alpha}{n+1}}\sum_{d=0}^{\infty}  {{\frac{n+1-\alpha}{n}}\choose{d}} \sum_{ i_1+..+i_d=n-\alpha} \frac{\varphi_{\left(n-i_1\right)}}{\varphi_{n+1}}....\frac{\varphi_{\left(n-i_d\right)}}{\varphi_{n}}\\
&=\frac{n}{n+1-\alpha}\left(\varphi_{n}\right)^{\frac{n+1-\alpha}{n}}\sum_{d=0}^{\infty}  {{\frac{n+1-\alpha}{n}}\choose{d}} \Phi_{n-\alpha}^d\\
&=\frac{n}{n+1-\alpha}\left(\varphi_{n}\right)^{\frac{n+1-\alpha}{n}}\left(1+\Phi_{n-\alpha}  \right)^{\frac{n+1-\alpha}{n}}.\\
\end{split}
\end{equation*}

\end{proof}

\begin{corollary}
The coordinate $t^{1}$ can be written in terms of the coordinates $\varphi_0,\varphi_1,\varphi_2,..,\varphi_{n}$ as
\begin{equation}
t^{1}=\varphi_1
\end{equation}
\end{corollary}
\begin{proof}
Using the relation (\ref{relation t wrt varphi 1}) for $\alpha=1$
\begin{equation}
\begin{split}
t^{1}&=\varphi_{n}\left( 1+\Phi_{n-1} \right)=\varphi_{n}\sum_{i_1=n-1} \frac{\varphi_{n-i_1}}{\varphi_n}=\varphi_1.
\end{split}
\end{equation}

\end{proof}

\begin{lemma}
Let $\varphi_0,\varphi_1,\varphi_2,..,\varphi_{n}$ and $\lambda(v)$ be defined in (\ref{superpotentialAn}), then
\begin{equation}\label{formula varphi 1}
k\varphi_k=\res_{v=0} k\lambda v^{k-1}dv. 
\end{equation}
\end{lemma}
\begin{proof}

\begin{equation*}
\begin{split}
\res_{v=0} k\lambda v^{k-1}dv &= \res_{v=0} k\left(\frac{\varphi_{n}}{v^{n}}+\frac{\varphi_{n-1}}{v^{n-1}}+..+\frac{\varphi_{k}}{v^{k}}+..+\frac{\varphi_{2}}{v^{2}}+ \frac{\varphi_{1}}{v}+ O(1)     \right) v^{k-1}dv\\
&=k\varphi_k.
\end{split}
\end{equation*}
\end{proof}

\begin{lemma}
Let $\varphi_0,\varphi_1,\varphi_2,..,\varphi_{n}$ and $\lambda(v)$ be defined in (\ref{superpotentialAn}), then
\begin{equation}\label{formula varphi 2}
k\varphi_k=-\res_{\lambda=\infty} v^{k}d\lambda.
\end{equation}
\end{lemma}
\begin{proof}
Using formula (\ref{formula varphi 2}) and integration by parts
\begin{equation*}
\begin{split}
k\varphi_k=\res_{v=0} k\lambda v^{k-1}dv &= -\res_{\lambda=\infty}  v^{k}d\lambda.
\end{split}
\end{equation*}
\end{proof}

\begin{lemma}
Let $\varphi_0,,\varphi_1,\varphi_2,..,\varphi_{n}$, $\lambda(v)$ be defined in (\ref{superpotentialAn}) and $(t^1,..,t^{n})$ be defined in (\ref{generating function of flat of eta}) , then
\begin{equation}\label{formula varphi 3}
k\varphi_k=\frac{(-1)^k}{n^{k-1}}T_{n}^k,
\end{equation}
where 
\begin{equation}\label{big T}
T_{n}^k=\sum_{i_1+..+i_k=n}  t^{\left(n+1-i_1\right)}...t^{\left(n+1-i_k\right)}.
\end{equation}
\end{lemma}
\begin{proof}
Let $z:=\left(\frac{1}{\lambda}\right)^{\frac{1}{n} }$, then by using equation (\ref{formula varphi 2}):

\begin{equation*}
\begin{split}
 -\res_{\lambda=\infty}  v^{k}d\lambda&= \res_{z=0} \frac{nv^{k}(z) dz}{z^{n+1}}\\
&=\res_{z=0}\frac{(-1)^k}{n^{k-1}} \left( t^{n}z+t^{n-1}z^2+..+t^2z^{n-1}+O(z^{n+1}) \right)^{k} \frac{dz}{z^{n+1}}\\
&=\res_{z=0}\frac{(-1)^k}{n^{k-1}} \sum_{j_1+j_2+..+j_n+j_{n+1}=k} \left(t^{n}z\right)^{j_1}\left(t^{n-1}z^2\right)^{j_2}..\left(t^{2}z^{n-1}\right)^{j_n} \left( O(z^{n+1})  \right)^{j_{n+2}}  \frac{dz}{z^{n+1}}\\
&=\frac{(-1)^k}{n^{k-1}} \sum_{\substack{ j_1+j_2+..+j_n=k\\ j_1+2j_2+3j_3+..+(n)j_{n}=n}} \frac{k!}{j_1!j_2!..j_n!} \left(t^{n}\right)^{j_1}\left(t^{n-1}\right)^{j_2}..\left(t^{2}\right)^{j_n}\\
&=\frac{(-1)^k}{n^{k-1}} \sum_{i_1+..i_k=n} t^{\left(n+1-i_1\right)}...t^{\left(n+1-i_k\right)}   \\
&=\frac{(-1)^k}{n^{k-1}} T_{n}^k.
\end{split}
\end{equation*}

\end{proof}

\begin{lemma}
Let $T_{n}^k$ be defined in (\ref{big T}), then
\begin{equation}\label{ derivative T}
\frac{\partial T_{n}^k}{\partial t^{\alpha}}=kT^{k-1}_{\alpha-1}.
\end{equation}
\end{lemma}
\begin{proof}
\begin{equation*}
\begin{split}
\frac{\partial T_{n}^k}{\partial t^{\alpha}}&=\frac{\partial}{\partial t^{\alpha}}\left(  \sum_{i_1+..i_k=n} t^{(n+1-i_1)}...t^{(n+1-i_k)} \right)\\
&= \sum_{i_1+..i_k=n}k\delta_{n+1-i_k,\alpha} t^{(n+1-i_1)}...t^{(n+1-i_{k-1})}\\
&=k \sum_{i_1+..i_{k-1}=\alpha-1} t^{(n+1-i_1)}...t^{(n+1-i_{k-1})}\\
&=kT^{k-1}_{\alpha-1}.
\end{split}
\end{equation*}

\end{proof}

At this stage, we are able to compute  the coefficients  of $\eta(d t^{\alpha},dt^{\beta})$.

\begin{theorem}\label{main theorem}
Let $(t^1,..,t^{n})$ defined in (\ref{generating function of flat of eta}), and $\eta^{*}$ defined in (\ref{metric eta def}). Then,
\begin{equation}\label{metric 2n+1 eta}
\eta^{*}(dt^{\alpha},dt^{n+3-\beta})=n\delta_{\alpha\beta}.
\end{equation}
\end{theorem}

\begin{proof}
If $i,j\neq 0$
\begin{equation}\label{transition function metric}
\eta^{*}(d\varphi_i,d\varphi_j)=\sum_{\alpha=1}^{n}\sum_{\beta=1}^{n}\frac{\partial \varphi_i}{\partial t^{\alpha}}\frac{\partial \varphi_j}{\partial t^{\beta}}\eta^{*}(dt^{\alpha},dt^{\beta})
\end{equation}
Using  (\ref{ derivative T}) and (\ref{formula varphi 3}), we get
\begin{equation}\label{derivative of varphi wrt t}
\begin{split}
\frac{\partial \varphi_k}{\partial t^{\alpha}}=\frac{(-1)^k}{n^{k-1}}k T_{\alpha-1}^{k-1}
\end{split}
\end{equation}

\begin{equation}\label{eq necessary transition function}
\sum_{\alpha=1}^{n}\frac{\partial \varphi_i}{\partial t^{\alpha}}\frac{\partial \varphi_{n+3-j}}{\partial t^{n+3-\alpha}}=\sum_{\alpha=1}^{n}\frac{(-1)^{i-j+n+1}}{n^{i-j+n+1}}T_{\alpha-1}^{i-1}T_{n+2-\alpha}^{n+2-j}
\end{equation}
Using the second of the equation (\ref{technical lemma Coxeter}) in (\ref{eq necessary transition function})
\begin{equation}\label{last eq saito}
\begin{split}
\sum_{\alpha=1}^{n}  \frac{\partial \varphi_i}{\partial t^{\alpha}}\frac{\partial \varphi_{n+3-j}}{\partial t^{n+3-\alpha}}&=\frac{T_{n+1}^{n+1+i-j}}{n}\\
&=\frac{(n+1+i-j)}{n}\varphi_{(n+1+i-j)}
\end{split}
\end{equation}
Note that the following identity holds
\begin{equation}
\begin{split}
\sum_{\alpha=2}^{n+1}\frac{\partial \varphi_i}{\partial t^{\alpha}}\frac{\partial \varphi_{n+3-j}}{\partial t^{n+3-\alpha}}&=\sum_{\alpha=2}^{n+1}\sum_{\alpha=2}^{n+1}\frac{\partial \varphi_i}{\partial t^{\alpha}}\frac{\partial \varphi_{n+3-j}}{\partial t^{n+3-\beta}}\delta_{\alpha\beta}.
\end{split}
\end{equation}
On another hand, using the first equation of  (\ref{formula main theorem for eta varphi}), we have
\begin{equation*}
\begin{split}
\eta^{*}(d\varphi_i,d\varphi_{n+3-j})&=(n+1+i-j)\varphi_{n+1+i-j}\\
&=n\sum_{\alpha=2}^{n+1}\sum_{\alpha=2}^{n+1}\frac{\partial \varphi_i}{\partial t^{\alpha}}\frac{\partial \varphi_{n+3-j}}{\partial t^{n+3-\beta}}\delta_{\alpha\beta}\\&=\sum_{\alpha=2}^{n+1}\sum_{\alpha=2}^{n+1}\frac{\partial \varphi_i}{\partial t^{\alpha}}\frac{\partial \varphi_{n+3-j}}{\partial t^{n+3-\beta}}\eta^{*}(dt^{\alpha}, dt^{n+3-\beta})\\
\end{split}
\end{equation*}
Then, we obtain
\begin{equation*}
\begin{split}
\eta^{*}(dt^{\alpha},dt^{n+3-\beta})=n\delta_{\alpha\beta}.
\end{split}
\end{equation*}

\end{proof}

\begin{lemma}\label{final lemma 2}
Let $(t^1,..,t^{n})$ be defined in (\ref{generating function of flat of eta}), and $\eta^{*}$ be defined in (\ref{metric eta def}). Then,
\begin{equation}
\begin{split}
&\eta^{*}(dt^i,d\tau)=-2\pi i \delta_{i0},\\
&\eta^{*}(dt^i,dv_{n+1})=-\frac{\delta_{i1}}{n+1}.\\
\end{split}
\end{equation}
\end{lemma}
\begin{proof}
Using corollary \ref{corollary vn1 eta}   and lemma \ref{lemma tau varphi}, we have\\
\begin{equation*}
\begin{split}
\eta^{*}(dt^i,d\tau)&=\frac{\partial t^i}{\partial \varphi_j}\eta^{*}(d\varphi_j,d\tau)=-2\pi i \delta_{i0}.
\end{split}
\end{equation*}
In addition, if $i\neq 0$
\begin{equation*}
\begin{split}
\eta^{*}(dt^i,dv_{n+1})&=\frac{\partial t^i}{\partial \varphi_j}\eta^{*}(d\varphi_j,dv_{n+1})=0, \quad i\neq 1,\\
\eta^{*}(dt^1,dv_{n+1})&=\eta^{*}(d\varphi_1,dv_{n+1})=-\frac{1}{n+1}.
\end{split}
\end{equation*}
Computing $dt^0$ with respect the variables $\varphi_i$ by using (\ref{definition of t0}),
\begin{equation}\label{differential t0}
\begin{split}
dt^0&=d\varphi_0-(n+1)\frac{\partial^2\log(\theta_1((n+1)v_{n+1})}{\partial v_{n+1}^2}\varphi_1dv_{n+1}-\frac{\partial\log(\theta_1((n+1)v_{n+1})}{\partial v_{n+1}}d\varphi_1\\
&+4\pi ig_1^{\prime}(\tau)\varphi_2d\tau+4\pi ig_1(\tau)d\varphi_2.
\end{split}
\end{equation}
Hence,
\begin{equation*}
\begin{split}
\eta^{*}(dt^0,dv_{n+1})&=\eta^{*}(d\varphi_0,dv_{n+1})-\frac{\partial\log(\theta_1((n+1)v_{n+1})}{\partial v_{n+1}}\eta^{*}(d\varphi_1,dv_{n+1})\\
&=0.
\end{split}
\end{equation*}

\end{proof}

\begin{lemma}\label{final lemma}
Let $t^0$ be defined in (\ref{definition of t0}), and $\eta^{*}$ be defined in (\ref{metric eta def}). Then,
\begin{equation}
\eta^{*}(dt^0,dt^{\alpha})=0, \quad \alpha\neq 0.
\end{equation}
\end{lemma}

\begin{proof}
Using the definition of $\eta^{*}$ in equation (\ref{metric eta def}),  formula (\ref{formula main theorem for eta varphi}), and (\ref{metric eta consequence})\\
If $i>1$, 
\begin{equation*}
\begin{split}
\eta^{*}(dt^0,d\varphi_i)&=\eta^{*}(d\varphi_0,d\varphi_i)+4\pi ig_1(\tau)\eta^{*}(d\varphi_2,d\varphi_i)+4\pi ig_1^{\prime}(\tau)\eta^{*}(d\tau,d\varphi_i)\\
&=\eta^{*}(d\varphi_0,d\varphi_i)+4\pi ig_1(\tau)\eta^{*}(d\varphi_2,d\varphi_i)\\
&=\tilde \eta^{*}(d\varphi_0,d\varphi_i)-4\pi ig_1(\tau)k_i\varphi_i+4\pi ig_1(\tau)k_i\varphi_i=0\\
&=-4\pi ig_1(\tau)k_i\varphi_i+4\pi ig_1(\tau)k_i\varphi_i=0\\
\end{split}
\end{equation*}
Then, if $\alpha>1$
\begin{equation*}
\begin{split}
\eta^{*}(dt^0,dt^{\alpha})=\sum_{\alpha=2}^{n} \frac{\partial t^{\alpha}}{\partial \varphi_i}\eta^{*}(dt^0,d\varphi_i)=0.
\end{split}
\end{equation*}
Computing $\eta^{*}(dt^0,dt^1)$ by using (\ref{differential t0})
\begin{equation*}
\begin{split}
\eta^{*}(dt^0,dt^1)&=\eta^{*}(d\varphi_0,d\varphi_1)-(n+1)\frac{\partial^2\log(\theta_1((n+1)v_{n+1})}{\partial v_{n+1}^2}\varphi_1\eta^{*}(dv_{n+1},d\varphi_1)\\
&=0.
\end{split}
\end{equation*}

\end{proof}

Instead of considering the coefficients  of the metric $\eta^{*}$,  let us investigate the flatness of its inverse $\eta$ in order to make the computations shorter.

\begin{lemma}\label{sub metric}
The metric
\begin{equation}\label{eq metric 2n+1}
\sum_{\alpha=2}^{n} \eta(dt^{\alpha},dt^{n+3-\alpha})dt^{\alpha}dt^{n+3-\alpha}-2(n+1)dt^1dv_{n+1}-\frac{1}{\pi i}dt^0d\tau
\end{equation}
is invariant under the second action of (\ref{jacobigroupAntilde}).
\end{lemma}
\begin{proof}
Under the second action of (\ref{jacobigroupAntilde}), we have that 
\begin{equation*}
\sum_{\alpha=2}^{n} \eta(dt^{\alpha},dt^{n+3-\alpha})dt^{\alpha}dt^{n+3-\alpha}
\end{equation*}
are invariant under the second action of (\ref{jacobigroupAntilde}), because  the relationship between $t^i$ and $\varphi_i$ be given by (\ref{relation t wrt varphi 1}), and the fact that the Jacobi forms $\{\varphi_i\}$ are invariant under the second action of (\ref{jacobigroupAntilde}). $t^0$ and $v_{n+1}$ have the following transformation law
\begin{equation}
\begin{split}
&t^0\mapsto t^0-2\pi i(n+1)\lambda_{n+1}t^1\\
&v_{n+1}\mapsto v_{n+1}+\lambda_{n+1}\tau+\mu_{n+1}\\
\end{split}
\end{equation}
Hence, its differentials are
\begin{equation}\label{eq lemma metric equivariant 2n+1}
\begin{split}
&dt^0\mapsto dt^0-2\pi i(n+1)\lambda_{n+1}dt^1\\
&dv_{n+1}\mapsto dv_{n+1}+\lambda_{n+1}d\tau
\end{split}
\end{equation}
Substituting (\ref{eq lemma metric equivariant 2n+1}) in (\ref{eq metric 2n+1}) we get the desired result.
\end{proof}

\begin{theorem}\label{main theorem 1}
Let $(t^0,t^1,t^2,..,t^{n})$ defined in (\ref{generating function of flat of eta}), and $\eta^{*}$ defined in (\ref{metric eta def}). Then,
\begin{equation}
\begin{split}
&\eta^{*}(dt^{\alpha},dt^{n+3-\beta})=-(n+1)\delta_{\alpha\beta}, \quad 2\leq\alpha,\beta\leq n\\
&\eta^{*}(dt^1,dt^{\alpha})=0,\\
&\eta^{*}(dt^0,dt^{\alpha})=0,\\
&\eta^{*}(dt^i,d\tau)=-2\pi i \delta_{i0},\\
&\eta^{*}(dt^i,dv_{n+1})=-\frac{\delta_{i1}}{n+1}.\\
\end{split}
\end{equation}
Moreover, the coordinates $t^0,t^1,t^2,..,t^{n},v_{n+1},\tau$ are the flat coordinates of $\eta^{*}$.
\end{theorem}
\begin{proof}
The theorem is already proved  for $\alpha,\beta \in \{2,..,n  \}$ in theorem \ref{main theorem}, and for the rest  in the lemma \ref{final lemma} and \ref{final lemma 2}. The only missing part is to prove 
 \begin{equation}
\eta^{*}(dt^{0},dt^{0})=0.
\end{equation}
Recall that from corollary \ref{important final corollary}, the metric $\eta^{*}$  is invariant under the second action of (\ref{jacobigroupAntilde}). Moreover, the same statement is valid for (\ref{eq metric 2n+1}), because of lemma \ref{sub metric}.
 However, the tensor $dt^0\otimes dt^0$ have a non-trivial transformation law under this action. Hence, if the coefficient of the component $dt^0\otimes dt^0$  is different from $0$, we have a contradiction with corollary \ref{important final corollary}.

\end{proof}

\begin{corollary}\label{corollary non degenerate}
The metric $\eta^{*}(d\varphi_i,d\varphi_j):=\frac{\partial g^{*}(d\varphi_i,d\varphi_j)}{\partial \varphi_0}$ is triangular, and non degenerate.
\end{corollary}

\begin{definition}
Let $\eta^{*}=\eta^{\alpha\beta}\frac{\partial}{\partial t^{\alpha}}\otimes\frac{\partial}{\partial t^{\beta}}$ defined in (\ref{metric eta def}). The metric defined by 
\begin{equation}
\eta=\eta_{\alpha\beta}d t^{\alpha}\otimes dt^{\beta}
\end{equation}
 is denoted by $\eta$.
\end{definition}

\subsection{The extended ring of Jacobi forms}\label{extended ring jtildean}

 The flat coordinates of the Saito metric $\eta$ of the orbit space of $\Ja(\tilde A_n)$ does not live in the orbit space of $\Ja(\tilde A_n)$, but live in a suitable covering of this orbit space. The main goal of this section is in describing this covering as the space such that the ring of functions of this covering is a suitable extension of the ring of Jacobi forms.

\begin{lemma}
The coordinates $(t^0,t^1,t^2,..,v_{n+1},\tau)$ defined on (\ref{generating function of flat of eta}) have the following transformation laws under the action of the group $\Ja(\tilde A_n)$: They transform as follows under  the second action of (\ref{jacobigroupAntilde}):
\begin{equation}\label{transformation law of the Saito flat coordinates 1}
\begin{split}
&t^0\mapsto t^0-2\pi i(n+1)\lambda_{n+1}t^1\\
&t^{\alpha}\mapsto t^{\alpha},\quad \alpha\neq 0\\
&v_{n+1}\mapsto v_{n+1}+\mu_{n+1}+\lambda_{n+1}\tau\\
&\tau\mapsto \tau
\end{split}
\end{equation}

Moreover, they transform as follows under  the third action (\ref{jacobigroupAntilde})
\begin{equation}\label{transformation law of the Saito flat coordinates 2}
\begin{split}
&t^0\mapsto t^0+\frac{2c\sum_{\alpha,\beta\neq 0,\tau}\eta_{\alpha\beta}t^{\alpha}t^{\beta}}{c\tau+d}\\
&t^{\alpha}\mapsto \frac{t^{\alpha}}{c\tau+d}, \quad \alpha\neq 0\\
&v_{n+1}\mapsto \frac{v_{n+1}}{c\tau+d}\\
&\tau\mapsto \frac{a\tau+b}{c\tau+d}
\end{split}
\end{equation}
\end{lemma}
\begin{proof}

 Note that the term $\Phi_i^d$ equation (\ref{definition of formal powers of varphi}) has weight $+i$, then using that $\varphi_{n}$ has weight $-n$, we have that the weight of $t^{\alpha}$ for $\alpha\neq 1$ must have weight $-1$ due to (\ref{relation t wrt varphi 1}). The transformation law of $t^1$ follows from the transformation law of $\frac{\theta_1^{\prime}((n+1)v_{n+1})}{\theta_1((n+1)v_{n+1})}$ and $g_1(\tau)$
 \begin{equation}
 \begin{split}
 \frac{\theta_1^{\prime}((n+1)v_{n+1}+(n+1)\lambda_{n+1}\tau+\mu_{n+1},\tau)}{\theta_1((n+1)v_{n+1}+(n+1)\lambda_{n+1}\tau+\mu_{n+1},\tau)}&=\frac{\theta_1^{\prime}((n+1)v_{n+1})}{\theta_1((n+1)v_{n+1})}-2\pi i(n+1)\lambda_{n+1},\\
\frac{\theta_1^{\prime}(\frac{(n+1)v_{n+1}}{c\tau+d},\frac{a\tau+b}{c\tau+d})}{\theta_1(\frac{(n+1)v_{n+1}}{c\tau+d},\frac{a\tau+b)}{c\tau+d}}&=\frac{\theta_1^{\prime}((n+1)v_{n+1})}{\theta_1((n+1)v_{n+1})}+2\pi ic(n+1)v_{n+1},\\
g_1(\frac{a\tau+b}{c\tau+d})&=(c\tau+d)^2g_1(\tau)+2c(c\tau+d),
\end{split}
\end{equation}
and by using equation (\ref{formula varphi 3}) for $k=2$.

\end{proof}

In addition , from the formula (\ref{definition of formal powers of varphi}) it is clear that the multivaluedness of $(t^1,..,t^{n})$ comes from $\left(\varphi_{n}\right)^{\frac{1}{n}}$. Therefore, the coordinates lives in a suitable covering over the orbit space of the group $\Ja(\tilde A_n)$. This covering is obtained by forgetting to act the Coxeter group $A_n$ and the $SL_2(\mathbb{Z})$ action, and the translation action $v_{n+1}\mapsto v_{n+1}+\lambda_{n+1}\tau+\mu_{n+1}$ of $\Ja(\tilde A_n)$ on $\mathbb{C}\oplus\mathbb{C}^{n+1}\oplus\mathbb{H}$. The only remaining  part of the $\Ja(\tilde A_n)$ action are the translations 
\begin{equation*}
v_i\mapsto v_i+\lambda_i\tau+\mu_i, \quad i\neq n+1.
\end{equation*}
Hence, the coordinates $(t^1,..,t^{n})$ live in  n-dimensional tori with fixed symplectic base of the torus homology, a fixed chamber in the tori parametrised by $(v_{n+1},\tau)$, and with a branching divisor $Y:=\{\varphi_{n}=0 \}$.  Another way to describe this covering is using the flat coordinates of the intersection form $(u,v_0,v_1,..,v_{n+1},\tau)$, and to fix a lattice $\tau$, a representative of the action 
\begin{equation}
v_{n+1}\mapsto v_{n+1}+\lambda_{n+1}\tau+\mu_{n+1}, 
\end{equation}
and a representative of the $A_n$ action. Then, the desired covering of the orbit space of the group $\Ja(\tilde A_n)$ is defined  by
\begin{equation}\label{covering space for the tilde an case}
\widetilde{\mathbb{C}\oplus\mathbb{C}^{n+1}\oplus\mathbb{H}}/\Ja(\tilde A_n):=\mathbb{C}\oplus\mathbb{C}^{n+1}\oplus\mathbb{H}/(\mathbb{Z}^n\oplus\tau\mathbb{Z}^n),
\end{equation}
where $\mathbb{Z}^n\oplus\tau\mathbb{Z}^n$ acts on $\mathbb{C}\oplus\mathbb{C}^{n+1}\oplus\mathbb{H}$ by
\begin{equation}
\begin{split}
&v_i\mapsto v_i+\lambda_i\tau+\mu_i, \quad i\neq n+1,\\
&u\mapsto u-2A_{ij}\lambda_iv_j-A_{ij}\lambda_i\lambda_j\tau,\\
&v_{n+1}\mapsto v_{n+1},\\
&\tau\mapsto \tau.
\end{split}
\end{equation}
where $A_{ij}$ is given by (\ref{metric An Coxeter}).

\begin{remark}
Note that due to the lemma \ref{jacobiform1} the covering space \ref{covering space for the tilde an case} is isomorphic to a suitable covering over the Hurwitz space $H_{1,n-1,0}$. The covering over $H_{1,n-1,0}$ is given by a fixation of base of the  homology in the tori generated by the lattice $(1,\tau)$,  fixation of root $\lambda$ (\ref{superpotentialAn}) near $\infty$, and a fixation of a logarithm root. 
\end{remark}

 In order to manipulate the geometric objects of this covering, it is more convenient  to use their ring of functions. Hence, we define:

\begin{definition}
The extended ring of Jacobi forms with respect the ring of coefficients is the following ring
\begin{equation}
\widetilde E_{\bullet,\bullet}[\varphi_0,\varphi_1,..,\varphi_n],
\end{equation}
where 
\begin{equation}
\widetilde E_{\bullet,\bullet}=E_{\bullet,\bullet}\oplus\{ g_1(\tau)\} \oplus  \{   \frac{\theta_1^{\prime}((n+1)v_{n+1})}{\theta_1((n+1)v_{n+1})}      \}.
\end{equation}
\end{definition}

\begin{lemma}\label{coefficient g varphi ring}
The coefficients of the intersection form $g^{ij}$ on the coordinates $\varphi_0,\varphi_1,..,\varphi_n,v_{n+1},\tau$ belong to the ring $\widetilde E_{\bullet,\bullet}[\varphi_0,\varphi_1,..,\varphi_n]$.
\end{lemma}
\begin{proof}
It is a consequence of the formula (\ref{definition of M coef}).

\end{proof}

\begin{lemma}\label{lemma gab index}
The coefficients of the intersection form $g^{\alpha\beta}$ on the coordinates $t^0,t^1,..,t^n,v_{n+1},\tau$ belong to the ring $\widetilde E_{\bullet,\bullet}[t^0,t^1,..,t^n,\frac{1}{t^n}]$.
\end{lemma}
\begin{proof}
Using the transformation law of $g^{\alpha\beta}$
\begin{equation}
g^{\alpha\beta}=\frac{\partial t^{\alpha}}{\partial \varphi_i}\frac{\partial t^{\beta}}{\partial \varphi_j}g(d\varphi_i,d\varphi_j),
\end{equation}
we realise the term $\frac{\partial t^{\alpha}}{\partial \varphi_i}$ as polynomial in $t^0,t^1,..,t^n,\frac{1}{t^n}$ due to the relations (\ref{relation t wrt varphi 1}) and (\ref{formula varphi 3}).

\end{proof}

\subsection{Christoffel symbols of the intersection form}

In this section, we consider the Christoffel symbols of the intersection form (\ref{metrich1n cotangent}). In particular, we prove that this Christoffel symbols, in coordinates $\varphi_0,,\varphi_1,\varphi_2,..,\varphi_{n},v_{n+1},\tau$, live  in $\widetilde E_{\bullet,\bullet}\left[\varphi_0,\varphi_1,..,\varphi_n\right]$. In addition, we will show that the Christoffel symbols depend at most linear in $\varphi_0$ in this coordinates.

Recall that the Christoffel symbols $\Gamma_k^{ij}(\varphi)$ associated with the intersection form $g^{*}$ is given in terms of the conditions (\ref{Levi Civita contravariant coxeter chapter}).

\begin{lemma}\label{diagonal Christoffel symbol}
Let $\varphi_0,,\varphi_1,\varphi_2,..,\varphi_{n},v_{n+1},\tau$ be defined in (\ref{superpotentialAn}), then $\Gamma_{j}^{ii}$ depend at most linear on $\varphi_0$.
\end{lemma}
\begin{proof}
Using the first  condition of (\ref{Levi Civita contravariant coxeter chapter})
\begin{equation*}
\begin{split}
\partial_kg^{ii}&=2\Gamma_k^{ii}
\end{split}
\end{equation*}
Recall that due to the corollary \ref{main corollary eta}, the metric $g^{ij}$ depend at most linear on $\varphi_0$. Then,
\begin{equation*}
\begin{split}
2\frac{\partial^2\Gamma_k^{ii}}{\partial \varphi_0^2}=\partial_0^2\partial_kg^{ii}=\partial_k\partial_0^2g^{ii}=0.
\end{split}
\end{equation*}

\end{proof}

\begin{lemma}\label{ Levi Civita connection tau}
Let $\varphi_0,,\varphi_1,\varphi_2,..,\varphi_{n},v_{n+1},\tau$ be defined in (\ref{superpotentialAn}), then
\begin{equation}
\begin{split}
&\Gamma_{j}^{i\tau}=0,\\
&\Gamma_{k}^{\tau k}=-2\pi i\frac{\delta_{jk}}{k}.
\end{split}
\end{equation}
\end{lemma}

\begin{proof}
Let $\Gamma_k^{ij}(x)$, in the coordinates $x_1,..,x_n$, and $\Gamma_{l}^{pq}(y)$ in the coordinates $y_1,..,y_n$, then the transformation law of the Christoffel symbol defined in the cotangent bundle is the following
\begin{equation}
\Gamma_k^{ij}(x)=\frac{\partial x^i}{\partial y^p}\frac{\partial x^j}{\partial y^q}\frac{\partial y^l}{\partial x^k}\Gamma_{l}^{pq}(y)+\frac{\partial x^i}{\partial y^p}\frac{\partial}{\partial x^k}\left(\frac{\partial x^j}{\partial y^q}\right)g^{pq}(y).
\end{equation}
In particular, the $\Gamma_k^{ij}(\varphi)$ in the coordinates $(\varphi_0,\varphi_1,..,\varphi_n,v_{n+1,\tau})$ could be derived from the Christoffel symbol in the coordinates $v_0,v_1,..,v_{n+1},\tau$ which is $0$. Then,
\begin{equation}
\Gamma_k^{ij}(\varphi)=\frac{\partial \varphi_i}{\partial v_p}\frac{\partial}{\partial \varphi_k}\left(\frac{\partial \varphi_j}{\partial v_q}\right)g^{pq}(v).
\end{equation}
Computing $\Gamma_{j}^{iv_{n+1}}$,
\begin{equation}
\begin{split}
\Gamma_k^{i\tau}(\varphi)&=\frac{\partial \varphi_i}{\partial v_p}\frac{\partial}{\partial \varphi_k}\left(\frac{\partial \tau}{\partial v_q}\right)g^{pq}(v)\\
&=-2\pi  \varphi_i\frac{\partial}{\partial \varphi_k}\left(1\right)=0.
\end{split}
\end{equation}
Computing $\Gamma_{j}^{v_{n+1}i}$ by using the first  condition of (\ref{Levi Civita contravariant coxeter chapter}),
\begin{equation}
\begin{split}
\Gamma_j^{\tau k}(\varphi)&=\partial_jg^{k\tau}-\Gamma_{j}^{k\tau}\\
&=\partial_jg^{k\tau}=-2\pi i \frac{\delta_{jk}}{k}.
\end{split}
\end{equation}

\end{proof}

\begin{lemma}\label{Levi Civita vn1}
Let $\varphi_0,,\varphi_1,\varphi_2,..,\varphi_{n},v_{n+1},\tau$ be defined in (\ref{superpotentialAn}), then
\begin{equation}
\begin{split}
&\Gamma_{j}^{iv_{n+1}}=0,\\
&\Gamma_{j}^{v_{n+1}i}=\frac{\partial g^{iv_{n+1}}}{\partial \varphi_j}\in \widetilde E_{\bullet,\bullet}[\varphi_0,\varphi_1,.,\varphi_n],\\
&\Gamma_{v_{n+1}}^{ij}\in \widetilde E_{\bullet,\bullet}[\varphi_0,\varphi_1,.,\varphi_n].
\end{split}
\end{equation}
Moreover, these Christoffel symbols are at most linear on $\varphi_0$.
\end{lemma}
\begin{proof}
Let $\Gamma_k^{ij}(x)$, in the coordinates $x_1,..,x_n$, and $\Gamma_{l}^{pq}(y)$ in the coordinates $y_1,..,y_n$, then the transformation law of the Christoffel symbol  defined in the cotangent bundle is the following
\begin{equation}
\Gamma_k^{ij}(x)=\frac{\partial x^i}{\partial y^p}\frac{\partial x^j}{\partial y^q}\frac{\partial y^l}{\partial x^k}\Gamma_{l}^{pq}(y)+\frac{\partial x^i}{\partial y^p}\frac{\partial}{\partial x^k}\left(\frac{\partial x^j}{\partial y^q}\right)g^{pq}(y).
\end{equation}
In particular, the $\Gamma_k^{ij}(\varphi)$ in the coordinates $(\varphi_0,\varphi_1,..,\varphi_n,v_{n+1,\tau})$ could be derived from the Christoffel symbol in the coordinates $v_0,v_1,..,v_{n+1},\tau$ which is $0$. Then,
\begin{equation}
\Gamma_k^{ij}(\varphi)=\frac{\partial \varphi_i}{\partial v_p}\frac{\partial}{\partial \varphi_k}\left(\frac{\partial \varphi_j}{\partial v_q}\right)g^{pq}(v).
\end{equation}
Computing $\Gamma_{k}^{iv_{n+1}}$,
\begin{equation}
\begin{split}
\Gamma_k^{iv_{n+1}}(\varphi)&=\frac{\partial \varphi_i}{\partial v_p}\frac{\partial}{\partial \varphi_k}\left(\frac{\partial v_{n+1}}{\partial v_q}\right)g^{pq}(v)\\
&=\frac{\partial \varphi_i}{\partial v_{n+1}}\frac{\partial}{\partial \varphi_k}\left(1\right)g^{v_{n+1}v_{n+1}}(v)=0.
\end{split}
\end{equation}
Computing $\Gamma_{k}^{v_{n+1}i}$ by using the first  condition of (\ref{Levi Civita contravariant coxeter chapter}),
\begin{equation}
\begin{split}
\Gamma_k^{v_{n+1}i}(\varphi)&=\partial_kg^{iv_{n+1}}-\Gamma_{j}^{iv_{n+1}}\\
&=\partial_kg^{iv_{n+1}}.
\end{split}
\end{equation}
Since, $g^{iv_{n+1}}\in \widetilde E_{\bullet,\bullet}[\varphi_0,\varphi_1,.,\varphi_n]$, we have that $\Gamma_k^{v_{n+1}i}(\varphi) \in \widetilde E_{\bullet,\bullet}[\varphi_0,\varphi_1,.,\varphi_n]$. In addition, since the metric $\eta^{iv_{n+1}}$ is independent of $\varphi_0$ due to the corollary \ref{main corollary eta}, we have that $\partial_kg^{iv_{n+1}}$ is at most linear on $\varphi_0$\\
Computing $\Gamma_{v_{n+1}}^{ij}$,
\begin{equation}
\begin{split}
\Gamma_{v_{n+1}}^{ij}(\varphi)&=\frac{\partial \varphi_i}{\partial v_p}\frac{\partial}{\partial v_{n+1}}\left(\frac{\partial \varphi_j}{\partial v_q}\right)g^{pq}(v)\\
&=\frac{\partial \varphi_i}{\partial v_p}\frac{\partial}{\partial v_q }\left(\frac{\partial \varphi_j}{\partial  v_{n+1}}\right)g^{pq}(v)\\
\end{split}
\end{equation}
 In the subsequent computation, whenever appears a function depend only on $v_{n+1},\tau$, we will call it by $h(v_{n+1},\tau)$, because for our purpose, it is enough to prove that the subsequent function belong to the ring $ \widetilde E_{\bullet,\bullet}[\varphi_0,\varphi_1,.,\varphi_n]$.\\
 If $j>1$, using equation (\ref{derivative of varphij wrt vn1}),
\begin{equation}
\begin{split}
\Gamma_{v_{n+1}}^{ij}(\varphi)&=\frac{\partial \varphi_i}{\partial v_p}\frac{\partial}{\partial v_q }\left(\frac{\partial \varphi_j}{\partial  v_{n+1}}\right)g^{pq}(v)\\
&=\frac{\partial \varphi_i}{\partial v_p}\frac{\partial}{\partial v_q }\left( \varphi_{j-1}+h(v_{n+1},\tau )\varphi_n    \right)g^{pq}(v)\\
&=\frac{\partial \varphi_i}{\partial v_p}\frac{\partial}{\partial v_q }\left( \varphi_{j-1}+h(v_{n+1},\tau)\varphi_n    \right)g^{pq}(v)\\
&=\frac{\partial \varphi_i}{\partial v_p}\frac{\partial}{\partial v_q }\left( \varphi_{j-1} \right)g^{pq}(v)+h(v_{n+1},\tau)\frac{\partial \varphi_i}{\partial v_p}\frac{\partial}{\partial v_q }\left( \varphi_n    \right)g^{pq}(v)\\
&+\varphi_n\frac{\partial \varphi_i}{\partial v_{n+1}}\frac{\partial}{\partial v_{n+1} }\left(h(v_{n+1},\tau)    \right)g^{v_{n+1}v_{n+1}}(v)\\
&=g^{i(j-1)}(\varphi)+h(v_{n+1},\tau)g^{in}(\varphi_n)+\frac{h^{\prime}(v_{n+1},\tau)}{n(n+1)}\left( \varphi_i+g(v_{n+1},\tau)\varphi_n\right )\varphi_n.
\end{split}
\end{equation}
If $j=1$, using equation (\ref{derivative of varphi1 wrt vn1})
\begin{equation}
\begin{split}
\Gamma_{v_{n+1}}^{i1}(\varphi)&=\frac{\partial \varphi_i}{\partial v_p}\frac{\partial}{\partial v_q }\left(\frac{\partial \varphi_1}{\partial  v_{n+1}}\right)g^{pq}(v)\\
&=\frac{\partial \varphi_i}{\partial v_p}\frac{\partial}{\partial v_q }\left( n\varphi_{0}+h(v_{n+1},\tau )\varphi_1    \right)g^{pq}(v)\\
&=\frac{\partial \varphi_i}{\partial v_p}\frac{\partial}{\partial v_q }\left( n\varphi_0+h(v_{n+1},\tau)\varphi_1    \right)g^{pq}(v)\\
&=\frac{\partial \varphi_i}{\partial v_p}\frac{\partial}{\partial v_q }\left( n\varphi_{0} \right)g^{pq}(v)+h(v_{n+1},\tau)\frac{\partial \varphi_i}{\partial v_p}\frac{\partial}{\partial v_q }\left( \varphi_1    \right)g^{pq}(v)\\
&+\varphi_1\frac{\partial \varphi_i}{\partial v_{n+1}}\frac{\partial}{\partial v_{n+1} }\left(h(v_{n+1},\tau)    \right)g^{v_{n+1}v_{n+1}}(v)\\
&=ng^{i0}(\varphi)+h(v_{n+1},\tau)g^{i1}(\varphi_1)+\frac{h^{\prime}(v_{n+1},\tau)}{n(n+1)}\left( \varphi_i+h_3(v_{n+1},\tau)\varphi_n\right )\varphi_1.
\end{split}
\end{equation}
If $j=0$, using equation (\ref{derivative of varphi0 wrt vn1})
\begin{equation}
\begin{split}
\Gamma_{v_{n+1}}^{i0}(\varphi)&=\frac{\partial \varphi_i}{\partial v_p}\frac{\partial}{\partial v_q }\left(\frac{\partial \varphi_0}{\partial  v_{n+1}}\right)g^{pq}(v)\\
&=\frac{\partial \varphi_i}{\partial v_p}\frac{\partial}{\partial v_q }\left( -nh_1(v_{n+1},\tau)\varphi_{0}+h_2(v_{n+1},\tau )\varphi_1    \right)g^{pq}(v)\\
&=\frac{\partial \varphi_i}{\partial v_p}\frac{\partial}{\partial v_q }\left( -nh_1(v_{n+1},\tau)\varphi_0+h_2(v_{n+1},\tau)\varphi_1    \right)g^{pq}(v)\\
&=\frac{\partial \varphi_i}{\partial v_p}\frac{\partial}{\partial v_q }\left( -nh_1(v_{n+1},\tau)\varphi_{0} \right)g^{pq}(v)+h_2(v_{n+1},\tau)\frac{\partial \varphi_i}{\partial v_p}\frac{\partial}{\partial v_q }\left( \varphi_1    \right)g^{pq}(v)\\
&+\varphi_1\frac{\partial \varphi_i}{\partial v_{n+1}}\frac{\partial}{\partial v_{n+1} }\left(h_2(v_{n+1},\tau)    \right)g^{v_{n+1}v_{n+1}}(v)\\
&=-nh_1(v_{n+1},\tau)g^{i0}(\varphi)+h_2(v_{n+1},\tau)g^{i1}(\varphi_n)+\frac{h_2^{\prime}(v_{n+1},\tau)}{n(n+1)}\left( \varphi_i+h_3(v_{n+1},\tau)\varphi_n\right )\varphi_1\\
&-n\frac{h_1^{\prime}(v_{n+1},\tau)}{n(n+1)}\left( \varphi_i+h_3(v_{n+1},\tau)\varphi_n\right )\varphi_0.
\end{split}
\end{equation}
Hence $\Gamma_{v_{n+1}}^{ij}(\varphi) \in  \widetilde E_{\bullet,\bullet}[\varphi_0,\varphi_1,.,\varphi_n]$, furthermore, it is at most linear on $\varphi_0$.

\end{proof}

\begin{proposition}\label{Levi Civita proposition}
The Christoffel symbols $\Gamma_k^{ij}(\varphi)$ belong to the ring $\widetilde E_{\bullet,\bullet}[\varphi_0,\varphi_1,..,\varphi_n]$.
\end{proposition} 
\begin{proof}
Note that the invariance of the Jacobi form $\varphi_i$ with respect the first two actions of (\ref{jacobigroupAntilde}), and equivariance by the third one implies that the differential $d\varphi_i$ is invariant under the first two actions of (\ref{jacobigroupAntilde}), and behaves as follows under the $SL_2(\mathbb{Z})$
\begin{equation}
\begin{split}
d\varphi_i\mapsto \frac{d\varphi_i}{(c\tau+d)^{k_i}}-\frac{c\varphi_i}{(c\tau+d)^{k_i+1}}
\end{split}
\end{equation}
Therefore the Christoffel symbol $\Gamma^{ij}_k$
\begin{equation}
\nabla_{(d\varphi_i)^{\#}}d\varphi_j=\Gamma_k^{ij}d\varphi_k
\end{equation}
is a Jacobi form if $\varphi_i$ has weight $0$. Hence, doing the change of coordinates
\begin{equation}
\varphi_i\mapsto\hat\varphi_i:=\eta^{2i}(\tau)\varphi_i,
\end{equation}
we have that the Christoffel symbol $\hat \Gamma^{ij}_k$
\begin{equation}
\frac{1}{\eta^{2i+2j}}\nabla_{(d\hat\varphi_i)^{\#}}d\hat\varphi_j=\hat\Gamma_k^{ij}d\hat\varphi_k
\end{equation}
is a Jacobi form.\\

Comparing $\hat\Gamma_k^{ij}$ with $\Gamma_k^{ij}$
\begin{equation}\label{big equation Levi Civita}
\begin{split}
\nabla_{(d\hat\varphi_j)^{\#}}d\hat\varphi_i&=\nabla_{\left(2jg_1\eta^{2j}\varphi_jd\tau+\eta^{2j}d\varphi_j\right)^{\#}}\left(2ig_1\eta^{2i}\varphi_id\tau+\eta^{2i}d\varphi_i    \right)\\
&=\nabla_{\left(2jg_1\eta^{2j}\varphi_jd\tau\right)^{\#}}\left(2ig_1\eta^{2i}\varphi_id\tau    \right)+\nabla_{\left(2jg_1\eta^{2j}\varphi_jd\tau\right)^{\#}}\left(\eta^{2i}d\varphi_i    \right)\\
&+\nabla_{\left(\eta^{2j}d\varphi_j\right)^{\#}}\left(2ig_1\eta^{2i}\varphi_id\tau   \right)+\nabla_{\left(\eta^{2j}d\varphi_j\right)^{\#}}\left(\eta^{2i}d\varphi_i    \right)\\
&=2jg_1\eta^{2j}\varphi_jg^{l\tau}\nabla_{\frac{\partial}{\partial\varphi_l}}\left(2i\eta^{2i}g_1\varphi_id\tau    \right)+2jg_1\eta^{2j}\varphi_jg^{l\tau}\nabla_{\frac{\partial}{\partial\varphi_l}}\left(\eta^{2i}d\varphi_i    \right)\\
&+\eta^{2j}g^{lj}\nabla_{\frac{\partial}{\partial\varphi_l}}\left(2ig_1\eta^{2i}\varphi_id\tau   \right)+\eta^{2j}g^{lj}\nabla_{\frac{\partial}{\partial\varphi_l}}\left(\eta^{2i}d\varphi_i    \right)\\
&=4ijg_1^{\prime}g_1\varphi_i\eta^{2i+2j}\varphi_jg^{\tau\tau}d\tau  +4i^2 jg_1^3\eta^{2i+2j}\varphi_j\varphi_ig^{\tau\tau}d\tau+4ijg_1\eta^{2i+2j}\varphi_jg^{i\tau}d\tau\\
&+4ijg_1^2\eta^{2i+2j}\varphi_jg^{\tau\tau}d\varphi_i  +2jg_1\eta^{2i+2j}\varphi_j\Gamma^{\tau i}_{k}d\varphi_k+4i^2g_1^2\eta^{2i+2j}\varphi_ig^{\tau j}d\tau \\
&+2ig_1^{\prime}\eta^{2i+2j}\varphi_ig^{lj}d\tau +2ig_1\eta^{2i+2j}g^{ij}d\tau +2ig_1\varphi_i\eta^{2i+2j}\Gamma^{i\tau}_kd\varphi_k\\
&+\eta^{2i+2j}g_1g^{l\tau}d\varphi_i+\eta^{2i+2j}\Gamma_k^{ji}d\varphi_k.
\end{split}
\end{equation}
Dividing the equation (\ref{big equation Levi Civita}) by $\eta^{2i+2j}$ and isolating $\Gamma_k^{ji}d\varphi_k$, we have
\begin{equation}\label{ comparison between Christoffel symbols}
\begin{split}
\Gamma_k^{ji}d\varphi_k&=-4ijg_1^{\prime}g_1\varphi_i\varphi_jg^{\tau\tau}d\tau -4i^2 jg_1^3\varphi_j\varphi_ig^{\tau\tau}d\tau+4ijg_1\eta^{2i+2j}\varphi_jg^{i\tau}d\tau\\
&-4ijg_1^2\varphi_jg^{\tau\tau}d\varphi_i  -2jg_1\varphi_j\Gamma^{\tau i}_{k}d\varphi_k-4i^2g_1^2\varphi_ig^{\tau j}d\tau \\
&-2ig_1^{\prime}\varphi_ig^{lj}d\tau-2ig_1\eta^{2i+2j}g^{ij}d\tau -2ig_1\varphi_i\Gamma^{i\tau}_kd\varphi_k\\
&-g_1g^{l\tau}d\varphi_i+\hat\Gamma_k^{ji}d\varphi_k.
\end{split}
\end{equation}
Since the differential forms has a vector space structure and the right hand side of  (\ref{ comparison between Christoffel symbols}) depends only on $g^{ij},g_1(\tau),\varphi_i,$ and $\Gamma_k^{\tau i}$ which belongs to the ring $\widetilde E_{\bullet,\bullet}[\varphi_0,\varphi_1,..,\varphi_n]$, the desired result is proved.

\end{proof}

\begin{lemma}\label{Levi Civita connection flat pencil}
The Christoffel symbols $\Gamma_k^{ij}(\varphi)$  depend linearly on $\varphi_0$.
\end{lemma}
\begin{proof}
The result was already proved, when $k=v_{n+1}$ in the lemma \ref{Levi Civita vn1}. The proposition \ref{Levi Civita proposition} gives to the space of Christoffel symbols the structure of graded algebra, in particular we can compute the degree $m$ regarding to the algebra of Jacobi forms.  Let $\phi \in \widetilde E_{\bullet,\bullet}[\varphi_0,\varphi_1,..,\varphi_n]$ with index $m_{\phi}$ and weight $k_{\phi}$, then we write
\begin{equation}
\begin{split}
deg_m\phi&=m_{\phi}, \\
deg_k\phi&=k_{\phi}. \\
\end{split}
\end{equation}
If $k\neq \tau,v_{n+1}$,
\begin{equation}
\begin{split}
deg_m \Gamma_k^{ij}=deg_m  \left(\frac{\partial \varphi_i}{\partial v_p}\frac{\partial}{\partial \varphi_k}\left(\frac{\partial \varphi_j}{\partial v_q}\right)g^{pq}(v)     \right)=1.
\end{split}
\end{equation}
Therefore, $\Gamma_k^{ij}$ is at most linear on $\varphi_0$. If $k=\tau$,
\begin{equation*}
\begin{split}
deg_k \Gamma_\tau^{ij}=deg_k  \left(\frac{\partial \varphi_i}{\partial v_p}\frac{\partial}{\partial\tau}\left(\frac{\partial \varphi_j}{\partial v_q}\right)g^{pq}(v)     \right)=-i-j+4.
\end{split}
\end{equation*}
Suppose that $\Gamma_\tau^{ij}$ contains a the term $a(v_{n+1},\tau)\varphi_0^2$, where $a(v_{n+1},\tau)$ is an elliptic function in $v_{n+1}$, then
\begin{equation*}
\begin{split}
deg_k a(v_{n+1},\tau)=-i-j+4>0.
\end{split}
\end{equation*}
The possible Christoffel symbols that could depend on $\varphi_0^2$ are
\begin{equation}\label{Levi Civita connection to be computed}
\begin{split}
\Gamma_{\tau}^{04}, \Gamma_{\tau}^{40},\Gamma_{\tau}^{13},\Gamma_{\tau}^{31}, \Gamma_{\tau}^{22}, \Gamma_{\tau}^{21}, \Gamma_{\tau}^{12},\Gamma_{\tau}^{20},\Gamma_{\tau}^{02},\Gamma_{\tau}^{11}, \Gamma_{\tau}^{10}, \Gamma_{\tau}^{01}, \Gamma_{\tau}^{00}.
\end{split}
\end{equation}
But $\Gamma_{\tau}^{22}, \Gamma_{\tau}^{11}, \Gamma_{\tau}^{00}$ is linear on $\varphi_0$ due to lemma \ref{diagonal Christoffel symbol}.\\
Computing $\Gamma_{\tau}^{ij}$
\begin{equation}\label{Christoffel symbol tau 1}
\begin{split}
\Gamma_{\tau}^{ij}&=\frac{\partial \varphi_i}{\partial v_p}\frac{\partial}{\partial \tau}\left(\frac{\partial \varphi_j}{\partial v_q}\right)g^{pq}(v)=\frac{\partial \varphi_i}{\partial v_p}\frac{\partial}{\partial v_q}\left(\frac{\partial \varphi_j}{\partial \tau}\right)g^{pq}(v)\\
\end{split}
\end{equation}
In order to compute it, recall there exist a relationship between the holomorphic Jacobi forms of $A_{n+1}$ type and the meromorphic Jacobi forms of $\tilde A_n$ type given by (\ref{relationship between Jacobi forms}). Moreover, in (\ref{relation Coxeter Bertola Jacobi forms}) it was demonstrated that the lowest degree term in the Taylor expansion of $\varphi_{i}^{\Ja(A_{n+1})}$ with respect the variables $v_0,v_1,..,v_{n+1}$ are the elementary symmetric polynomials $a_i(v_0,v_1,..,v_{n+1})$ of degree $i$.
Hence, using the equations (\ref{relation Coxeter Bertola Jacobi forms}) and (\ref{relationship between Jacobi forms}), we can estimate the degree of the lowest degree term of the meromorphic Jacobi forms of $\tilde A_n$, more specifically,
\begin{equation}\label{lowest degree taylor expansion of meromorphic jacobi forms}
\begin{split}
&\varphi_{n}^{\Ja(\tilde A_{n})}\varphi_{2}^{\Ja(A_{1})}=a_{n+2}(v)+b_{n+2}(v_{n+1},\tau)a_{n+3}(v)+O(||v||^{n+4}),\\
&\varphi_{n-1}^{\Ja(\tilde A_{n})}\varphi_{2}^{\Ja(A_{1})}+b_{n-1}^{n}\varphi_{n}^{\Ja(\tilde A_{n})}\varphi_{2}^{A_{1}}=a_{n+1}(v)+b_{n+1}(v_{n+1},\tau)a_{n+2}(v)+O(||v||^{n+3}),\\
&\varphi_{n-2}^{\Ja(\tilde A_{n})}\varphi_{2}^{\Ja(A_{1})}+b_{n-2}^{n-1}\varphi_{n-1}^{\Ja(\tilde A_{n})}\varphi_{2}^{\Ja(A_{1})}+a_{n-2}^n\varphi_{n}^{\Ja(\tilde A_{n})}\varphi_{2}^{\Ja(A_{1})}=a_{n}(v)+b_{n+2}(v_{n+1},\tau)a_{n+1}(v)+O(||v||^{n+2}),\\
.\\
.\\
&\varphi_{0}^{\Ja(\tilde A_{n})}\varphi_{2}^{\Ja(A_{1})}+\sum_{j=1}^{n}  b_0^j\varphi_{j}^{\Ja(\tilde A_{n})}\varphi_{2}^{\Ja(A_{1})}=a_{2}(v)+b_{3}(v_{n+1},\tau)a_{3}(v)+O(||v||^{4}).\\
\end{split}
\end{equation}
Note that the Christoffel symbol depend on $\varphi_0$ iff it contains the term $a_2^2(v)$ in its expansion. Our strategy is to show that the Christoffel symbols (\ref{Levi Civita connection to be computed}) contains only higher order polynomials in its expansions. Computing the lowest degree term in the expansion of (\ref{Christoffel symbol tau 1})
\begin{equation}\label{Christoffel symbol tau 2}
\begin{split}
\Gamma_{\tau}^{01}&=\frac{\partial \varphi_i}{\partial v_p}\frac{\partial}{\partial  v_q}\left(\frac{\partial \varphi_j}{\partial \tau}\right)g^{pq}(v)\\
&=\frac{\partial a_{i+2}}{\partial v_p}\frac{\partial}{\partial  v_q}\left( \frac{\partial b_{j+1}(v_{n+1},\tau)}{\partial\tau}a_{j+3}\right)g^{pq}(v)+..\\
&=\frac{\partial a_{i+2}}{\partial v_p} \frac{\partial b_{j+1}(v_{n+1},\tau)}{\partial\tau}\frac{\partial a_{j+3}}{\partial  v_q}g^{pq}(v)+..\\
&=\frac{\partial b_{j+1}(v_{n+1},\tau)}{\partial\tau}a_{i+j+3}+...
\end{split}
\end{equation}
Therefore, for $i+j>1$, we have that the associated Christoffel symbol do not depend on $\varphi_0^2$. It remains to check only $\Gamma_{\tau}^{01}$ and $\Gamma_{\tau}^{10}$. Computing $\Gamma_{\tau}^{10}$ by using the second equation of (\ref{Levi Civita contravariant coxeter chapter}) for $i=1,j=0 ,k=0$
\begin{equation*}
\begin{split}
\sum_{l=0}^{n+2}g^{1l}\Gamma_l^{00}&=\sum_{l=0}^{n+2}g^{0l}\Gamma_l^{10}\\
&=\sum_{l=0}^{n+1}g^{0l}\Gamma_l^{10}+g^{0\tau}\Gamma_{\tau}^{10}.
\end{split}
\end{equation*}
Isolating $\Gamma_{\tau}^{10}$,
\begin{equation}\label{equation Christoffel symbol almost last equation}
\begin{split}
\Gamma_{\tau}^{10}=\frac{1}{2\pi i\varphi_0}\left[\sum_{l=0}^{n+2}g^{1l}\Gamma_l^{00}- \sum_{l=0}^{n+1}g^{0l}\Gamma_l^{10}     \right],
\end{split}
\end{equation}
we have that the right hand side of (\ref{equation Christoffel symbol almost last equation}) depend at most linear on $\varphi_0$. Moreover, using the first equation (\ref{Levi Civita contravariant coxeter chapter}), we have
\begin{equation*}
\begin{split}
\partial_0^2\Gamma_{\tau}^{10}&=\partial_{\tau}\partial_0^2g^{10}-\partial_0^2\Gamma_{\tau}^{10}\\
&=\partial_{\tau}\partial_0^2g^{10}=0.
\end{split}
\end{equation*}
Lemma proved.

\end{proof}

\subsection{Unit and Euler vector field of the orbit space of $\Ja(\tilde A_n)$}

The aim of this section is to define the Unit and Euler vector field and its respective actions on the geometric data of the orbit space of $\Ja(\tilde A_n)$. Further, these objects will be fundamental to define the unit of the Frobenius algebra of the desired Dubrovin Frobenius structure, and to give a quasi homogeneous property to the desired WDVV solution.

\begin{definition}
The Euler vector field with respect the orbit space $\Ja(\tilde A_n)$ is defined by the last equation of (\ref{jacobiform}), i.e
\begin{equation}\label{Euler vector field jtildean}
E:=-\frac{1}{2\pi i}\frac{\partial}{\partial u}
\end{equation}
\end{definition}

\begin{definition}
A $f$ is quasi homogeneous of degree $d$ if it is an eigenfunction of the Euler vector field  (\ref{Euler vector field jtildean}) with eigenvalue $d$, i.e.
\begin{equation*}
E(f)=df.
\end{equation*}
\end{definition}

\begin{lemma}
Let $\lambda,\varphi_0,..,\varphi_n,\varphi_{n+1}=v_{n+1},\varphi_{n+2}=\tau$ be defined in (\ref{superpotentialAn}) and $(t^0,.,t^n,v_{n+1},\tau)$ the flat coordinates of eta defined in (\ref{generating function of flat of eta}). Then,
\begin{equation}\label{degree of }
\begin{split}
&E(\lambda)=\lambda,\\
&E(\varphi_i)=d_i\varphi_i,\\
&E(t^\alpha)=d_{\alpha}t^{\alpha},
\end{split}
\end{equation}
where 
\begin{equation}
\begin{split}
&d_i=1, \quad i<n+1,\\
&d_{i}=0, \quad i\geq n+1,\\
&d_{\alpha}=\frac{n+1-\alpha}{n}, \quad \alpha\neq 0,\\
&d_{0}=1.
\end{split}
\end{equation}
\end{lemma}
\begin{proof}
Recall that the function $\lambda$ is given by
\begin{equation*}
\begin{split}
\lambda&=e^{-2\pi i u}\frac{\prod_{i=0}^n \theta_1(z-v_i+v_{n+1},\tau)}{\theta_1^{n}(z,\tau)\theta_1(z+(n+1)v_{n+1})}\\
&=\varphi_n\wp^{n-2}(z,\tau)+\varphi_{n-1}\wp^{n-3}(z,\tau)+...+\varphi_{2}\wp(z,\tau)\\
&+\varphi_{1}[\zeta(z,\tau)-\zeta(z+(n+1)v_{n+1},\tau)+\varphi_0.
\end{split}
\end{equation*}
Hence, 
\begin{equation*}
\begin{split}
\lambda&=\frac{1}{2\pi i}\frac{\partial}{\partial u}(\lambda)\\
&=E(\varphi_n)\wp^{n-2}(z,\tau)+E(\varphi_{n-1})\wp^{n-3}(z,\tau)+...+E(\varphi_{2})\wp(z,\tau)\\
&+E(\varphi_{1})[\zeta(z,\tau)-\zeta(z+(n+1)v_{n+1},\tau)+E(\varphi_0),
\end{split}
\end{equation*}
therefore, $E(\varphi_i)=\varphi_i,$ and $E(v_{n+1})=E(\tau)=0.$

Recall that $t^{\alpha}$ can written in terms of equation (\ref{generating function of flat of eta 2}) or in more convenient way
\begin{equation}\label{flat coordinates written in a convenient way}
\begin{split}
&t^{\alpha}=\frac{n}{n+1-\alpha}\res_{v=0} \lambda^{\frac{n+1-\alpha}{n}}(v)dv, \quad \alpha\neq 0,\\
&t^0=\varphi_0-\frac{\theta_1^{\prime}((n+1)v_{n+1})}{\theta_1((n+1)v_{n+1})}\varphi_1+4\pi ig_1(\tau)\varphi_2.
\end{split}
\end{equation}
Applying the Euler vector in (\ref{flat coordinates written in a convenient way}) we get the desired result.

\end{proof}

\begin{corollary}
The Euler vector field (\ref{Euler vector field jtildean}) in the flat coordinates of $\eta^{*}$ has the following form
\begin{equation}\label{Euler vector field jtildean in the flat coordinates of eta}
E:=\sum_{\alpha=0}^n d_{\alpha}t^{\alpha}\frac{\partial}{\partial t^{\alpha}},
\end{equation}
where 
\begin{equation}
\begin{split}
d_{\alpha}&=\frac{n+1-\alpha}{n}, \quad \alpha\neq 0,\\
d_{0}&=1.
\end{split}
\end{equation}
\end{corollary}

\begin{proof}
Recall that 

\begin{equation*}
\begin{split}
E&=\frac{1}{2\pi i}\frac{\partial}{\partial u}=E(t^{\alpha})\frac{\partial}{\partial t^{\alpha}}=\sum_{\alpha=0}^n d_{\alpha}t^{\alpha}\frac{\partial}{\partial t^{\alpha}}.
\end{split}
\end{equation*}
\end{proof}

\begin{lemma}
The Euler vector field (\ref{Euler vector field jtildean}) acts on the vector fields $\frac{\partial}{\partial t^{\alpha}}$, $\frac{\partial}{\partial \varphi_i}$ and differential forms $dt^{\alpha}$, $d\varphi_i$ as follows:
\begin{equation}\label{degree vector field and forms}
\begin{split}
Lie_E d\varphi_i&=d_id\varphi_i,\\
Lie_E dt^{\alpha}&=d_{\alpha}dt^{\alpha},\\
Lie_E \frac{\partial}{\partial \varphi_i}&=-d_{i}\frac{\partial}{\partial \varphi_i},\\
Lie_E \frac{\partial}{\partial t^{\alpha}}&=-d_{\alpha}\frac{\partial}{\partial t^{\alpha}}.\\
\end{split}
\end{equation}
\end{lemma}
\begin{proof}
Recall that the Lie derivative acts in vector fields by using the Lie bracket and in differential forms by the use of Cartan's magic formula
\begin{equation}\label{Cartan Magic Formula and Lie bracket}
\begin{split}
Lie_E \frac{\partial}{\partial t^{\alpha}}&=\left[E,\frac{\partial}{\partial t^{\alpha}}   \right],\\
Lie_E dt^{\alpha}&=d E(d t^{\alpha})+E(d^2 t^{\alpha})=d E(d t^{\alpha}).\\
\end{split}
\end{equation}
Using (\ref{Cartan Magic Formula and Lie bracket}) and (\ref{degree of }), we obtain the desired result.

\end{proof}

\begin{lemma}
The intersection form $g^{ij}$  defined in (\ref{metrich1n}) and its Christoffel symbol $\Gamma_{k}^{ij}$ in the coordinates $\varphi_0,..,\varphi_n,\varphi_{n+1}=v_{n+1},\varphi_{n+2}=\tau$ defined in (\ref{superpotentialAn}) are weighted polynomials in the variables $\varphi_0,..,\varphi_n,$ with degrees
\begin{equation}
\begin{split}
deg\left( g^{ij} \right)&=d_{i}+d_{j},\quad deg\left( \Gamma_{k}^{\alpha\beta} \right)=d_{i}+d_j-d_{k}.
\end{split}
\end{equation}

\end{lemma}
\begin{proof}
The function $g^{ij}$ and $\Gamma_{k}^{ij}$ belong to the ring $\widetilde E_{\bullet,\bullet}[\varphi_0,\varphi_1,..,\varphi_n]$ due to \ref{intersection form and M} and \ref{Levi Civita proposition}. The degrees are computed by using the following formulae
\begin{equation*}
\begin{split}
E\left( g^{ij}(\varphi) \right)&=E\left( \frac{\partial \varphi_i}{\partial v^l}\frac{\partial\varphi_j}{\partial v^m}g^{lm}(v) \right)\\
&=E\left( \frac{\partial \varphi_i}{\partial v^l} \right)\frac{\partial\varphi_j}{\partial v^m}g^{lm}(v)+\frac{\partial \varphi_i}{\partial v^l}E\left( \frac{\partial\varphi_j}{\partial v^m}g^{lm}(v) \right)\\
&=\frac{\partial E(\varphi_i)}{\partial v^l} \frac{\partial\varphi_j}{\partial v^m}g^{lm}(v)+\frac{\partial \varphi_i}{\partial v^l}\frac{\partial E(\varphi_j)}{\partial v^m}g^{lm}(v) \\
&=(d_i+d_j)\frac{\partial \varphi_i}{\partial v^l}\frac{\partial\varphi_j}{\partial v^m}g^{lm}(v).
\end{split}
\end{equation*}
and 
\begin{equation*}
\begin{split}
E\left( \Gamma_k^{ij}(\varphi) \right)&=E\left( \frac{\partial \varphi_i}{\partial v^l}\frac{\partial}{\partial \varphi_k}\left(\frac{\partial\varphi_j}{\partial v^m}\right)g^{lm}(v) \right)\\
&=E\left( \frac{\partial \varphi_i}{\partial v^l} \right)\frac{\partial}{\partial \varphi_k}\left(\frac{\partial\varphi_j}{\partial v^m}\right)g^{lm}(v)+\frac{\partial \varphi_i}{\partial v^l}E\left( \frac{\partial}{\partial \varphi_k}\left(\frac{\partial\varphi_j}{\partial v^m}\right)g^{lm}(v) \right)\\
&=\frac{\partial E(\varphi_i)}{\partial v^l} \frac{\partial}{\partial \varphi_k}\left(\frac{\partial\varphi_j}{\partial v^m}\right)g^{lm}(v)+\frac{\partial \varphi_i}{\partial v^l}\frac{\partial}{\partial \varphi_k}\left(\frac{\partial E(\varphi_j)}{\partial v^m}\right)g^{lm}(v)\\
&-d_k\frac{\partial \varphi_i}{\partial v^l}\frac{\partial}{\partial \varphi_k}\left(\frac{\partial\varphi_j}{\partial v^m}\right)g^{lm}(v) \\
&=(d_i+d_j-d_k)\frac{\partial \varphi_i}{\partial v^l}\frac{\partial}{\partial \varphi_k}\left(\frac{\partial\varphi_j}{\partial v^m}\right)g^{lm}(v).
\end{split}
\end{equation*}

\end{proof}

\begin{lemma}
The intersection form $g^{\alpha\beta}$ defined in (\ref{metrich1n})  in the coordinates $(t^0,.,t^n,v_{n+1},\tau)$ defined in (\ref{generating function of flat of eta}) and its Christoffel symbol $\Gamma_{\gamma}^{\alpha\beta}$ are weighted polynomials in the variables $t^0,t^1,..,t^n,\frac{1}{t^n}$ with degrees
\begin{equation}\label{degree Christoffel flat eta}
\begin{split}
deg\left( g^{\alpha\beta} \right)&=d_{\alpha}+d_{\beta},\\
deg\left( \Gamma_{\gamma}^{\alpha\beta} \right)&=d_{\alpha}+d_{\beta}-d_{\gamma}.\\
\end{split}
\end{equation}
\end{lemma}

\begin{proof}
Lemma \ref{lemma gab index} guarantee that $g^{\alpha\beta} \in \widetilde E_{\bullet,\bullet}[t^0,t^1,..,t^n,\frac{1}{n}]$. Using the formula
\begin{equation*}
\begin{split}
E(g^{\alpha\beta})&=E(\frac{\partial t^{\alpha}}{\partial \varphi_i}\frac{\partial t^{\beta}}{\partial \varphi_j}g^{ij}(\varphi) )\\
&=E(\frac{\partial t^{\alpha}}{\partial \varphi_i} )\frac{\partial t^{\beta}}{\partial \varphi_j}g^{ij}(\varphi)+\frac{\partial t^{\alpha}}{\partial \varphi_i}E(\frac{\partial t^{\beta}}{\partial \varphi_j} )g^{ij}(\varphi)+\frac{\partial t^{\alpha}}{\partial \varphi_i}\frac{\partial t^{\beta}}{\partial \varphi_j}E(g^{ij}(\varphi) )\\
&=\frac{\partial E(t^{\alpha})}{\partial \varphi_i} \frac{\partial t^{\beta}}{\partial \varphi_j}g^{ij}(\varphi)-d_i\frac{\partial t^{\alpha}}{\partial \varphi_i} \frac{\partial t^{\beta}}{\partial \varphi_j}g^{ij}(\varphi)+\frac{\partial t^{\alpha}}{\partial \varphi_i}\frac{\partial E(t^{\beta})}{\partial \varphi_j} g^{ij}(\varphi)-d_j\frac{\partial t^{\alpha}}{\partial \varphi_i} \frac{\partial t^{\beta}}{\partial \varphi_j}g^{ij}(\varphi)\\
&+(d_i+d_j)\frac{\partial t^{\alpha}}{\partial \varphi_i}\frac{\partial t^{\beta}}{\partial \varphi_j}g^{ij}(\varphi) \\
&=(d_{\alpha}+d_{\beta})\frac{\partial t^{\alpha}}{\partial \varphi_i}\frac{\partial t^{\beta}}{\partial \varphi_j}g^{ij}(\varphi).
\end{split}
\end{equation*}
The Christoffel symbol $\Gamma^{\alpha\beta}_{\gamma}$ is given by the following
\begin{equation*}
\begin{split}
 \Gamma^{\alpha\beta}_{\gamma}&=     \frac{\partial t^{\alpha}}{\partial \varphi_i}\frac{\partial t^{\beta}}{\partial \varphi_j}\frac{\partial \varphi_k}{\partial t^{\gamma}} \Gamma^{ij}_{k}+   \frac{\partial t^{\alpha}}{\partial \varphi_i} \frac{\partial }{\partial t^{\gamma}}\left(\frac{\partial t^{\beta}}{\partial \varphi_j}\right) g^{ij} .               
\end{split}
\end{equation*}
$\Gamma_{j}^{ij}, \frac{\partial \varphi_k}{\partial t^{\gamma}} \in \widetilde E_{\bullet,\bullet}[t^0,t^1,..,t^n]$ due to Lemma \ref{Levi Civita proposition} and equations (\ref{derivative of varphi wrt t}), (\ref{ derivative T}). But using (\ref{relation t wrt varphi 1}), we realise that $ \frac{\partial t^{\alpha}}{\partial \varphi_i}$ is polynomial in $t^{1},t^{2},..,t^n,\frac{1}{t^n}$ due to due to equations (\ref{formula varphi 3}) and (\ref{definition of formal powers of varphi}). Therefore, $\Gamma_{\gamma}^{\alpha\beta}$ are weighted polynomials in the variables $t^0,t^1,..,t^n,\frac{1}{t^n}$. Computing the degree of $\Gamma_{\gamma}^{\alpha\beta}$

\begin{equation*}
\begin{split}
E( \Gamma^{\alpha\beta}_{\gamma})&=E(      \frac{\partial t^{\alpha}}{\partial \varphi_i}\frac{\partial t^{\beta}}{\partial \varphi_j}\frac{\partial \varphi_k}{\partial t^{\gamma}} \Gamma^{ij}_{k}+   \frac{\partial t^{\alpha}}{\partial \varphi_i} \frac{\partial }{\partial t^{\gamma}}\left(\frac{\partial t^{\beta}}{\partial \varphi_j}\right) g^{ij}                      )\\
&=  (d_{\alpha}-d_i)   \frac{\partial t^{\alpha}}{\partial \varphi_i}\frac{\partial t^{\beta}}{\partial \varphi_j}\frac{\partial \varphi_k}{\partial t^{\gamma}} \Gamma^{ij}_{k}+ 
(d_{\beta}-d_j)   \frac{\partial t^{\alpha}}{\partial \varphi_i}\frac{\partial t^{\beta}}{\partial \varphi_j}\frac{\partial \varphi_k}{\partial t^{\gamma}} \Gamma^{ij}_{k}\\
&+(d_{k}-d_{\gamma})   \frac{\partial t^{\alpha}}{\partial \varphi_i}\frac{\partial t^{\beta}}{\partial \varphi_j}\frac{\partial \varphi_k}{\partial t^{\gamma}} \Gamma^{ij}_{k}
+(d_{\alpha}-d_i) \frac{\partial t^{\alpha}}{\partial \varphi_i} \frac{\partial }{\partial t^{\gamma}}\left(\frac{\partial t^{\beta}}{\partial \varphi_j}\right) g^{ij} \\
 & +(d_{\beta}-d_{\gamma}) \frac{\partial t^{\alpha}}{\partial \varphi_i} \frac{\partial }{\partial t^{\gamma}}\left(\frac{\partial t^{\beta}}{\partial \varphi_j}\right) g^{ij} 
 + (d_{\alpha}-d_i) \frac{\partial t^{\alpha}}{\partial \varphi_i} \frac{\partial }{\partial t^{\gamma}}\left(\frac{\partial t^{\beta}}{\partial \varphi_j}\right) g^{ij} \\
 & +(d_{i}+d_j) \frac{\partial t^{\alpha}}{\partial \varphi_i} \frac{\partial }{\partial t^{\gamma}}\left(\frac{\partial t^{\beta}}{\partial \varphi_j}\right) g^{ij}\\
 &=(d_{\alpha}+d_{\beta}-d_{\gamma})\Gamma^{\alpha\beta}_{\gamma}.            
\end{split}
\end{equation*}

\end{proof}

\begin{definition}
The Unit vector field with respect the orbit space $\Ja(\tilde A_n)$ is the vector associated to the invariant coordinate $\varphi_0$ defined in (\ref{superpotentialAn}) , i.e
\begin{equation}\label{unit vector field jtildean}
e:=\frac{\partial}{\partial \varphi_0}.
\end{equation}
\end{definition}

\begin{lemma}
The Unit vector field (\ref{Euler vector field jtildean}) in the flat coordinates of $\eta^{*}$ has the following form
\begin{equation}\label{unit vector field jtildean in the flat coordinates of eta}
e=\frac{\partial}{\partial t^{0}}.
\end{equation}
\end{lemma}

\begin{proof}
\begin{equation*}
\begin{split}
\frac{\partial}{\partial \varphi_0}=\frac{\partial t^{\alpha}}{\partial \varphi_0}\frac{\partial}{\partial t^{\alpha}}=\frac{\partial}{\partial t^{0}}.
\end{split}
\end{equation*}

\end{proof}

\begin{lemma}\label{degree eta}
Let the metric $\eta^{*}$ be defined on (\ref{metric eta def}) and the Euler vector field (\ref{Euler vector field jtildean}). Then,
\begin{equation}
Lie_E \eta^{\alpha\beta}=(d_{\alpha}+d_{\beta}-d_1)\eta^{\alpha\beta}.
\end{equation}
\end{lemma}

\subsection{Discriminant locus and the monodromy of the orbit space of $\Ja(\tilde A_n)$}

This section relates the critical points of the (\ref{superpotentialAn}) with the zeros of the determinant of the intersection form $g^{*}$ (\ref{metrich1n cotangent}). Further, in section \ref{Construction of WDVV solution}, it will be built a Frobenius algebra in the sections of the orbit space of $\Ja(\tilde A_n)$, furthermore, the intersection form $g^{*}$ (\ref{metrich1n cotangent}) can be realised as the multiplication by the Euler vector field (\ref{Euler vector field jtildean}). The results of this section will imply that the intersection form $g^{*}$ (\ref{metrich1n cotangent}) is diagonalisable with eigenvalues generically different from 0, and this is equivalent to the Frobebius algebra be semisimple. Moreover, we can realise the isomorphism of orbit space of $\Ja(\tilde A_n)$ with the Hurwitz space $H_{1,n-1,0}$ as a Dubrovin Frobenius manifolds, see Theorem \ref{Dubrovin Frobenius structure of the orbit space Mirror symmetry}  for details.

\begin{definition}
Let $g^{*}$ the metric defined on (\ref{metrich1n cotangent}). The discriminant locus of the orbit space of $\Ja(\tilde A_n)$ $\mathbb{C}\oplus\mathbb{C}^{n+1}\oplus\mathbb{H}/\Ja(\tilde A_n)$ is defined by
\begin{equation}\label{discriminat locus}
\Sigma=\{ x \in  \mathbb{C}\oplus\mathbb{C}^{n+1}\oplus\mathbb{H}/\Ja(\tilde A_n): \text{det}(g^{*})=0      \}.
\end{equation}
\end{definition}

\begin{lemma}
The fixed points of the action $\Ja(\tilde A_n)$ belong to the discriminant locus (\ref{discriminat locus}).
\end{lemma}
\begin{proof}
Note that the fixed points of the action $\Ja(\tilde A_n)$ on $\mathbb{C}\oplus\mathbb{C}^{n+1}\oplus\mathbb{H}\ni (u,v_0,v_1,..,v_{n-1},v_{n+1},\tau)$  are the fixed points of the action $A_n$ on $\mathbb{C}^n \ni (v_0,v_1,..,v_{n_1})$. Therefore, the fixed points are the hyperplanes 
\begin{equation}\label{discriminant locus}
v_i=v_j \quad i,j\in \{0,1,..,n-1\}.
\end{equation}
The intersection form (\ref{metrich1n cotangent})  is given by
\begin{equation}
\begin{split}
g&=\sum_{i=0}^n \left.dv_{i}^2\right|_{\sum_{i=0}^n v_i=0}-n(n+1)dv_{n+1}^2+2du d\tau\\
&=\sum_{i,j=0}^{n-1} A_{ij}dv_idv_j-n(n+1)dv_{n+1}^2+2du d\tau\\
\end{split}
\end{equation}
The intersection form is given by 
\begin{equation}
\begin{split}
g^{*}&=\sum_{i,j=0}^{n-1} A_{ij}^{-1}\frac{\partial}{\partial v_i}\otimes \frac{\partial}{\partial v_j}-\frac{1}{n(n+1)}\frac{\partial}{\partial v_{n+1}}\otimes \frac{\partial}{\partial v_{n+1}}+\frac{\partial}{\partial u}\otimes\frac{\partial}{\partial \tau}+\frac{\partial}{\partial \tau}\otimes\frac{\partial}{\partial u}\\
\end{split}
\end{equation}
became degenerate on the hyperplanes (\ref{discriminant locus}), because two columns of the matrix $A_{ij}^{-1}$ became proportional.

\end{proof}

\begin{lemma}\label{superpotential vs discriminant locus}
The function $\lambda(p,u,v_0,v_1,..,v_{n+1},\tau)$  defined on (\ref{superpotentialAn})  has simple critical points if and only if $(u,v_0,v_1,..,v_{n+1},\tau)$ is a fixed point of the action $\Ja(\tilde A_n)$.
\end{lemma}
\begin{proof}
Using the local isomorphism given by \ref{lambda0}
\begin{equation}
[(u,v_0,v_1,..,v_{n-1},v_{n+1},\tau)]\longleftrightarrow \lambda(v)=e^{-2\pi i u}\frac{\prod_{i=0}^{n}\theta_1(z-v_i,\tau)}{\theta_1^{n}(v,\tau)\theta_1(v+(n+1)v_{n+1},\tau)},
\end{equation}
we can realise the discriminant locus as the space of parameters of $\lambda(p,u,v_0,v_1,..,v_{n+1},\tau)$ such that $\lambda(p,u,v_0,v_1,..,v_{n+1},\tau)$ has repeated roots. In these cases $\lambda(p,u,v_0,v_1,..,v_{n+1},\tau)$ has non simple critical points.

\end{proof}

\begin{definition}
The canonical coordinates $(u_1,u_2,..,u_{n+2})$ of the orbit space $\Ja(\tilde A_n)$ is given by the following relation
\begin{equation}
\begin{split}
\lambda(q_i)&=u_i,\\
\lambda^{\prime}(q_i)&=0.
\end{split}
\end{equation}
\end{definition}

\begin{lemma}\label{distinct roots}
The determinant of the intersection form $g^{*}$ defined on (\ref{metrich1n}) is proportional to $\prod_{i=1}^{n+2}u_i$.

\end{lemma}

\begin{proof}
If $u_i=\lambda(q_i,u,v_0,..,v_{n+1},\tau)=0$, we have that det$g^{*}=0$ due to lemma (\ref{superpotential vs discriminant locus}), then $u_i$ are zeros of the equation det$g^{*}=0$.

\end{proof}

\begin{proposition}\label{same Euler, unit, and intersection form}
In the canonical coordinates $(u_1,u_2,..,u_{n+2})$ the unit vector field (\ref{unit vector field jtildean}), the Euler vector field (\ref{Euler vector field jtildean}), and the intersection form (\ref{metrich1n})  have the following form
\begin{equation}
\begin{split}
g^{ii}&=u^i\eta^{ii}\delta_{ij},\\
e&=\sum_{i=1}^{n+2}\frac{\partial}{\partial u_i},\\
E&=\sum_{i=1}^{n+2} u_i\frac{\partial}{\partial u_i}.
\end{split}
\end{equation}
where $\eta^{ii}$ are the coefficients  of second metric $\eta^{*}$ in canonical coordinates.
\end{proposition}
\begin{proof}
Note that $g^{*}$ is diagonalisable with distinct  eigenvalues   if the following equation
\begin{equation}\label{equation det An}
det(\eta_{\alpha\mu}g^{\mu\beta}-u\delta_{\alpha}^{\beta})=0,
\end{equation}
has only simple roots. Since $det(\eta_{\alpha\mu})\neq 0$, the equation (\ref{equation det An}) is equivalent to
\begin{equation}\label{equation det 1 An}
det(g^{\alpha\beta}-u\eta^{\alpha\beta})=0.
\end{equation}
Using that $\eta^{\alpha\beta}=\partial_0g^{\alpha\beta}$, we have that
\begin{equation}\label{equation det 2 An}
det(g^{\alpha\beta}-u\eta^{\alpha\beta})=det \left ( g^{\alpha\beta}(t^0-u,t^1,t^2,t^3,..,t^n,v_{n+1},\tau \right)=0.
\end{equation}
Due to the lemma \ref{distinct roots} the equation (\ref{equation det 1 An}) has $n+2$ distinct roots 
\begin{equation}\label{canonical coordinates}
u^i=t^1-y^i(t^1,t^2,t^3,..,t^n,v_{n+1},\tau).
\end{equation}
In the coordinates $(u^1,u^2,..,u^{n+2})$ the matrix $g^{i}_j$ is diagonal, then
\begin{equation}\label{intersection form in canonical coordinates}
g^{ij}=u^i\eta^{ij}\delta_{ij}
\end{equation}
 and the unit vector field have the following form
\begin{equation}\label{canonical coordinates}
\frac{\partial}{\partial t^1}=\sum_{i=1}^{n+2} \frac{\partial u_i}{\partial t^1}\frac{\partial}{\partial u_i}=\sum_{i=1}^{n+2} \frac{\partial}{\partial u_i}
\end{equation}
Moreover, since 
\begin{equation}
[E,e]=[\sum_{\alpha=0}^{n}d_{\alpha}t^{\alpha}\frac{\partial}{\partial t^{\alpha}},\frac{\partial}{\partial t^1}]=-e,
\end{equation}
the Euler vector field in the coordinates $(u^1,u^2,..,u^{n+2})$ takes the following form
\begin{equation}\label{Euler vf in canonical coordinates}
E=\sum_{i=1}^{n+2} u^i\frac{\partial}{\partial u_i}.
\end{equation}
Lemma proved.

\end{proof}

\subsection{Construction of WDVV solution}\label{Construction of WDVV solution}
The main aim of this section is to extract a WDVV equation from the data of the group $\Ja(\tilde A_n)$.

\begin{lemma}
The orbit space of $\Ja(\tilde A_n)$ carries a flat pencil of metrics
\begin{equation}\label{flat pencil Jtildean}
\begin{split}
g^{\alpha\beta}, \quad \eta^{\alpha\beta}:=\frac{\partial g^{\alpha\beta}}{\partial t^0}
\end{split}
\end{equation}
with the correspondent Christoffel symbols.
\begin{equation}
\begin{split}
\Gamma_{\gamma}^{\alpha\beta}, \quad \eta^{\alpha\beta}:=\frac{\partial \Gamma_{\gamma}^{\alpha\beta}}{\partial t^0}
\end{split}
\end{equation}
\end{lemma}
\begin{proof}   
The metric (\ref{flat pencil Jtildean}) satisfies the hypothesis of Lemma \ref{flat pencil almost Dubrovin frobenius manifold Coxeter} which proves the desired result.

\end{proof}

\begin{lemma}\label{main final lemma last chapter 1}
Let the intersection form be (\ref{metrich1n}), unit vector field be (\ref{unit vector field jtildean}), and Euler vector field be (\ref{Euler vector field jtildean}). Then, there exist a  function 
\begin{equation}\label{WDVV solution}
F(t^0,t^1,t^2,..,t^n.v_{n+1},\tau)=-\frac{(t^0)^2\tau}{4\pi i}+\frac{t^0}{2}\sum_{\alpha,\beta\neq 0,\tau}\eta_{\alpha\beta}t^{\alpha}t^{\beta}+G(t^1,t^2,..,t^n,v_{n+1},\tau),
\end{equation}
such that
\begin{equation}\label{quasi homogeneous condition of Dubrovin Frobenius structure}
\begin{split}
&Lie_{E}F=2F+\text{quadratic terms},\\
&Lie_{E}\left(  F^{\alpha\beta}   \right)=g^{\alpha\beta},\\
&\frac{\partial ^2G(t^1,t^2,..,t^n,v_{n+1},\tau)}{\partial t^{\alpha}\partial t^{\beta}} \in \widetilde E_{\bullet,\bullet}[t^1,t^2,..,t^n,\frac{1}{t^n}],
\end{split}
\end{equation}
where 
\begin{equation}
F^{\alpha\beta}  =\eta^{\alpha\alpha^{\prime}}\eta^{\beta\beta^{\prime}}\frac{\partial F^2}{\partial t^{\alpha^{\prime}} \partial t^{\beta^{\prime}} }.
\end{equation}

\end{lemma}
\begin{proof}
Let $\Gamma_{\gamma}^{\alpha\beta}(t)$ the Christoffel symbol of the intersection form (\ref{metrich1n}) in the coordinates  the flat coordinates of $\eta^{*}$, i.e $t^0,t^1,t^2,..,t^n.v_{n+1},\tau$. According to the lemma \ref{flat pencil almost Dubrovin frobenius manifold Coxeter}, we can represent $\Gamma_{\gamma}^{\alpha\beta}(t)$ as 
\begin{equation}\label{potentiality for the Levi Civita connection}
\Gamma_{\gamma}^{\alpha\beta}(t)=\eta^{\alpha\epsilon}\partial_{\epsilon}\partial_{\gamma}f^{\beta}(t).
\end{equation}
Using  the relations (\ref{degree Christoffel flat eta}), (\ref{degree vector field and forms})  and lemma \ref{degree eta}
\begin{equation*}
\begin{split}
Lie_E(\Gamma_{\gamma}^{\alpha\beta}(t))&=Lie_E(\eta^{\alpha\epsilon})\partial_{\epsilon}\partial_{\gamma}f^{\beta}(t)+\eta^{\alpha\epsilon}Lie_E(\partial_{\epsilon}\partial_{\gamma}f^{\beta}(t))\\
&=(d_{\alpha}+d_{\epsilon}-d_1)\eta^{\alpha\epsilon}\partial_{\epsilon}\partial_{\gamma}f^{\beta}(t)+(-d_{\epsilon}-d_{\gamma})\eta^{\alpha\epsilon}\partial_{\epsilon}\partial_{\gamma}Lie_E(f^{\beta}(t))\\
&=(d_{\alpha}+d_{\beta}-d_{\gamma})\eta^{\alpha\epsilon}\partial_{\epsilon}\partial_{\gamma}f^{\beta}(t).
\end{split}
\end{equation*}
Then, by isolation $Lie_E\left(f^{\beta}(t)\right)$ we get
\begin{equation}\label{almost quasi homogeneous condition}
Lie_E\left(f^{\beta}(t)\right)=(d_{\beta}+d_1)f^{\beta}+A_{\sigma}^{\beta}t^{\sigma}+B^{\beta}, \quad A_{\sigma}^{\beta}, B^{\beta} \in \mathbb{C}.
\end{equation}
Considering the second relation of  (\ref{Levi Civita contravariant coxeter chapter}) for $\alpha=\tau$ 
\begin{equation}
\begin{split}
g^{\tau\sigma}\Gamma_{\sigma}^{\beta\gamma}&=g^{\beta\sigma}\Gamma_{\sigma}^{\tau\gamma},\\
\end{split}
\end{equation}
and  using lemma \ref{lemma tau varphi} and \ref{ Levi Civita connection tau}, we have.
\begin{equation}
\begin{split}
-2\pi i d_{\sigma}t^{\sigma}\eta^{\beta\epsilon}\partial_{\sigma}\partial_{\epsilon}f^{\gamma}&=-2\pi id_{\sigma}\delta^{\gamma}_{\sigma}g^{\beta\sigma},\\
\end{split}
\end{equation}
which is equivalent to 
\begin{equation}\label{equation of second condition levi Civita for tau}
\begin{split}
Lie_E\left(\eta^{\beta\epsilon}\partial_{\epsilon}f^{\gamma}\right)&=d_{\gamma}g^{\beta\gamma}.\\
\end{split}
\end{equation}
Using  (\ref{almost quasi homogeneous condition}) in the equation (\ref{equation of second condition levi Civita for tau}), we have
\begin{equation}\label{almost integrability condition}
\begin{split}
(d_{\beta}+d_{\gamma})\eta^{\beta\epsilon}\partial_{\epsilon}f^{\gamma}&=d_{\gamma}g^{\beta\gamma}.\\
\end{split}
\end{equation}
If $\gamma\neq v_{n+1},\tau$, we define 
\begin{equation}\label{definition of Fgamma}
F^{\gamma}=\frac{f^{\gamma}}{d_{\gamma}},
\end{equation}
 and note that $g^{\beta\gamma}$ is symmetric with respect the indices $\beta, \gamma$. Hence,
\begin{equation}\label{integrability condition}
\begin{split}
(d_{\beta}+d_{\gamma})\eta^{\beta\epsilon}\partial_{\epsilon}F^{\gamma}&=(d_{\beta}+d_{\gamma})\eta^{\gamma\epsilon}\partial_{\epsilon}F^{\beta},\\
\end{split}
\end{equation}
which is the integrability condition for
\begin{equation}\label{definition of F almost}
\begin{split}
F^{\gamma}=\eta^{\gamma\mu}\partial_{\mu}F.
\end{split}
\end{equation}
In order to extract information from $\gamma=\tau$, take $\beta=\tau$ in equation (\ref{almost integrability condition})
\begin{equation}
\begin{split}
d_{\gamma}\eta^{\tau 0}\partial_{0}f^{\gamma}&=d_{\gamma}g^{\tau\gamma}\\
-2\pi id_{\gamma}\partial_{0}f^{\gamma}&=-2\pi id_{\gamma}t^{\gamma}
\end{split}
\end{equation}
which is equivalent to
\begin{equation*}
\begin{split}
\eta^{\gamma\epsilon}\partial_{\epsilon}\partial_0F=t^{\gamma},
\end{split}
\end{equation*}
inverting $\eta^{\gamma\epsilon}$
\begin{equation}\label{first part of the Hessian}
\begin{split}
\partial_{\alpha}\partial_0F=\eta_{\alpha\gamma}t^{\gamma},
\end{split}
\end{equation}
integrating equation (\ref{first part of the Hessian}), we obtain
\begin{equation}\label{WDVV solution inside the proof}
F(t^0,t^1,t^2,..,t^n,v_{n+1},\tau)=-\frac{(t^0)^2\tau}{4\pi i}+\frac{t^0}{2}\sum_{\alpha,\beta\neq 0,\tau}\eta_{\alpha\beta}t^{\alpha}t^{\beta}+G(t^1,t^2,..,t^n,v_{n+1},\tau).
\end{equation}
Substituting the equation (\ref{WDVV solution inside the proof}) in the (\ref{almost integrability condition}) for $\gamma\neq v_{n+1},\tau$, we get
\begin{equation}\label{almost the second quasi homogeneous equation}
\begin{split}
g^{\beta\gamma}&=(d_{\beta}+d_{\gamma})\eta^{\beta\epsilon}\eta^{\gamma\mu}\partial_{\epsilon}\partial_{\mu}F,\\
&=Lie_E(F^{\beta\gamma})
\end{split}
\end{equation}
 Since $g^{\beta\gamma}$ is a symmetric matrix the equation (\ref{almost the second quasi homogeneous equation}) is equivalent to the second equation of (\ref{quasi homogeneous condition of Dubrovin Frobenius structure}) for either $\beta$ and $\gamma$ different from $v_{n+1},\tau$. Therefore, the missing part of the second equation of (\ref{quasi homogeneous condition of Dubrovin Frobenius structure}) is only for the cases $\beta=\gamma=v_{n+1}$ and $\beta=\gamma=\tau$. Moreover, the intersection form $g^{\beta\gamma}$ is proportional to the Hessian of  the equation  (\ref{WDVV solution inside the proof}) for for either $\beta$ and $\gamma$ different from $v_{n+1},\tau$. Recall that from the data of a Hessian, we can reconstruct uniquely a function up to quadratic terms, therefore, by defining 
 \begin{equation}\label{missing part of the hessian}
\begin{split}
Lie_E\left( \frac{\partial^2F}{\partial {t^1}^2}   \right)&=g^{v_{n+1}v_{n+1}},\\
Lie_E\left( \frac{\partial^2F}{\partial {t^1}^2}   \right)&=g^{\tau\tau}.
\end{split}
\end{equation}
Just the second equation of (\ref{missing part of the hessian}) need to be proved, since the first equation defines the coefficients of the Hessian $\frac{\partial^2F}{\partial {t^1}^2}  $, in another words, it defines $\frac{\partial^2G}{\partial {t^1}^2} $. The second equation must be compatible with the equation (\ref{WDVV solution inside the proof}), then substituting (\ref{WDVV solution inside the proof}) in the second equation of (\ref{missing part of the hessian}).
 \begin{equation*}
\begin{split}
Lie_E\left( \frac{\partial^2F}{\partial {t^1}^2}   \right)&=Lie_E\left( \frac{\tau}{2\pi i}   \right)\\
&=0=g^{\tau\tau}.
\end{split}
\end{equation*}
Hence, we proved the second equation (\ref{quasi homogeneous condition of Dubrovin Frobenius structure}). Substituting the equation (\ref{WDVV solution inside the proof}) in the second equation (\ref{quasi homogeneous condition of Dubrovin Frobenius structure}) for $\alpha,\beta\neq \tau$
 \begin{equation*}
\begin{split}
Lie_E\left( F^{\alpha\beta}   \right)&=Lie_E\left( \eta^{\alpha\alpha^{\prime}}\eta^{\beta\beta^{\prime}}\frac{\partial F^2}{\partial t^{\alpha^{\prime}} \partial t^{\beta^{\prime}} }  \right)\\
&=Lie_E\left( \eta^{\alpha\alpha^{\prime}}\eta^{\beta\beta^{\prime}}\frac{\partial G^2}{\partial t^{\alpha^{\prime}} \partial t^{\beta^{\prime}} }  \right) \\
&=g^{\alpha\beta} \in \widetilde E_{\bullet,\bullet}[t^1,t^2,..,t^n,\frac{1}{t^n}]
\end{split}
\end{equation*}

Hence, the second equation (\ref{quasi homogeneous condition of Dubrovin Frobenius structure} prove the third equation of (\ref{quasi homogeneous condition of Dubrovin Frobenius structure}).  \\

Substituting (\ref{WDVV solution inside the proof}) in (\ref{almost quasi homogeneous condition}), (\ref{definition of F almost})
 \begin{equation*}
\begin{split}
Lie_E\left( f^{\beta}  \right)&=Lie_E\left( \eta^{\beta\epsilon}\partial_{\epsilon}F  \right)\\
&=Lie_E\left( \eta^{\beta\epsilon}\partial_{\epsilon}F  \right)\partial_{\epsilon}F +\eta^{\beta\epsilon}Lie_E\left( \partial_{\epsilon}F  \right) \\
&=( d_{\beta}+d_{\epsilon}-d_1)\eta^{\beta\epsilon}\partial_{\epsilon}F \partial_{\epsilon}F+\eta^{\beta\epsilon}\partial_{\epsilon}Lie_E\left( F  \right)-d_{\epsilon}\eta^{\beta\epsilon}\partial_{\epsilon}F \\
&=( d_{\beta}-d_1)\eta^{\beta\epsilon}\partial_{\epsilon}F \partial_{\epsilon}F+\eta^{\beta\epsilon}\partial_{\epsilon}Lie_E\left( F  \right)\\
&=(d_{\beta}+d_1) \eta^{\beta\epsilon}\partial_{\epsilon}F+A_{\sigma}^{\beta}t^{\sigma}+B^{\beta}
\end{split}
\end{equation*}
Hence, isolating $Lie_E\left( F  \right)$
 \begin{equation*}
\begin{split}
\eta^{\beta\epsilon}\partial_{\epsilon}Lie_E\left( F  \right)=2\eta^{\beta\epsilon}\partial_{\epsilon}F+A_{\sigma}^{\beta}t^{\sigma}+B^{\beta},
\end{split}
\end{equation*}
inverting $\eta^{\beta\epsilon}$
 \begin{equation*}
\begin{split}
\partial_{\alpha}Lie_E\left( F  \right)=2\partial_{\alpha}F+\eta_{\alpha\beta}A_{\sigma}^{\beta}t^{\sigma}+\eta_{\alpha\beta}B^{\beta},
\end{split}
\end{equation*}
integrating
 \begin{equation*}
\begin{split}
Lie_E\left( F  \right)=2F+\eta_{\alpha\beta}A_{\sigma}^{\beta}t^{\alpha}t^{\sigma}+\eta_{\alpha\beta}B^{\beta}t^{\alpha},
\end{split}
\end{equation*}
Lemma proved.

\end{proof}

\begin{lemma}\label{definition of the Frobenius algebra}
Let  be
\begin{equation}
c_{\alpha\beta\gamma}=\frac{\partial F^3}{\partial t^{\alpha}\partial t^{\beta}\partial t^{\gamma}},
\end{equation}
then,
\begin{equation}\label{structure constant}
c_{\alpha\beta}^{\gamma}=\eta^{\gamma\epsilon}c_{\alpha\beta\epsilon}
\end{equation}
is a structure constant of a commutative algebra given by the following rule in the flat coordinate of $\eta$
\begin{equation}\label{Frobenius product}
\partial_{\alpha}\bullet\partial_{\beta}=c_{\alpha\beta}^{\gamma}\partial_{\gamma}
\end{equation}
such that
\begin{equation}\label{Frobenius condition}
\eta(\partial_{\alpha}\bullet\partial_{\beta},\partial_{\gamma})=\eta(\partial_{\alpha}, \partial_{\beta}\bullet\partial_{\gamma}), \quad \textbf{Frobenius condition}.
\end{equation}
\end{lemma}
\begin{proof}
\mbox{}\\*
\begin{enumerate}
\item \textbf{Commutative}\\
The product defined in (\ref{Frobenius product}) is commutative, because its structure constant (\ref{structure constant}) is symmetric with respect its indices $\alpha, \beta, \gamma$ due to the commutative behaviour of the partial derivatives $\frac{\partial}{\partial t^{\alpha}}, \frac{\partial}{\partial t^{\beta}}, \frac{\partial}{\partial t^{\gamma}}$.
\item \textbf{Frobenius condition}\\
\begin{equation*}
\begin{split}
\eta(\partial_{\alpha}\bullet\partial_{\beta},\partial_{\gamma})&=c_{\alpha\beta}^{\epsilon}\eta(\partial_{\epsilon},\partial_{\gamma})\\
&=c_{\alpha\beta}^{\epsilon}\eta_{\epsilon\gamma}\\
&=c_{\alpha\beta\gamma}\\
&=c_{\beta\gamma}^{\epsilon}\eta_{\alpha\epsilon}
=\eta(\partial_{\alpha}, \partial_{\beta}\bullet\partial_{\gamma}).
\end{split}
\end{equation*}
Lemma proved.
\end{enumerate}

\end{proof}

\begin{lemma}\label{main final lemma last chapter 3}
The unit vector field be defined in (\ref{unit vector field jtildean}) is the unit of the algebra defined in lemma \ref{definition of the Frobenius algebra}.
\end{lemma}
\begin{proof}
Substituting (\ref{definition of F almost}) and (\ref{definition of Fgamma})  in  (\ref{potentiality for the Levi Civita connection}), we obtain
\begin{equation}\label{relation Levi Civita connection and structure constants}
\Gamma_{\gamma}^{\alpha\beta}=d_{\beta}c_{\gamma}^{\alpha\beta},
\end{equation}
where
\begin{equation}
c_{\gamma}^{\alpha\beta}=\eta^{\alpha\mu}\eta^{\beta\epsilon}c_{\epsilon\mu\gamma}.
\end{equation}
Substituting $\alpha=\tau$ in (\ref{relation Levi Civita connection and structure constants}) and using lemma \ref{ Levi Civita connection tau}
\begin{equation*}
\begin{split}
\Gamma_{\gamma}^{\tau\beta}&=-2\pi i d_{\beta}\delta_{\gamma}^{\beta},\\
&=d_{\beta}c_{\gamma}^{\tau\beta}.
\end{split}
\end{equation*}
Then,
\begin{equation*}
\begin{split}
c^{\beta}_{0\gamma}=\delta_{\gamma}^{\beta}.\\
\end{split}
\end{equation*}
Computing
\begin{equation*}
\begin{split}
\partial_0\bullet\partial_{\gamma}=c^{\beta}_{0\gamma}\partial_{\beta}=\partial_{\gamma}.
\end{split}
\end{equation*}
Lemma proved.
\end{proof}

\begin{lemma}\label{main final lemma last chapter 4}
The algebra defined in lemma \ref{definition of the Frobenius algebra} is associative and semisimple.
\end{lemma}

\begin{proof}
Recall that the Christoffel symbol $\Gamma_{\gamma}^{\alpha\beta}$ is proportional to the structure constant of the algebra defined in lemma \ref{definition of the Frobenius algebra} for $\beta\neq v_{n+1},\tau$
\begin{equation*}
\Gamma_{\gamma}^{\alpha\beta}=d_{\beta}c_{\gamma}^{\alpha\beta}.
\end{equation*}
Then, using (\ref{eq3 Coxeter}), we obtain
\begin{equation}\label{associative for the ordinary Levi Civita connection}
\Gamma_{\sigma}^{\alpha\beta}\Gamma_{\epsilon}^{\sigma\gamma}=\Gamma_{\sigma}^{\alpha\gamma}\Gamma_{\epsilon}^{\sigma\beta}
\end{equation}
Substituting  (\ref{relation Levi Civita connection and structure constants})  in (\ref{associative for the ordinary Levi Civita connection}), we have
\begin{equation*}
c_{\sigma}^{\alpha\beta}c_{\epsilon}^{\sigma\gamma}=c_{\sigma}^{\alpha\gamma}c_{\epsilon}^{\sigma\beta}, \quad \text{for}\quad  \beta,\gamma\neq v_{n+1},\tau.
\end{equation*}
If $\beta=\tau$,
\begin{equation*}
\begin{split}
c_{\sigma}^{\alpha\tau}c_{\epsilon}^{\sigma\gamma}&=-2\pi i\delta^{\alpha}_{\sigma}c_{\epsilon}^{\sigma\gamma}\\
&=-2\pi ic_{\epsilon}^{\alpha\gamma}\\
&=-2\pi i\delta_{\epsilon}^{\sigma}c_{\sigma}^{\alpha\gamma}\\
&=c_{\sigma}^{\alpha\gamma}c_{\epsilon}^{\sigma\tau}.
\end{split}
\end{equation*}
In order to prove the associativity for $\beta=v_{n+1}$, note that the multiplication by the Euler vector field is almost the same of the intersection form $g^{*}$. Indeed,
\begin{equation}\label{relation multiplication Euler and intersection form An}
\begin{split}
E\bullet\partial_{\alpha}&=t^{\sigma}c_{\sigma\alpha}^{\beta}\partial_{\beta}=t^{\sigma}\partial_{\sigma}\left(\eta^{\beta\mu}\partial_{\alpha}\partial_{\mu}F  \right)  \partial_{\beta}=\\
&=(d_{\alpha}-d_{\beta})\eta^{\beta\mu}\partial_{\alpha}\partial_{\mu}F  \partial_{\beta}=\eta_{\alpha\mu}g^{\mu\beta} \partial_{\beta}
\end{split}
\end{equation}
Then,
\begin{equation}\label{relation multiplication Euler and intersection form An}
\begin{split}
E^{\sigma}c_{\sigma}^{\alpha\beta}=g^{\alpha\beta}.
\end{split}
\end{equation}
Using the relation (\ref{relation multiplication Euler and intersection form An}) in the coordinates  $(u^1,u^2,...,u^{n+2})$, we have
 \begin{equation}\label{final equation associativite An}
u^i\eta^{ij}\delta_{ij}=u^l\eta^{im}\eta^{jn}c_{lmn},
\end{equation}
differentiating both side of the equation (\ref{final equation associativite An}) with respect $t^1$
 \begin{equation}
 c_{ij}^k=\delta_{ij},
\end{equation}
which proves that the algebra is associative and semisimple.

\end{proof}

Recall of the covering space of the orbit space of $\Ja(\tilde A_n)$ defined in (\ref{covering space for the tilde an case}), see section \ref{extended ring jtildean} for details.

\begin{theorem}\label{Dubrovin Frobenius structure of the orbit space}
The  covering  $\widetilde{\mathbb{C}\oplus\mathbb{C}^n\oplus\mathbb{H}}/\Ja(\tilde A_n)$ with the intersection form (\ref{metrich1n}), unit vector field (\ref{unit vector field jtildean}), and Euler vector field (\ref{Euler vector field jtildean}) has a Dubrovin Frobenius manifold structure.
\end{theorem}

\begin{proof}
The function (\ref{WDVV solution}) satisfy a WDVV equation due to the lemmas \ref{main final lemma last chapter 1}, \ref{definition of the Frobenius algebra}, \ref{main final lemma last chapter 3}, \ref{main final lemma last chapter 4}.

\end{proof}

Taking the same covering taking in the orbit space of $\Ja(\tilde A_n)$ in the Hurwitz space $H_{1,n-1,0}$, fixing a symplectic base of cycle, a chamber in the tori where the variable $v_{n+1}$ lives, and branching root of $\varphi_n$, denoting this covering by 
\begin{equation}
\widetilde H_{1,n-1,0},
\end{equation}
we obtain

\begin{theorem}\label{Dubrovin Frobenius structure of the orbit space Mirror symmetry}
The Dubrovin Frobenius structure of the covering space $\widetilde{\mathbb{C}\oplus\mathbb{C}^n\oplus\mathbb{H}}/\Ja(\tilde A_n)$  is isomorphic as Dubrovin Frobenius manifold to the covering  $\widetilde H_{1,n-1,0}$.
\end{theorem}
\begin{proof}
Both the orbit space $\Ja(\tilde A_n)$ and the Hurwitz space $H_{1,n-1,0}$ has the same intersection form, Euler vector, unit vector field due to proposition \ref{same Euler, unit, and intersection form}, lemma \ref{superpotential vs discriminant locus}  and \ref{distinct roots} From this data, one can reconstruct the WDVV solution by using the relation
\begin{equation}
F^{\alpha\beta}=\eta^{\alpha\alpha^{\prime}}\eta^{\beta\beta^{\prime}}\frac{\partial^2F}{\partial t^{\alpha^{\prime}} \partial t^{\beta^{\prime}}}=\frac{g^{\alpha\beta}}{\text{deg}g^{\alpha\beta}}.
\end{equation}
Theorem proved.
\end{proof}

E-mail address: galmeida@sissa.it.

\end{document}